\documentclass[11pt,reqno]{amsart}

\usepackage[utf8]{inputenc}
\usepackage[margin=1.18in]{geometry}
\usepackage{amsmath,amssymb,amsfonts,mathtools,mathrsfs}
\usepackage{enumitem}
\usepackage{comment}
\usepackage{microtype}
\usepackage{booktabs}
\usepackage[dvipsnames]{xcolor}
\usepackage[colorlinks=true,citecolor=MidnightBlue,linkcolor=BrickRed,
  urlcolor=RoyalBlue]{hyperref}
\usepackage[backend=bibtex,style=alphabetic,maxnames=99]{biblatex}
\numberwithin{equation}{section}

\newtheorem{theorem}{Theorem}[section]
\newtheorem{proposition}[theorem]{Proposition}
\newtheorem{lemma}[theorem]{Lemma}
\newtheorem{corollary}[theorem]{Corollary}
\theoremstyle{definition}
\newtheorem{definition}[theorem]{Definition}
\newtheorem{assumption}[theorem]{Assumption}
\theoremstyle{remark}
\newtheorem{remark}[theorem]{Remark}

\newcommand{\C}{\mathbb C}
\newcommand{\R}{\mathbb R}
\newcommand{\T}{\mathbb T}
\newcommand{\E}{\mathbb E}
\newcommand{\Pp}{\mathbb P}
\newcommand{\1}{\mathbf 1}
\newcommand{\HS}{\mathrm{HS}}
\newcommand{\op}{\mathrm{op}}
\newcommand{\cir}{\mathrm{circ}}

\newcommand{\Prob}{\mathbb P}
\renewcommand{\P}{\mathbb P}
\newcommand{\Coeff}{\operatorname{Coeff}}
\newcommand{\dist}{\operatorname{dist}}
\newcommand{\Gr}{\operatorname{Gr}}
\newcommand{\clip}{\operatorname{clip}}
\newcommand{\GL}{\operatorname{GL}}
\newcommand{\Emb}{\operatorname{Emb}}
\DeclareMathOperator{\rank}{rank}
\DeclareMathOperator{\tr}{tr}
\DeclareMathOperator{\TV}{TV}
\newcommand{\introheading}[1]{%
  \par\medskip\noindent\textbf{#1.}\par\smallskip}

\title[Non-Hermitian random band matrices]
{The circular law for non-Hermitian random band matrices: optimal bandwidth, periodic profile and discrete law}

\author{Yi Han}
\thanks{Department of Mathematics, Massachusetts Institute of Technology,
Cambridge, MA.\newline Email: \texttt{hanyi16@mit.edu}.}
\thanks{A Lean 4  formalization of the whole paper is available online,
see Section~\ref{subsec:lean-verification} for more details.}
\hypersetup{
  pdftitle={The circular law for non-Hermitian random band matrices: optimal bandwidth, periodic profile and discrete law},
  pdfauthor={Yi Han}
}

\begin{document}

\begin{abstract}
We consider non-Hermitian random band matrices with growing bandwidth and
study convergence of their empirical spectral distributions to the circular
law. Let \(N\) denote the matrix size and \(W_N\) the bandwidth. From a
universality perspective, it is conjectured that the circular law holds
whenever \(W_N\to\infty\). Previous results have mainly required
\(W_N\gg N^{1/2}\), with the principal exception of \cite{Han2511}, which
proves the \(W_N\to\infty\) circular law for an open-boundary
block-tridiagonal model. Here we prove the circular law at this optimal
threshold for several periodic models and under the near-optimal condition
\(W_N\gg\log N\) for a genuinely discrete model. For the periodic
hard-indicator profile and its uniform and polynomially tapered
generalizations, we prove the circular law under bounded-density and
finite-third-moment assumptions whenever \(W_N\to\infty\). At the same
threshold, we prove the circular law for
continuous, integrable profiles locally bounded below on every finite
interval, including exponentially and Gaussian decaying profiles, with
circular complex Gaussian entries. For the periodic full-block model, we
prove the circular law under bounded-density and finite-third-moment
assumptions whenever \(W_N\to\infty\). The finite-third-moment assumption in the
bounded-density results can be weakened to a finite
\((2+\alpha)\)-moment assumption after a corresponding modification of the
proof. For the same full-block model without a density assumption, we prove
the circular law for centered variance-one real subgaussian atoms when
\(W_N\gg\log N\). The
proof uses compositions of random transfer operators. An established
high-band circular-law input on a short auxiliary ring calibrates the full
exterior coefficient norm, and a boundary-uniform local comparison lifts
this calibration to target rings even when \(N/W_N\) is arbitrarily large.
\end{abstract}

\maketitle
\tableofcontents

\section{Introduction}
\label{sec:introduction}

Let \(A\) be an \(N\times N\) matrix with eigenvalues
\(\lambda_1,\ldots,\lambda_N\), counted with algebraic multiplicity.  Its
empirical spectral distribution (ESD) is denoted by
\[
 \mu_A:=\frac1N\sum_{j=1}^N\delta_{\lambda_j}.
\]
For a matrix with iid centered entries of variance \(N^{-1}\), the
convergence of \(\mu_A\) to the uniform probability measure
\(\mu_{\cir}\) on the unit disk is the classical circular law.  The
complex Gaussian ensemble goes
back to Ginibre \cite{Ginibre1965}, while Girko's non-rigorous work
\cite{Girko1985} initiated the Hermitization approach.  Rigorous proofs are gradually developed in \cite{Bai1997,GotzeTikhomirov2010,PanZhou2010}.
Tao and Vu proved the circular law under the optimal second-moment
assumption \cite{TaoVuKrishnapur2010}; see
\cite{BordenaveChafai2012} for a survey and further references.  Its
Hermitian counterpart is Wigner's semicircle law, originating in
\cite{Wigner1955CharacteristicVectors} and extended to substantially
more general real-symmetric ensembles in
\cite{Wigner1958RootsSymmetricMatrices}.  At the global scale, the
semicircle law convergence is accessible by the moment method since the eigenvalues are all real.  In the
non-Hermitian setting, however, eigenvalues are unstable under
perturbations and moments do not determine the spectral measure in the
same way. This is the reason why Hermitization arguments are introduced \cite{Girko1985}, and thus the proof of non-Hermitian spectral convergence is usually much more complicated.

Beyond dense iid ensembles, the circular law is known for
Bernoulli-thinned $p_N$-sparse iid matrices, first when
\(Np_N\gg(\log N)^2\) and subsequently under the optimal condition
\(Np_N\to\infty\)
\cite{BasakRudelson2019,RudelsonTikhomirov2019}.  At fixed average
degree (constant $p_N$),
\cite{SahSahasrabudheSawhney2025} proved that the ensemble still has a deterministic limiting
spectral law, but that limit is not the circular law.  For matrices with
independent entries and a nonconstant variance profile, deterministic
equivalents and limiting empirical spectral distributions are available
under quantitative irreducibility, while a local inhomogeneous circular
law is known when the variances are uniformly comparable
\cite{CookHachemNajimRenfrew2018,CookHachemNajimRenfrew2022,
AltErdosKruger2018}.  These settings do not directly cover a
deterministic one-dimensional band mask with \(W/N\to0\).

Random band matrices are the basic one-dimensional inhomogeneous model, and in this paper we consider non-Hermitian band matrices. Let $X_N$ be a square random matrix with independent entries where only entries within an effective distance \(W=W_N\) from the diagonal are
substantial, so \(X_N\) interpolates between a finite-range random operator
and a mean-field matrix.  At fixed bandwidth $W$, the limiting spectral measure, if it exists\footnote{The limiting ESD for these non-Hermitian $X_N$ with constant $W$ has only been proven to exist in the special case in \cite{GoldsheidKhoruzhenko2005}, but we expect that the proof strategy can be modified to cover more general model profiles.},
should retain detailed information about the entry law and the local
geometry of the band structure of $X_N$.  Once \(W_N\to\infty\), one expects this local dependence to be
washed out at the level of the empirical spectral distribution.  Under the doubly stochastic normalizing assumption $\mathbb{E}[XX^*]=\mathbb{E}[X^*X]=I_N$, this
leads to the global-universality prediction
\[
 W_N\longrightarrow\infty
 \qquad\Longrightarrow\qquad
 \mu_{X_N}\Longrightarrow\mu_{\cir},
\]
under suitable nondegeneracy assumptions. In contrast to the
expected Anderson type localization-delocalization transition of eigenvectors near \(W\asymp N^{1/2}\), this prediction on $\mu_{X_N}$ concerns
only the global eigenvalue distribution and is expected to remain valid
on both sides of that scale $N^{1/2}$. As a comparison, for Hermitian random band matrices which is much better studied by now, complete
delocalization, quantum unique ergodicity, and bulk universality when $W\gg N^{1/2}$ are established in \cite{YauYin2025,ErdosRiabov2025}, while
exponential localization on the sharp \(W^2\) scale when \(W^2\ll N\) is
proved in \cite{Drogin2025}. To our best knowledge, no such localization or delocalization results have been rigorously proven for non-Hermitian random band matrices by now, but see the recent paper \cite{ShcherbinaShcherbina2025} for a transition on the characteristic functions.

\introheading{The two basic periodic geometries and one generalization}In this paper we consider several different band matrix profiles, but they essentially belong to one of the two geometries below. All indices below are cyclic.  

The first model is the scalar model, which has the compact form
\begin{equation}\label{eq:intro-scalar-model}
 (X_N^{\rm sc})_{i,i+s}=\sqrt{q_s}\,\xi_{i,s},\qquad
 |s|\le W,\qquad
 \sum_{|s|\le W}q_s=1,\qquad q_s\asymp W^{-1},
\end{equation}
with all remaining entries zero, where the \(\xi_{i,s}\) are independent,
centered, and variance one.  This includes the hard indicator band
\(q_s=(2W+1)^{-1}\) and its uniformly comparable variants.

The second one is the periodic full-block model with \(N=mW\).  Its \((j,k)\)-th
\(W\times W\) block is
\begin{equation}\label{eq:intro-block-model}
 (X_N^{\rm blk})[j,k]=\frac1{\sqrt{3W}}
 \begin{cases}
  A_j,&k=j,\\
  B_j,&k=j+1,\\
  C_j,&k=j-1,\\
  0,&\text{otherwise},
 \end{cases}
 \qquad j,k\in\mathbb Z/m\mathbb Z,
\end{equation}
where all entries of the independent blocks \(A_j,B_j,C_j\) are copies
of one centered variance-one atom.  Thus \eqref{eq:intro-scalar-model}
is an entrywise scalar band, whereas \eqref{eq:intro-block-model} has
three complete random \(W\times W\) blocks in every block row.

We also consider an extension of the scalar model without a hard
band cutoff. The entries are independent centered circular complex
Gaussians, and the variance at displacement \(s\) (i.e., at a location with distance $s$ to the diagonal) is
\[
 q_s=\frac{f(s/W)}
 {\displaystyle\sum_{t\in\mathbb Z\cap[-N/2,N/2)}f(t/W)},
 \qquad s\in\mathbb Z\cap[-N/2,N/2),
\]
where \(f:\mathbb R\to(0,\infty)\) is an integrable variance profile,
for example we can take $f(x)=\exp(-x)$.
Thus entries at all cyclic displacements have positive variance,
and \(W\) determines the scale of the profile rather than a strict
band cutoff.

Section~\ref{sec:model} gives the
precise density, moment, profile, and bandwidth assumptions.

\introheading{The non-Hermitian obstruction}
Girko's Hermitization reduces the circular law to convergence of
logarithmic determinants.  If
\(s_1(X_N-zI)\ge\cdots\ge s_N(X_N-zI)\) are the singular values, then
\begin{equation}\label{logpotentialsums}
 \frac1N\log|\det(X_N-zI)|
 =\frac1N\sum_{j=1}^N\log s_j(X_N-zI).
\end{equation}Under the normalizing assumption $\mathbb{E}[XX^*]=\mathbb{E}[X^*X]=I_N$, and assuming suitable moments and for sufficiently large $W_N$, it is not hard to show that the singular value measure of $X_N-zI$ converges to a fixed limiting measure.
However, weak convergence of the singular-value measure controls every bounded
continuous test function, but not the logarithm at the hard edge.  One
must show that the least singular value and the remaining small singular
values make a negligible contribution.  Controlling this hard-edge
contribution is already the decisive step in the dense circular-law
proofs \cite{Bai1997,GotzeTikhomirov2010,PanZhou2010,
TaoVuKrishnapur2010}.  For dense non-Hermitian matrices,
sharp estimates on the least singular value are available for iid square and rectangular matrices
\cite{RV,RVrect}, and for broad inhomogeneous square ensembles
\cite{LivshytsTikhomirovVershynin2021}.  This polynomial behavior is the
benchmark for the band-matrix estimates below.  The problem becomes
substantially harder for a genuinely sublinear band profile, where each row sees only \(O(W_N)=o(N)\)
independent variables.
Least-singular-value and small-singular-value estimates are therefore the
main analytic obstruction to a band matrix circular law.  A deterministic band
structure, in contrast to a sparse iid matrix, requires band-dependent structural information that is absent from the dense
invertibility theory quoted above.

\introheading{Previous bandwidth regimes}
The first circular-law results for genuinely sublinear block bandwidth
were obtained in \cite{JainJanaLuhORourke2021} for full block matrices when $W\geq N^{32/33}\log N$.  For cyclic indicator band matrices with Gaussian entries, \cite[Corollary~1.2]{Tikhomirov2023} proved the
circular law when \(W\ge N^{33/34}\).  Subsequent high-band results
treated more general doubly stochastic profiles and improved the
admissible powers of \(N\); see \cite{Han2410}.  The preprint
\cite{Han2508} pushed several singular-value estimates in this program
toward much smaller bandwidths \(W_N\), though still in regimes above
\(\sqrt N\).  The band-profile least-singular-value inputs used in all these papers above give only an exponentially small estimate, typically of the form
\begin{equation}\label{singuupdates}
 s_{\min}(X_N-zI_N)\ge
 \exp\!\left\{-\frac{N}{W}N^{o(1)}\right\}
 \quad\text{or}\quad W^{-CN/W},
\end{equation}
rather than the polynomial hard-edge scale available in dense models
\cite{JainJanaLuhORourke2021,Tikhomirov2023,Han2508}.  This reflects a
limitation common to all currently available methods for non-Hermitian
band matrices.  For balanced periodic bands, one expects a much
stronger, plausibly polynomial, lower bound, but no such estimate is
known.  Thus all band circular-law proofs cited above use the exponential
lower bound that is currently available.

In the polynomial high-band regimes treated by the previous works, the strategy is to take a cutoff in the singular value summation \eqref{logpotentialsums} and treat the two components: the small singular values and the remaining bulk singular values, differently. We can for example choose the cutoff at, for some $y\in(0,1)$,
\[
 a_N=\frac{N^{o(1)}}{W^y}\longrightarrow0.
\]
The Hermitized bulk estimate already controls the logarithmic potential
coming from singular values above \(a_N\).  Writing
\(s_j=s_j(X_N-zI_N)\), the remaining task is precisely to prove
\[
 \frac1N\sum_{j:\,s_j\le a_N}|\log s_j|
 \xrightarrow{\mathbb P}0.
\]
Depending on the exact geometry of the band and the entry distribution of $X_N$, there exists a value $y\in(0,1)$ where we have proven (\cite{Han2410,Han2508}) mesoscopic singular value counting estimate which leaves
\(O(Na_N)=N^{1+o(1)}/W^y\) singular values in \([0,a_N]\). Meanwhile, the
exponential least singular value estimate in \eqref{singuupdates} only gives
\(|\log s_j|\le (N/W)N^{o(1)}\).  Their worst-case contribution is
therefore \(N^{1+o(1)}/W^{1+y}\). We call this a simultaneous-control strategy as we simultaneously use the worst case estimate for both the least singular value and for the mesoscopic singular value counting estimate. This explains why simultaneous-control arguments
weaken as \(W\) decreases and why \(\sqrt N\) appears as a natural scale
for that method\footnote{Determining the largest value $y$ where this mesoscopic counting still holds is an open problem, but the recent band matrix papers \cite{YauYin2025,ErdosRiabov2025} suggest we may hope for any $y<2$, which gives rise to $W_N\gg N^{1/3+\epsilon}$.}. But as earlier explained, $\sqrt{N}$ is clearly not the transition point for the
global circular law.

To our knowledge, \cite{Han2511} was the first circular-law result for a
genuinely banded non-Hermitian ensemble to reach the optimal qualitative
condition \(W_N\to\infty\), for an open-boundary block-tridiagonal model.
Its transfer-based mechanism was inspired in part by the
fixed-bandwidth non-Hermitian Jacobi model of
\cite{GoldsheidKhoruzhenko2005}.  In the Jacobi setting, a
\(2\times2\) transfer product relates
the limiting eigenvalue measure to a Lyapunov exponent through a Thouless
formula.  The argument of \cite{Han2511} adapts this viewpoint to growing
block size, and hence to transfer dimension growing with \(W\): the
determinant is computed recursively along the chain, and a long system is
decomposed into shorter systems whose log determinants can be identified
in a high-band regime.  The open endpoints supply fixed starting and
ending boundary conditions; neither exists in the cyclic models
\eqref{eq:intro-scalar-model}--\eqref{eq:intro-block-model}.  Thus neither
model is a direct specialization of the open-chain theorem.

\introheading{Main results}
This paper establishes the growing-bandwidth circular law for three
periodic model classes, with a fourth theorem accounting for discrete
atoms in the full-block model.
\begin{enumerate}[leftmargin=2.1em]
 \item For the cyclic scalar indicator band, and for uniformly comparable
       weights on the same band, Theorem~\ref{thm:indicator} proves the
       circular law under \(W_N\to\infty\), bounded density, and a finite
       third moment.  Corollary~\ref{cor:tapered-indicator} also allows
       polynomial decay at the two endpoints of the band.
 \item For a translation-invariant profile without compact support,
       Theorem~\ref{thm:gaussian-profile} proves the same conclusion under
       \(W_N\to\infty\) for circular complex Gaussian entries, provided
       that the profile is locally bounded below on every finite interval.
       This includes exponential and Gaussian decay.
 \item For the periodic full-block band model, Theorem~\ref{thm:density-main}
       proves the circular law under \(W_N\to\infty\) for normalized
       bounded-density entry laws with a finite third moment. 
 \item For the same periodic full-block geometry with a real subgaussian
       atom, which may be discrete, Theorem~\ref{thm:main} removes the
       density assumption under \(W_N/\log N\to\infty\).  In particular,
       the theorem includes Rademacher entries.
\end{enumerate}
The finite third moment in the bounded-density statements is a convenient
common hypothesis for the mesoscopic input.  As explained after that
input, it can be replaced by a finite \((2+\alpha)\)-moment at the cost of
changing the auxiliary scales; see Remark~\ref{rem:third-moment}.

\introheading{Idea of the proof}
The main new difficulty relative to the open-chain argument of
\cite{Han2511} comes from periodicity.  For an open chain, one can build
the determinant recursively from a distinguished first block to a last
one.  Once the two ends are joined, there is no preferred starting point,
and the terms created when the recursion closes may cancel.  We first use
an independent short ring, in a regime where the high-band circular law is
already available, to determine the typical logarithmic contribution per
unit length.  We then prove that this contribution can be passed through
successive portions of the original ring, with an error that is negligible
after normalization and with a bound that is independent of where the ring
is cut.  Iterating this local comparison around the ring gives the desired
logarithmic determinant.  The bounded-density and discrete full-block
cases use different local estimates to control cancellation, but the
passage from the short ring to the full periodic matrix is the same.  This
separation between a model-dependent local estimate and a common global
argument is the main methodological contribution.  The noncompact
Gaussian theorem is then obtained from the compact scalar result by a
variance cutoff and rotationally invariant tail removal.

\introheading{Related spectral questions}
Several recent works study related spectral questions for non-Hermitian
band and inhomogeneous matrices.  Absence of outliers, stability under
finite-rank perturbations, and high-probability invertibility for related
inhomogeneous and band models are studied in
\cite{Han2408Outliers,Han2507Invertibility}.  These results address
complementary questions but do not by themselves provide the logarithmic
integrability required for a circular law.  For the Gaussian Green-kernel
band profile, a transfer-operator analysis of characteristic
polynomials near \(W\asymp\sqrt N\) in the Hermitian setting appears in
\cite{Shcherbina2020}; the corresponding non-Hermitian off-critical and
critical regimes are treated in
\cite{ShcherbinaShcherbina2025,ShcherbinaShcherbina2604}.  Beyond band
matrices, least-singular-value estimates also drive Brown-measure
convergence for noncommutative polynomials in Ginibre matrices
\cite{Han2606Brown}.

\introheading{Organization}
Section~\ref{sec:model} defines the models and states the main results.
Section~\ref{sec:high-band-input} collects the common
least-singular-value, mesoscopic, and short-ring inputs.  The first proof
part proves the scalar indicator theorem and then extends it to
noncompact Gaussian profiles.  The second proof part proves the two
cyclic full-block theorems, first for discrete subgaussian atoms and then
for bounded-density atoms.  Appendix~A records
the minor repair of the full-block least-singular-value input from
\cite{JainJanaLuhORourke2021}, and Appendix~B proves the density-based
least-singular-value theorem used by the bounded-density branches.

\section{Models and main results}
\label{sec:model}

Write $\T_N=\mathbb Z/N\mathbb Z$ and choose the centered representatives
\[
 \mathcal D_N:=\{-\lfloor N/2\rfloor,\ldots,\lceil N/2\rceil-1\}.
\]
Thus every site in $\T_N$ is represented by exactly one element of
$\mathcal D_N$.  Set
\[
 |i-j|_N:=\min_{m\in\mathbb Z}|i-j+mN|,
\]
and read all matrix indices modulo $N$.  For $A\in\C^{N\times N}$, let
\[
 \mu_A:=\frac1N\sum_{j=1}^N\delta_{\lambda_j(A)},
\]
where the eigenvalues are counted with algebraic multiplicity.  Let
\(\mu_{\cir}\) denote normalized area measure on the unit disk and set
\begin{equation}\label{eq:circular-potential}
 U_{\cir}(z):=\int_{|w|\le1}\log|w-z|\,\frac{d^2w}{\pi}
 =\begin{cases}
 \frac12(|z|^2-1),&|z|\le1,\\[2mm]
 \log|z|,&|z|>1.
 \end{cases}
\end{equation}

\subsection{Scalar indicator bands}

Let $W=W_N\in\mathbb N$ satisfy $W\to\infty$ and $2W+1\le N$.  Denote by
\(\mathcal S_W=[-W,W]\cap\mathbb Z\).  The deterministic variance weights
may depend on $N$; we write $q_s=q_{N,s}$ and suppress the first index.
They are translation invariant and satisfy, for constants
$0<c_0\le C_0<\infty$ independent of $N$,
\begin{equation}\label{eq:indicator-weights}
 q_s=0\quad(s\notin\mathcal S_W),\qquad
 \sum_{s=-W}^Wq_s=1,\qquad
 \frac{c_0}{2W+1}\le q_s\le\frac{C_0}{2W+1}.
\end{equation}

We write $b_s=\sqrt{q_s}$ and then define the cyclic scalar band matrix via
\begin{equation}\label{eq:indicator-model}
 (X_N)_{i,i+s}=b_s\xi_{i,s},\qquad i\in\T_N,\quad -W\le s\le W.
\end{equation}
All other entries are zero. For example, we recover the periodic indicator hard cutoff profile considered in \cite{Tikhomirov2023} if we take $c_0=C_0=1$ and thus $q_s$ is constant on the band.

The variables $\xi_{i,s}$ are independent
copies of an atom $\xi$ satisfying the following standing assumption.

\begin{assumption}[Indicator atom and bounded density]
\label{ass:indicator-density}
The atom is centered, normalized, and has a finite third moment:
\begin{equation}\label{eq:atom-normalization}
 \E\xi=0,\qquad \E|\xi|^2=1.
\end{equation}
\begin{equation}\label{eq:atom-moment}
 \E|\xi|^3<\infty.
\end{equation}
In addition, there is a constant $L<\infty$ such that one of the following alternatives holds:
\begin{enumerate}[label=\textnormal{(\roman*)},leftmargin=2.2em]
 \item $\xi$ is real and its law has a Lebesgue density bounded by $L$;
 \item $\xi$ is complex and its law has a density, with respect to planar
       Lebesgue measure on $\C\simeq\R^2$, bounded by $L$.
 \item $\xi$ is complex and there is a fixed $\theta\in\R$ such that,
       writing
       \[
        U:=\Re(e^{-\mathrm i\theta}\xi),\qquad
        V:=\Im(e^{-\mathrm i\theta}\xi),
       \]
       the conditional law of $U$ given $V=v$ has a Lebesgue density
       bounded by $L$, uniformly for $v$ outside a set of $V$-measure zero.
       No independence between $U$ and $V$ is required.  This includes the
       case in which $(U,V)$ are independent and only $U$ has a bounded
       density.
\end{enumerate}
\end{assumption}

Constants in the indicator-profile part may depend on $c_0,C_0,L$ and on
the fixed spectral parameter $z$, but not on $N$ or $W$.

\begin{theorem}[Circular law for indicator profiles]\label{thm:indicator}
Let $X_N$ be the cyclic matrices defined by
\eqref{eq:indicator-weights}--\eqref{eq:indicator-model}, with bandwidths
$W_N\to\infty$ satisfying $2W_N+1\le N$ and with atom satisfying
Assumption~\ref{ass:indicator-density}.  Then
\[
 \mu_{X_N}\Longrightarrow\mu_{\cir}
 \qquad\text{in probability as }N\to\infty.
\]
 In fact, we prove the Hermitized statement that, for every fixed
$z\in\C$,
\begin{equation}\label{eq:indicator-logdet-goal}
 \frac1N\log|\det(X_N-zI_N)|\xrightarrow{\Pp}U_{\cir}(z).
\end{equation}
\end{theorem}

\medskip\noindent
The finite-third-moment hypothesis may be replaced by
\(\E|\xi|^{2+\alpha}<\infty\) for any fixed \(0<\alpha\le1\), after
adjusting the mesoscopic scales; see Remark~\ref{rem:third-moment}.

\begin{remark}\label{rem:uniform-generalization}
Nothing in the proof uses equality of the active weights $q_s$.  The result is
therefore an indicator-profile theorem together with its natural uniform
generalization \eqref{eq:indicator-weights}. 

Indeed, the same proof strategy extends to more general translation-invariant
profiles on $\mathcal S_W$ that are positive in the interior and decay no
faster than a fixed power near the two endpoints.  For example, taking
$f(x)=1-|x|$ for $|x|\le1$ in the following corollary gives a triangular
variance profile supported on the band.
\end{remark}

\begin{corollary}[Indicator profiles with polynomially controlled endpoints]
\label{cor:tapered-indicator}
Let $W=W_N\in\mathbb N$ satisfy $W\to\infty$ and $2W+1\le N$.
Let $f:\R\to[0,\infty)$ have finite total variation, vanish for
$|x|\ge1$, and satisfy, for fixed $c,C>0$ and $\kappa\ge0$,
\begin{equation}\label{eq:tapered-profile}
 c(1-|x|)^\kappa\le f(x)\le C,
 \qquad |x|<1.
\end{equation}
Define
\begin{equation}\label{eq:tapered-weights}
 Z_W^{\rm tap}:=\sum_{|s|\le W}f\left(\frac{s}{W+1}\right),
 \qquad q_s^{\rm tap}:=\frac{f(s/(W+1))}{Z_W^{\rm tap}},
 \qquad |s|\le W.
\end{equation}
Let $(\xi_{i,s})_{i\in\T_N,\,|s|\le W}$ be independent copies of an atom
satisfying Assumption~\ref{ass:indicator-density}, and define the cyclic
matrix $X_N^{\rm tap}$ by
\begin{equation}\label{eq:tapered-model}
 (X_N^{\rm tap})_{i,i+s}:=\sqrt{q_s^{\rm tap}}\,\xi_{i,s},
 \qquad i\in\T_N,\quad |s|\le W,
\end{equation}
with all other entries equal to zero.  Then all conclusions of
Theorem~\ref{thm:indicator} hold with $X_N^{\rm tap}$ in place of $X_N$.
\end{corollary}

\subsection{Translation-invariant profiles without compact support}
We next extend the preceding strategy to noncompact translation-invariant
profiles, including exponentially decaying examples.  The proof truncates the band, applies the
indicator Theorem \ref{thm:indicator} to the compact core of size $\approx W_N$, and then removes the outside profile using
Gaussian rotational invariance.

The precise assumption on the profile function $f$ is as follows.
Let $f:\R\to[0,\infty)$ be continuous, integrable and of finite total
variation.  We assume explicitly that it has no zeros on any finite
interval and normalize it by
\begin{equation}\label{eq:f-assumptions}
 \int_{\R}f(x)\,dx=1,\qquad \TV(f)<\infty,\qquad
 m_f(R):=\inf_{|x|\le R}f(x)>0\quad\text{for every }0<R<\infty.
\end{equation}
For an arbitrary positive sequence $W=W_N\to\infty$, define the variance
profile
\begin{equation}\label{eq:gaussian-profile}
 \begin{aligned}
 Z_{N,W}&:=\sum_{s\in\mathcal D_N}f(s/W),
 &q_s^{(N,W)}&:=\frac{f(s/W)}{Z_{N,W}},\\
 (G_{N,f})_{i,i+s}&:=\sqrt{q_s^{(N,W)}}\,g_{i,s},
 &i&\in\T_N,\quad s\in\mathcal D_N,
 \end{aligned}
\end{equation}
where the $g_{i,s}$ are independent standard circular complex Gaussians,
normalized by $\E g_{i,s}=0$ and $\E|g_{i,s}|^2=1$ (equivalently, with
density $\pi^{-1}e^{-|w|^2}\,d^2w$).
The variance matrix is doubly stochastic by translation invariance.

\begin{theorem}[Gaussian translation-invariant profiles]
\label{thm:gaussian-profile}
For every profile $f$ satisfying \eqref{eq:f-assumptions} and every
sequence $W_N\to\infty$, let $G_{N,f}$ be the matrix defined by
\eqref{eq:gaussian-profile}.  Then
\[
 \mu_{G_{N,f}}\Longrightarrow\mu_{\cir}
 \qquad\text{in probability.}
\]
Moreover, for every fixed \(z\in\C\),
\[
 \frac1N\log|\det(G_{N,f}-zI_N)|
 \xrightarrow{\Pp}U_{\cir}(z).
\]
\end{theorem}

\begin{remark}[On the circular Gaussian assumption]
In Theorem~\ref{thm:gaussian-profile}, Circular Gaussian law (as opposed to merely assuming the bounded density assumption) is used in the
tail-removal step \eqref{eq:tail-Jensen-lower}. But the actual property being used there is only the rotational invariance,
\(g_{i,s}\stackrel{\mathrm d}=e^{\mathrm i\theta}g_{i,s}\).
Accordingly, the proof extends to rotationally invariant complex
laws satisfying alternative~(ii) of Assumption~\ref{ass:indicator-density}.
Rotational invariance is used only for this unbounded-support profile model.
\end{remark}

\begin{remark}[Why local positivity is assumed]
\label{rem:gaussian-local-positivity}
The condition $m_f(R)>0$, although somewhat technical, is still quite general
as it imposes no lower bound on the decay as $|x|\to\infty$.  For example, it
covers functions with exponential or Gaussian decay:
\[
 f_{\exp}(x)=\frac12e^{-|x|},\quad
 f_{\text{gauss}}(x)=\frac{1}{\sqrt{2\pi}}e^{-|x|^2/2}.
\]
The former is the translation-invariant infinite-volume
Green-kernel profile corresponding to the covariance
$(I+W^2(-\Delta))^{-1}$ studied by Shcherbina and Shcherbina
\cite{ShcherbinaShcherbina2025}.  Thus the present theorem includes the
periodic exponential-profile analogue of their model.  The assumption
$m_f(R)>0$ is used in
\eqref{eq:normalized-core-comparability} and
\eqref{eq:dense-profile-comparability}.  It may be removable by a more
refined argument, and we do not expect it to be intrinsic.
\end{remark}

\subsection{The cyclic full-block model}

This subsection introduces a second matrix model.  To keep the formulas
light, we again write \(X_N\); here, in the full-block statements of
Section~\ref{sec:high-band-input}, and throughout the full-block part,
\(X_N\)
denotes the periodic full-block band matrix \eqref{eq:model}, not the
scalar indicator matrix \eqref{eq:indicator-model}.

Let \(N=mW\), where \(m\ge 4\), and read block indices modulo \(m\).
For \(1\le j\le m\), let \(A_j,B_j,C_j\) be mutually independent
\(W\times W\) random matrices.  All their scalar entries are independent
copies of one fixed real or complex random variable \(\xi\), normalized by
\begin{equation}
 \mathbb E\xi=0,\qquad \mathbb E|\xi|^2=1.
 \label{eq:atom}
\end{equation}
The additional assumptions on \(\xi\) are stated separately in the two
theorems below.

Define the variance-normalized blocks
\[
 \mathsf A_j=(3W)^{-1/2}A_j,\qquad
 \mathsf B_j=(3W)^{-1/2}B_j,\qquad
 \mathsf C_j=(3W)^{-1/2}C_j
\]
and the cyclic block-tridiagonal matrix
\begin{equation}
 X_N=
 \begin{pmatrix}
  \mathsf A_1&\mathsf B_1&&&\mathsf C_1\\
  \mathsf C_2&\mathsf A_2&\mathsf B_2&&\\
  &\ddots&\ddots&\ddots&\\
  &&\mathsf C_{m-1}&\mathsf A_{m-1}&\mathsf B_{m-1}\\
  \mathsf B_m&&&\mathsf C_m&\mathsf A_m
 \end{pmatrix}.
 \label{eq:model}
\end{equation}
Thus row \(j\) of the block equation contains
\(\mathsf C_j\) at the left neighbor, \(\mathsf A_j\) at the current
site, and \(\mathsf B_j\) at the right neighbor.  Every scalar row has
exactly \(3W\) independent entries of variance \((3W)^{-1}\).

\begin{theorem}[Subgaussian, possibly discrete, atoms]
\label{thm:main}
Let \(X_N\) be the periodic full-block band matrix defined in
\eqref{eq:model}.  Assume that \(\xi\) is real and
\(\|\xi\|_{\psi_2}\le K\), where \(K\) is fixed.  Here
\[
 \|Y\|_{\psi_2}:=\inf\left\{t>0:
   \E\exp\left(\frac{|Y|^2}{t^2}\right)\leq2\right\}
\]
denotes the subgaussian Orlicz norm.  Assume further that
\begin{equation}
 W=W_N\longrightarrow\infty,
 \qquad \frac{W}{\log N}\longrightarrow\infty.
 \label{eq:bandwidth}
\end{equation}
Then the empirical eigenvalue measure of \(X_N\) converges in probability
to the uniform probability measure \(\mu_{\cir}\) on the unit disk.
More precisely, for every fixed \(z\in\C\),
\begin{equation}
 \frac1N\log|\det(X_N-zI_N)|
 \xrightarrow{\mathbb P}
 U_{\cir}(z)
 \label{eq:logdet-goal}
\end{equation}
In particular, \(\Pp\{\det X_N=0\}\to0\), so \(X_N\) is nonsingular
with high probability.
\end{theorem}

\begin{remark}[On the real and subgaussian assumptions]
The Rademacher law \(\P\{\xi=1\}=\P\{\xi=-1\}=1/2\) is the main
discrete example.  The restriction to real-valued atoms is inherited
from Nguyen's overcrowding estimate \cite[Theorem~1.4]{Nguyen}, used in
Lemma~\ref{lem:local-interface-control}; it is not imposed by the
deterministic transfer argument.  The subgaussian hypothesis is used
mainly through the same input and the accompanying uniform block-norm
tails.  It is used once more to supply the norm tail in the repaired
full-block least-singular-value estimate,
Proposition~\ref{prop:jjlo-block-lsv}, whereas the high-band
singular-value comparison uses only a finite third moment through
Lemma~\ref{lem:local-finite-moment-bulk}.  We expect that the real and
subgaussian assumptions can be weakened once corresponding complex-valued
overcrowding and weaker-tail singular-value inputs are available.\end{remark}

In the case of bounded density laws, we again recover the optimal $W_N\to\infty$ circular law:

\begin{theorem}[Bounded-density atoms]
\label{thm:density-main}
Let \(X_N\) be the periodic full-block band matrix defined in
\eqref{eq:model}, and assume that \(\xi\) satisfies
Assumption~\ref{ass:indicator-density}.
If
\begin{equation}
 W=W_N\longrightarrow\infty,
 \label{eq:density-bandwidth}
\end{equation}
then the empirical eigenvalue measure of \(X_N\) converges in probability
to \(\mu_{\cir}\).  Moreover, for every fixed \(z\in\C\), the
logarithmic-determinant convergence
\eqref{eq:logdet-goal} holds.
\end{theorem}

\medskip\noindent
The finite-third-moment hypothesis may be replaced by
\(\E|\xi|^{2+\alpha}<\infty\) for any fixed \(0<\alpha\le1\), after
adjusting the mesoscopic scales; see Remark~\ref{rem:third-moment}.

\medskip\noindent
\paragraph{On the density hypothesis and band geometry.} Finally, we discuss the role of the bounded density hypothesis in this paper. The fact that we succeeded in proving circular law in Theorem \ref{thm:main} at the absence of a density hypothesis is tied to the periodic full-block
geometry.  Each block interface is a complete i.i.d. \(W\times W\) square,
and the packet reduction again exposes complete random square panels.
Nguyen's estimate \cite{Nguyen}, subgaussian block-norm tails, and the deformed-square
input of Lemma~\ref{lem:local-cook-input} can therefore replace
density-based small-ball estimates.  Each interface event fails with probability at most
\(e^{-cW}\), while the ring contains \(O(N/W)\) interfaces; hence

\[
 \P\{\text{at least one interface fails}\}
 \leq C\frac{N}{W}e^{-cW}=o(1)
\]
under \(W/\log N\to\infty\).  This is the source of the logarithmic
bandwidth condition.  The proof introduces another weaker error term from Cook \cite{Cook}, but that error is averaged through
the pressure recursion and is not union-bounded over all cells, preserving $W\gg \log N$.

The same mechanism does not presently give the optimal \(W\to\infty\)
result for Bernoulli scalar-indicator bands or generic non-block
profiles, because their boundary reductions do not in general expose the
same complete random squares.  It also does not cover the present full-block
model at sublogarithmic bandwidth \(W=o(\log N)\), where the simultaneous
interface estimate above no longer closes.  To our knowledge, these
optimal non-block and sublogarithmic full-block regimes with Bernoulli entry distribution remain open and
appear to require new discrete anti-concentration or an annealed interface
argument.

\subsection{Lean formalization and generative AI usage}\label{subsec:lean-verification}

The author has used ChatGPT 5.6 to assist with several technical parts of the mathematical proof. The main proof structure remains from the author, and he has independently checked all mathematical proofs and takes full responsibility for the correctness of the argument.

A companion formalization of the whole paper in Lean~4, based on mathlib, is publicly
available.\footnote{\url{https://github.com/hanyi162013-Yihan/random-band-circular-law-lean}}
The project formalizes substantial parts of the arguments in
Sections~3--10, and the proof for all the four main theorems. The formalized
proofs are checked by Lean's kernel without admitted proofs or additional
mathematical axioms except those stated below. We have adopted the following external hard analytic inputs in the project without formalizing them: (1) the free probability comparison
estimates of Bandeira, Boedihardjo and van Handel \cite{bandeira2023matrix}, (2) the estimates of Cook \cite{Cook}
and (3) Nguyen \cite{Nguyen}, (4) the geometric Brascamp--Lieb inequality and (5) a modified version of \cite{JainJanaLuhORourke2021}, Theorem 2.1: they appear explicitly as theorem hypotheses. Among these, all main results of the paper use (1), the results under the density assumption (i) or (iii) in Assumption \ref{ass:indicator-density} uses (4), and only Theorem \ref{thm:main} uses results (2), (3), (5).
Along the project, we have independently formalized other reusable results such as (1) the required Ginibre log potential estimates
 and (2) the Tao--Vu replacement principle, and (3) the mesoscopic singular value counting estimate in \cite{Han2410}.
The formalization verifies these specified proof chains, while the exact intermediate statement or numerical constant in Lean may differ from those printed in this manuscript.

\medskip\noindent\textbf{Acknowledgements.}
Part of this work was completed while the author was affiliated with the
Institute for Advanced Study in Princeton.  During that period, the author
was supported by an IAS fellowship provided by the S.~S. Chern Foundation
for Mathematical Research Fund and the Fund for Mathematics.

\section{The high-band input and mesoscopic scales}
\label{sec:high-band-input}
This section is a common toolbox for both proof parts.  We first record the
least-singular-value, mesoscopic counting, and local bulk inputs used in
both branches.  We then verify two high-band anchors: a finite-moment
short-ring anchor for scalar indicator bands and a subgaussian anchor for
cyclic full-block bands.

The scalar indicator model has a self-similarity property: a fresh local
ring has the same form as the target matrix, but with a different
relation between its size and the bandwidth.  We encode that
self-similarity in the following auxiliary model.

For each ambient target pair $(N,W)$, let
$\mathbf q^{(N,W)}=(q_s^{(N,W)})_{|s|\le W}$ be an admissible weight
vector satisfying \eqref{eq:indicator-weights}.  We keep this vector fixed
while the auxiliary size $M$ varies and suppress its superscript.
For every $M$ with $2W+1\le M$, let $H_{M,W}$ denote the auxiliary cyclic indicator band
\begin{equation}\label{eq:auxiliary-ring}
 (H_{M,W})_{i,i+s}:=\sqrt{q_s}\,\xi_{i,s},
 \qquad i\in\T_M,\quad |s|\le W,
\end{equation}
where the variables are independent copies of an atom $\xi$ satisfying
Assumption~\ref{ass:indicator-density}.  All scalar auxiliary-ring
assertions below are uniform over the admissible weight arrays.  When
calibrating a target $X_N$, this
$\mathbf q$ is its weight vector.  Thus $H_{M,W}$ is a fresh cyclic matrix
of size $M$, not a principal submatrix of $X_N$, and
$X_N\stackrel{d}=H_{N,W_N}$.  For any variance-profile
matrix $A$, write
\[
 \mathfrak b(A):=\left(\max_{i,j}\E|A_{ij}|^2\right)^{-1}.
\]
In particular, $\mathfrak b(H_{M,W})\asymp W$.

\subsection{High-band estimates}
This subsection records least-singular-value estimates and singular-value
rigidity estimates valid for large $W$.
\begin{theorem}[Least singular value bound]
\label{thm:high-band-lsv}
Let $X=(\sigma_{ij}\xi_{ij})_{i,j\in[N]}$, where the $\sigma_{ij}\ge0$ are
deterministic and the $\xi_{ij}$ are independent, centered, and normalized
by $\E|\xi_{ij}|^2=1$.  Suppose that $(\sigma_{ij}^2)$ is doubly stochastic and, for fixed
$c_{\mathrm{lsv}},C_{\mathrm{lsv}}>0$,
\[
 \max_{i,j}\sigma_{ij}^2\le \frac{C_{\mathrm{lsv}}}{W},\qquad
 \sigma_{ij}^2\ge \frac{c_{\mathrm{lsv}}}{W}
 \quad\text{if }|i-j|_N\le W.
\]
Assume uniformly in $i,j,N$ that, in the real case, every $\xi_{ij}$ has
a density bounded by $L$, while in the complex case its law has a density
with respect to planar Lebesgue measure on $\C\simeq\R^2$ bounded by $L$.
Alternative~\textnormal{(iii)} of
Assumption~\ref{ass:indicator-density} is also admissible, provided that
one fixed direction works for every
$i,j,N$ and that the conditional density bound $L$ is uniform.  If
$W\ge N^{1/2+\chi}$ for some fixed $\chi>0$, then for every $K_z,R>0$ and
$\kappa\in(0,\chi/4)$ there is
$C=C(K_z,R,\kappa,\chi,c_{\mathrm{lsv}},C_{\mathrm{lsv}},L)$ such that,
for all sufficiently large $N$,
uniformly in $|z|\le K_z$ and $t>0$,
\[
 \Pp\left\{s_{\min}(X-zI_N)
 \le t\exp\left(-N^{3\kappa}\frac NW\right),\ 
 \|X\|_{\HS}\le R\sqrt N\right\}
 \le Ct+\exp(-N^{1+\kappa/4}).
\]
\end{theorem}

A preliminary version of this least-singular-value estimate was proposed
in \cite{Han2508}.  The present formulation is proved self-containedly in
Appendix~\ref{app:high-band-lsv}.

A related least-singular-value estimate for much more general
inhomogeneous variance profiles was obtained in
\cite[Theorem~3.4]{Tikhomirov2023}.  That result allows an essentially
arbitrary deterministic variance profile, but assumes independent centered
subgaussian atoms with uniformly bounded densities (and, in the complex
case, independent real and imaginary parts).  The proof of
Theorem~\ref{thm:high-band-lsv} is independent of that work: it requires no
moment beyond the normalized second moment, while its doubly stochastic
band-ellipticity hypotheses cover every bounded-density profile to which
the theorem is applied below.

\paragraph{High probability Hilbert-Schmidt norm control}
For the iid matrices used below, the Hilbert--Schmidt cutoff can be removed
at a fixed $R$ using only the second moment.  Set
\[
 a_{ij}:=\frac{\sigma_{ij}^2}{N},\qquad Y_{ij}:=|\xi_{ij}|^2.
\]
Then $\sum a_{ij}=1$ and
$\max a_{ij}\le C_{\mathrm{lsv}}/(NW)$.  For
$Y_{ij}^{(K)}:=Y_{ij}\wedge K$ and
$r_K:=\E(Y_{11}-Y_{11}^{(K)})$, truncation, Markov, and Chebyshev give,
whenever $r_K\le\varepsilon/3$,
\[
 \Pp\left\{\left|\frac{\|X\|_{\HS}^2}{N}-1\right|>\varepsilon\right\}
 \le \frac{3r_K}{\varepsilon}
 +\frac{9C_{\mathrm{lsv}}K}{\varepsilon^2NW}.
\]
Indeed, the second term uses independence and
$(Y_{ij}^{(K)})^2\le K Y_{ij}$.  First let $N\to\infty$ with $K$ fixed
and then let $K\to\infty$.  More quantitatively, if
$K_R$ is fixed so that $r_{K_R}\le(R^2-1)/3$, then
\begin{equation}\label{eq:HS-quantitative-tail}
 \Pp\{\|X\|_{\HS}>R\sqrt N\}
 \le C_R\inf_{K\ge K_R}
 \left\{r_K+\frac{K}{NW}\right\}.
\end{equation}
This is the natural rate under a bare second moment and may be arbitrarily
slow.  If $\E|\xi|^{2+\alpha}<\infty$, choosing
$K=(NW)^{2/(2+\alpha)}$ gives the explicit bound
\begin{equation}\label{eq:HS-quantitative-moment}
 \Pp\{\|X\|_{\HS}>R\sqrt N\}
 \le C_{R,\alpha,C_{\mathrm{lsv}}}\E|\xi|^{2+\alpha}
 (NW)^{-\alpha/(2+\alpha)}.
\end{equation}
Assuming only two moments on $\xi$, by dominated convergence we have $\lim_{K\to\infty}r_K=0$. Then
\begin{equation}\label{eq:lsv-remove-HS-cutoff}
 \frac{\|X\|_{\HS}^2}{N}\xrightarrow{\Pp}1,
 \qquad
 \Pp\{\|X\|_{\HS}>R\sqrt N\}=o(1)
 \quad(R>1\text{ fixed}).
\end{equation}
Taking $R=2$ in Theorem~\ref{thm:high-band-lsv}, for every $t_N\to0$,
\begin{align*}
 &\Pp\left\{s_{\min}(X-zI_N)
 \le t_N\exp\left(-N^{3\kappa}\frac NW\right)\right\}\\
 &\qquad\le C t_N+\exp(-N^{1+\kappa/4})
 +C_2\inf_{K\ge K_2}\left\{r_K+\frac{K}{NW}\right\}=o(1).
\end{align*}
No moment above two is used here.  For non-identically distributed atoms,
the same argument requires uniform integrability of
$\{|\xi_{ij}|^2\}$.

The next proposition is a lightly repaired form of
\cite[Theorem~2.1]{JainJanaLuhORourke2021}.  Its conclusion and main
argument remain valid; only the auxiliary operator-norm estimate used to
construct the good event requires a small replacement under the stated
hypotheses.

\begin{proposition}[Repaired subgaussian full-block LSV]
\label{prop:jjlo-block-lsv}
Fix $K,K_0<\infty$ and a real atom law $\xi$ satisfying
$\E\xi=0$, $\E\xi^2=1$, and $\|\xi\|_{\psi_2}\le K$.  There are
$m_*=m_*(\xi,K_0)$ and $C_{\xi,K_0}<\infty$ such that the following holds.
Let $b\mid N$, write $m=N/b$, and let $\widetilde{\mathsf B}_N$ be the
periodic block-tridiagonal matrix whose diagonal and two cyclic neighboring
block diagonals consist of independent $b\times b$ iid matrices with common
atom $\xi$; all other blocks vanish.  Set
$\mathsf B_N=(3b)^{-1/2}\widetilde{\mathsf B}_N$.  Suppose that
$b\ge m\ge m_*$.
Then, for every deterministic $z\in\C$ with $|z|\le K_0$,
\[
 \Pp\left\{s_{\min}(\mathsf B_N-zI_N)\le(3b)^{-25m}\right\}
 \le \frac{C_{\xi,K_0}}{\sqrt{3b}}.
\]
No absolute-continuity assumption is imposed; in particular, the result
includes discrete atoms such as the Rademacher law.
\end{proposition}

\begin{proof}
See Appendix~\ref{app:repaired-jjlo}.
\end{proof}

\begin{remark}[Relation to the result of Jain--Jana--Luh--O'Rourke]
The least-singular-value argument of
\cite[Theorem~2.1]{JainJanaLuhORourke2021} is retained.  The only
modification concerns its auxiliary good event: the printed
\cite[Proposition~2.3]{JainJanaLuhORourke2021} asserts an
$O(b^{-2})$ tail for the event that an iid $b\times b$ block has norm
larger than $C\sqrt b$, under only a finite $(4+\varepsilon)$ moment.
Such a rate does not follow from that hypothesis: the event that one of
the $b^2$ entries exceeds $C\sqrt b$ may already have probability of
order $b^{-\varepsilon/2}$ up to logarithmic factors.
Appendix~\ref{app:repaired-jjlo} replaces
this step by the standard exponential block-norm tail available under
the subgaussian hypothesis of Theorem~\ref{thm:main}.
\end{remark}

This is precisely the full-block mask in \eqref{eq:model}, with $b=W$.
It is complementary to, but does not replace,
Theorem~\ref{thm:high-band-lsv}, which is the estimate used for the
bounded-density indicator and full-block branches.  We use
Proposition~\ref{prop:jjlo-block-lsv} only once, at the hard edge of the
subgaussian full-block proof below.  Notice also that this
least-singular-value argument only requires
$b\ge N/b$; the stronger condition
$b\ge N^{32/33}\log N$ belongs to the circular-law theorem in the same
paper, not to its least-singular-value estimate.

\begin{proposition}[Mesoscopic singular-value counting]
\label{prop:mesoscopic-counting}
Let $X=(b_{ij}\xi_{ij})_{i,j\in[N]}$ have a doubly stochastic variance
profile, where the $\xi_{ij}$ are independent copies of a fixed centered
unit-variance atom $\xi$ satisfying $\E|\xi|^3<\infty$, and suppose
$\mathfrak b(X)\ge N^{c}$ for some $c>0$.  For fixed $z\in\C$ and
$\tau>0$, with probability at
least $1-N^{-10}$, every interval $I\subset[-5,5]$ of length
$|I|\ge\mathfrak b(X)^{-1/8}N^\tau$ satisfies
\[
 \#\left\{\lambda\in I:\lambda\in\operatorname{Spec}
 \begin{pmatrix}0&X-zI_N\\(X-zI_N)^*&0\end{pmatrix}\right\}
 \le C_zN|I|.
\]
\end{proposition}

This is \cite[Corollary~3.5]{Han2410}, whose underlying resolvent estimate
\cite[Proposition~3.4]{Han2410} requires only a finite third moment.

\begin{lemma}[Local finite-moment bulk comparison]
\label{lem:local-finite-moment-bulk}
Let $A_N=(b_{ij}\xi_{ij})_{i,j\in[N]}$ have a doubly stochastic variance profile, where
the $\xi_{ij}$ are independent copies of an atom satisfying
$\E\xi=0$, $\E|\xi|^2=1$, and $\E|\xi|^3<\infty$.  Denote by
$\mathfrak b_N=(\max_{i,j}b_{ij}^2)^{-1}$ and assume
$\mathfrak b_N\ge N^\varepsilon$.  If $G_N$ is normalized circular
Ginibre, then, for every fixed $z\in\C$ and $R<\infty$, there is
$\zeta>0$ such that
\begin{equation}\label{eq:local-finite-moment-bulk}
 \sup_{0\le x\le R^2}
 \left|\nu_{A_N-zI_N}([0,x])-\nu_{G_N-zI_N}([0,x])\right|
 =O_{\Pp}(N^{-\zeta}),
 \qquad
 \nu_B:=\frac1N\sum_{j=1}^N\delta_{s_j(B)^2}.
\end{equation}
\end{lemma}

\begin{proof}
Let $\mu_A^z$ denote the empirical law of the Hermitization of $A-zI$,
and let $m_A^z$ be its Stieltjes transform.  Let $m_z^\circ$ denote the
deterministic free Hermitized Stieltjes transform written
$m_{z,\mathrm{free}}^o$ in \cite{Han2410}; for a doubly stochastic profile
it is the same scalar reference transform as for normalized Ginibre.  Choose
$0<c<\varepsilon/8$ and set $v_N=\mathfrak b_N^{-1/8}N^c$.
The explicit bound \cite[Proposition~3.4, (3.11)]{Han2410}, together with
the net argument in its proof, gives, for every fixed $R$ and some $d>0$,
\begin{equation}\label{eq:local-stieltjes-comparison}
 \sup_{|u|\le R+2}
 |m_{A_N}^z(u+iv_N)-m_z^\circ(u+iv_N)|
 =O_{\Pp}(N^{-d}).
\end{equation}
The same proof is uniform for $u$ in any fixed compact interval, with
constants depending on that interval.
Indeed, all four terms in (3.11) are polynomially small after this
substitution.  The Gaussian estimate \cite[(3.16)]{Han2410} gives the same
conclusion for $G_N$.

For completeness, convolve interval indicators with the Poisson kernel
$P_v(t)=\pi^{-1}v/(t^2+v^2)$ and enlarge or shrink their endpoints by
$\Delta=\sqrt v$.  The kernel mass outside distance $\Delta$ is
$O(v/\Delta)$, while the reference Hermitized law has bounded density by
\cite[(3.15) and the proof of Theorem~3.6]{Han2410}.  Thus the usual upper
and lower smoothing inequalities on $[-R-1,R+1]$ give
\[
 \sup_{I\subset[-R,R]}
 |\mu_{A_N}^z(I)-\mu_{G_N}^z(I)|
 \le C_R\bigl(N^{-d}+\Delta+v_N/\Delta\bigr)
 =O_{\Pp}(N^{-\zeta}).
\]
No spectral-tail estimate enters because every interval is contained in a
fixed compact set.  Taking $I=[-\sqrt x,\sqrt x]$ proves
\eqref{eq:local-finite-moment-bulk}.
\end{proof}

\subsection{Finite-moment short-ring anchor}

\begin{proposition}[Finite-moment short-ring anchor]
\label{prop:finite-moment-anchor}
Let $H_{M,W}$ be as in \eqref{eq:auxiliary-ring}, with atom satisfying
Assumption~\ref{ass:indicator-density}.  For every
$\omega\in(0,1/9)$, if $M,W\to\infty$ and
\begin{equation}\label{eq:finite-moment-high-band}
 W\ge M^{8/9+\omega},
\end{equation}
then, for every fixed $z\in\C$,
\begin{equation}\label{eq:HB-anchor}
 \frac1M\log|\det(H_{M,W}-zI_M)|
 \xrightarrow{\Pp}U_{\cir}(z).
\end{equation}
\end{proposition}

\begin{proof}
Write $\beta=8/9+\omega$ and choose $\chi,\kappa,\tau>0$ so small that
\begin{equation}\label{eq:hard-edge-parameter-choice}
 \frac12+\chi<\beta,\qquad \kappa<\frac\chi4,\qquad
 \tau<\frac\beta8,\qquad \tau+3\kappa<\frac{9\beta}{8}-1,
\end{equation}
which is possible because $\beta>8/9$.  Apply Theorem
\ref{thm:high-band-lsv} with $t=M^{-2}$ and Proposition
\ref{prop:mesoscopic-counting} to the interval
\[
 [-a_M,a_M],\qquad
 a_M:=\mathfrak b(H_{M,W})^{-1/8}M^\tau
 \asymp W^{-1/8}M^\tau=o(1).
\]
With probability $1-o(1)$,
\[
 s_{\min}(H_{M,W}-zI_M)\ge e^{-L_M},\qquad
 \#\{j:s_j(H_{M,W}-zI_M)\le a_M\}\le CMa_M,
\]
where $L_M=M^{1+3\kappa}/W+2\log M$.  Here
\eqref{eq:indicator-weights} satisfies the least-singular-value hypotheses
with $c_{\mathrm{lsv}}=c_0/3$ and
$C_{\mathrm{lsv}}=C_0$.  Consequently,
\begin{align}
 \frac1M\left|\sum_{s_j\le a_M}\log s_j\right|
 &\le Ca_ML_M \notag\\
 &\le C\left(M^{1+\tau+3\kappa}W^{-9/8}
       +M^\tau W^{-1/8}\log M\right)=o(1).       \label{eq:hard-edge-sum}
\end{align}

It remains to identify the logarithm above $a_M$.  Let $G_M$ be normalized
circular Ginibre and define $\nu_A=M^{-1}\sum_j\delta_{s_j(A)^2}$.
For each fixed $R>1$, Lemma~\ref{lem:local-finite-moment-bulk} gives
\begin{equation}\label{eq:finite-moment-bulk-comparison}
 \sup_{0\le x\le R^2}
 |\nu_{H_{M,W}-zI_M}([0,x])-\nu_{G_M-zI_M}([0,x])|
 =O_{\Pp}(M^{-\zeta}).
\end{equation}
Since $a_M$ is polynomially small, integration by parts, equivalently
\cite[Lemma~4.2]{Han2410}, yields
\begin{equation}\label{eq:truncated-log-bulk-comparison}
 \left|\frac12\int_{a_M^2}^{R^2}\log x\,
 d(\nu_{H_{M,W}-zI_M}-\nu_{G_M-zI_M})(x)\right|
 \le C_R(1+|\log a_M|)M^{-\zeta}=o_{\Pp}(1).
\end{equation}
No operator-norm estimate is needed at the upper edge.  Indeed, for
$R>\sqrt e$ and $A=H_{M,W}$ or $G_M$,
\begin{equation}\label{eq:upper-log-tail-second-moment}
 \E\frac1M\sum_{s_j(A-zI_M)>R}\log s_j(A-zI_M)
 \le \frac{\log R}{R^2}\frac1M\E\|A-zI_M\|_{\HS}^2
 =(1+|z|^2)\frac{\log R}{R^2}.
\end{equation}

Finally, the Ginibre small-singular-value argument in
\cite[(4.9), Corollary~4.10, and Lemmas~4.11--4.12]
{BordenaveChafai2012}
gives, for some $p=p(z)>0$,
\[
 \frac1M\sum_{j=1}^M s_j(G_M-zI_M)^{-p}=O_{\Pp}(1).
\]
For $0<s\le a\le1$,
$\log(1/s)\le(2/p)a^{p/2}s^{-p}$; hence, because
$a_M\lesssim M^{-(\beta/8-\tau)}$,
\begin{equation}\label{eq:ginibre-hard-edge}
 \frac1M\sum_{s_j(G_M-zI_M)\le a_M}
 |\log s_j(G_M-zI_M)|=o_{\Pp}(1).
\end{equation}
The full Ginibre logarithmic determinant converges to $U_{\cir}(z)$ by the
same references.  Combine \eqref{eq:hard-edge-sum},
\eqref{eq:truncated-log-bulk-comparison},
\eqref{eq:upper-log-tail-second-moment}, and
\eqref{eq:ginibre-hard-edge}, first letting $M\to\infty$ and then
$R\to\infty$, to obtain \eqref{eq:HB-anchor}.
\end{proof}
We now choose the mesoscopic scales.  Fix
\begin{equation}\label{eq:scale-choice}
 0<\delta<\gamma<\frac18,\qquad
 0<\omega_0<\frac1{1+\gamma}-\frac89,\qquad
 m_0(W):=\lceil W^{1+\delta}\rceil.
\end{equation}
If $M\le W^{1+\gamma}$, then
\begin{equation}\label{eq:mesoscopic-window-high-band}
 W\ge M^{1/(1+\gamma)}\ge M^{8/9+\omega_0}.
\end{equation}
For large $W$, this covers both $M\in[m_0(W),2m_0(W)]$ and the direct
target regime $N\le W^{1+\gamma}$.
\begin{remark}[Lower moments]\label{rem:third-moment}
For both bounded-density branches, the finite third moment enters only
through Proposition~\ref{prop:mesoscopic-counting} and
Lemma~\ref{lem:local-finite-moment-bulk}.  Under
$\E|\xi|^{2+\alpha}<\infty$, $0<\alpha\le1$, the underlying resolvent
comparison in the proof of Proposition \ref{prop:mesoscopic-counting} can instead be run at scale
\[
 \mathfrak b_N^{-\rho_\alpha}N^\tau,
 \qquad \rho_\alpha:=\frac{\alpha}{2(3+\alpha)}.
\]
Then a similar version of Proposition \ref{prop:finite-moment-anchor} works for
\begin{equation}\label{eq:lower-moment-anchor-scale}
 W\ge M^{\theta_\alpha+\omega},
 \qquad \theta_\alpha:=\frac1{1+\rho_\alpha}
 =\frac{2(3+\alpha)}{6+3\alpha}.
\end{equation}
The indicator lifting in Part~1 then uses the window in
\eqref{eq:lower-moment-anchor-scale}, with
$0<\delta<\gamma<\rho_\alpha$; the bounded-density
full-block anchor in Part~2 is adjusted at the same scale.  The lifting
arguments themselves are unchanged, and we can then prove the circular
law for both bounded-density models under the finite-$(2+\alpha)$-moment
assumption.  We omit this longer resolvent
bookkeeping and retain the finite third moment in the theorem; at
$\alpha=1$ the displayed exponents are $1/8$ and $8/9$.  Reaching exactly
two moments is qualitatively different: the available singular-value
rigidity estimates for band and inhomogeneous matrices require a moment
strictly above two.
\end{remark}

\subsection{Subgaussian full-block high-band anchor}

\begin{proposition}[High-band log potential for subgaussian full-block rings]
\label{prop:subgaussian-block-high-band}
Let $X_N$ be the cyclic full-block matrix \eqref{eq:model}, with $N=mW$,
and let the atom be real, centered, variance one, and uniformly subgaussian.
For every fixed $\omega\in(0,1/9)$, if $N,W\to\infty$ and
\begin{equation}
 W\ge N^{8/9+\omega},
 \label{eq:subgaussian-block-high-band}
\end{equation}
then, for every fixed $z\in\C$,
\begin{equation}
 \frac1N\log|\det(X_N-zI_N)|
 \xrightarrow{\Pp}U_{\cir}(z).
 \label{eq:gb-high-band-logdet}
\end{equation}
\end{proposition}

\begin{proof}
Set $\beta=8/9+\omega$ and choose
\begin{equation}
 0<\tau<\min\left\{\frac\beta8,\frac{9\beta}{8}-1\right\},
 \qquad a_N:=(3W)^{-1/8}N^\tau=o(1).
 \label{eq:subgaussian-block-cutoff}
\end{equation}
Write $m=N/W$.  Since $\beta>1/2$, the high-band condition implies
$m\le W^{(1-\beta)/\beta}\le W$ for all large $N$.  If $m\ge m_*$, the sole use of the
Jain--Jana--Luh--O'Rourke estimate,
Proposition~\ref{prop:jjlo-block-lsv}, gives
\begin{equation}
 s_{\min}(X_N-zI_N)\ge (3W)^{-25m}
 \label{eq:subgaussian-block-floor}
\end{equation}
outside an event of probability $O(W^{-1/2})$.  For the finitely many
values $4\le m<m_*$, multiply by $\sqrt{3W}$ and apply the
broadly-connected-profile estimate \cite[Theorem~1.12]{Cook}.  Each fixed
cyclic three-neighbor block mask is broadly connected: at the block
level, the radius-one cyclic neighborhood of every nonempty proper set
contains an adjacent block, and each supported block is complete
bipartite.  Since only finitely many $m$ occur, the connectivity
parameters may be chosen uniformly.  Since \(W\asymp N\) in this case,
the standard subgaussian matrix-norm tail
\cite[Theorem~4.4.5]{vershynin2019high} gives the fixed
\(O(\sqrt N)\) operator-norm event required there with probability
\(1-o(1)\).
Taking $t=N^{-1/4}$ in that theorem gives
$s_{\min}(\sqrt{3W}(X_N-zI_N))\ge N^{-3/4}$ with probability
$1-o(1)$, and hence $s_{\min}(X_N-zI_N)\ge N^{-C}$.
Thus in both
cases, with probability $1-o(1)$,
\begin{equation}
 s_{\min}(X_N-zI_N)\ge e^{-L_N},
 \qquad
 L_N\le C\left(\frac NW\log(3W)+\log N\right).
 \label{eq:subgaussian-block-unified-floor}
\end{equation}

The scalar variance profile of $X_N$ is doubly stochastic and has
$\mathfrak b(X_N)=3W$.  Proposition~\ref{prop:mesoscopic-counting},
applied to $[-a_N,a_N]$, therefore gives
\(\#\{j:s_j(X_N-zI_N)\le a_N\}\le C_zN a_N\) with probability
$1-o(1)$.  On the intersection with
\eqref{eq:subgaussian-block-unified-floor},
\begin{align}
 \frac1N\sum_{s_j(X_N-zI_N)\le a_N}|\log s_j(X_N-zI_N)|
 &\le C_za_NL_N \notag\\
 &\le C_z\left(N^{1+\tau}W^{-9/8}\log(3W)
       +N^\tau W^{-1/8}\log N\right)=o(1),
 \label{eq:subgaussian-block-hard-edge}
\end{align}
where the last equality follows from \eqref{eq:subgaussian-block-high-band}
and the choice of $\tau$.

Let $G_N$ be normalized circular Ginibre and write
$\nu_A:=N^{-1}\sum_j\delta_{s_j(A)^2}$.  By
Lemma~\ref{lem:local-finite-moment-bulk}, for every fixed $R>1$ and some
$\zeta>0$,
\begin{equation}
 \sup_{0\le x\le R^2}
 |\nu_{X_N-zI_N}([0,x])-\nu_{G_N-zI_N}([0,x])|
 =O_{\Pp}(N^{-\zeta}).
 \label{eq:subgaussian-block-bulk}
\end{equation}
Integration by parts on $[a_N^2,R^2]$ makes the corresponding truncated
logarithms differ by
$O_{\Pp}((1+|\log a_N|)N^{-\zeta})=o_{\Pp}(1)$.
No operator-norm estimate is needed at the upper edge: for $R>\sqrt e$
and $A=X_N$ or $G_N$,
\begin{equation}
 \E\frac1N\sum_{s_j(A-zI_N)>R}\log s_j(A-zI_N)
 \le (1+|z|^2)\frac{\log R}{R^2}.
 \label{eq:subgaussian-block-upper-tail}
\end{equation}
The Ginibre inverse-moment estimate quoted in the proof of
Proposition~\ref{prop:finite-moment-anchor} shows that its logarithmic
mass below $a_N$ is $o_{\Pp}(1)$, while its full logarithmic determinant
converges to $U_{\cir}(z)$.  Combining these facts with
\eqref{eq:subgaussian-block-hard-edge} and
\eqref{eq:subgaussian-block-bulk}, then letting first $N\to\infty$ and
then $R\to\infty$, proves \eqref{eq:gb-high-band-logdet}.
\end{proof}

\part{Scalar indicator bands and Gaussian profiles}

\section{Exterior transfer and local density tools}
This section has two main outputs.  First, the periodic determinant identity
\eqref{eq:periodic-det} expresses $\det(X_N-zI_N)$ through the exterior
powers of a transfer product.  Second, the bounded-density hypothesis yields
the fresh-closure estimate of Proposition~\ref{prop:fresh-closure}, the
projective estimate of Lemma~\ref{lem:projective}, and the pressure
concentration bound of Proposition~\ref{prop:pressure-concentration}.
\subsection{Scalar transfer and the periodic determinant}
\label{sec:transfer}

We now prove Theorem~\ref{thm:indicator}.  The direct regime
$N\le W^{1+\gamma}$ is already covered by
Proposition~\ref{prop:finite-moment-anchor}.  Hence suppose
\begin{equation}\label{eq:long-regime}
 N>W^{1+\gamma}.
\end{equation}
In particular $2W<N$ for all sufficiently large $W$.  For $X_N-zI_N$,
write the coefficient in row $i$ at signed offset $s$, $-W\le s\le W$, as
\begin{equation}\label{eq:row-coefficients}
 a_{i,s}:=b_s\xi_{i,s}-z\1_{\{s=0\}},\qquad
 \alpha_i:=a_{i,-W},\qquad \beta_i:=a_{i,W}.
\end{equation}
$\alpha_i$ and $\beta_i$ are the leftmost and rightmost entries. Since $q_{-W},q_W>0$ and Assumption~\ref{ass:indicator-density} implies
$\Pp\{\xi=0\}=0$, both
$\alpha_i=b_{-W}\xi_{i,-W}$ and
$\beta_i=b_W\xi_{i,W}$ are nonzero almost surely.

For a solution of $(X_N-zI_N)u=0$, use the $2W$-dimensional state
\[
 s_i=(u_{i-W},u_{i-W+1},\ldots,u_{i+W-1})^{\mathsf T}.
\]
To determine the relation between successive states, let
$S:\mathbb C^{2W}\to\mathbb C^{2W}$ be the left shift
$Se_j=e_{j-1}$, with $e_0=0$, and write
\(c_i=(a_{i,-W},\ldots,a_{i,W-1})^{\mathsf T}\).  Solving the $i$th row
for $u_{i+W}$ gives the linear operator
$T_i:\mathbb C^{2W}\to\mathbb C^{2W}$ defined by
\begin{equation}\label{eq:companion-transfer}
 s_{i+1}=T_i s_i,\qquad
 T_i=S-\beta_i^{-1}e_{2W}c_i^{\mathsf T}.
\end{equation}
Here the transpose denotes the coordinate bilinear pairing
\[
 c_i^{\mathsf T}s_i:=\sum_{j=1}^{2W}(c_i)_j(s_i)_j,
\]
with no complex conjugation.  Thus
$e_{2W}c_i^{\mathsf T}$ maps $s_i$ to
$e_{2W}(c_i^{\mathsf T}s_i)$.

By a standard computation we verify that
\begin{equation}\label{eq:companion-determinant}
 \det T_i=\frac{\alpha_i}{\beta_i},
\end{equation}
so that matrices $T_i$ and all their exterior powers are invertible almost
surely.

For $0\le r\le2W$, define the denominator-cleared exterior transfer
\begin{equation}\label{eq:exterior-transfer}
 \mathcal A_i^{(r)}:=\beta_i\,\bigwedge\nolimits^rT_i,\qquad
 \mathcal A_{[p,q]}^{(r)}
 :=\mathcal A_q^{(r)}\cdots\mathcal A_p^{(r)}.
\end{equation}
The scalar factor in \eqref{eq:exterior-transfer} is one copy of
$\beta_i$, independently of $r$.  Let
$(e_1^*,\ldots,e_{2W}^*)$ be the basis dual to
$(e_1,\ldots,e_{2W})$.  For each $r\ge1$, we define the contraction by $e_j^*$ as the following map
\begin{align}
 \iota_{e_j^*}:\bigwedge^r\C^{2W}&\longrightarrow
 \bigwedge^{r-1}\C^{2W},                                      \label{eq:contraction-definition}\\
 \iota_{e_j^*}(v_1\wedge\cdots\wedge v_r)
 &:=\sum_{k=1}^r(-1)^{k-1}e_j^*(v_k)\,
 v_1\wedge\cdots\wedge\widehat{v_k}\wedge\cdots\wedge v_r.
 \notag
\end{align}
We write $e_{2W}\wedge\,\cdot$ for left exterior multiplication by
$e_{2W}$.

To better exploit the independence of the random entries $a_{i,s}$, we rewrite the exterior product $\mathcal A_i^{(r)}$ as an affine sum over the random entries $a_{i,s}$:
\begin{lemma}[Row-linearity]\label{lem:row-linearity}
There are deterministic contractions $K_s^{(r)}$ such that
\begin{equation}\label{eq:row-linear-exterior}
 \mathcal A_i^{(r)}=\sum_{s=-W}^W a_{i,s}K_s^{(r)}.
\end{equation}
More explicitly, with $j=s+W+1$ for $-W\le s<W$,
\begin{equation}\label{eq:coefficient-operators}
 K_W^{(r)}=\bigwedge\nolimits^rS,\qquad
 K_s^{(r)}=-\big(e_{2W}\wedge\,\cdot\big)\circ
 \big(\bigwedge\nolimits^{r-1}S\big)\circ\iota_{e_j^*},\qquad
 \|K_s^{(r)}\|\le1.
\end{equation}
For $r=0$, the second expression is zero and $K_W^{(0)}=1$.
\end{lemma}

\begin{proof}
We expand the wedge of the rank-one perturbation in
\eqref{eq:companion-transfer} depending on how many times the expansion uses the perturbation $e_{2W}c_i^{\mathsf T}$.  Terms using the perturbation twice contain
$e_{2W}\wedge e_{2W}$ and vanish.  The term using it zero times is
$\beta_i\,\bigwedge\nolimits^rS$; the terms using it once are exactly the
contractions in
\eqref{eq:coefficient-operators}.
\end{proof}

Then we relate the determinant $\det(X_N-zI_N)$ to the trace of these transfer operators:
\begin{lemma}[Periodic determinant identity]\label{lem:periodic-det}
There is a deterministic sign $\varsigma_{N,W}\in\{\pm1\}$ such that
\begin{equation}\label{eq:periodic-det}
 \det(X_N-zI_N)=\varsigma_{N,W}
 \sum_{r=0}^{2W}(-1)^r\tr\mathcal A_{[1,N]}^{(r)}.
\end{equation}
\end{lemma}

\begin{proof}
Let $\mathscr L$ be the cyclic block matrix of the state equations
$s_{i+1}-T_is_i=0$:
\[
 \mathscr L=
 \begin{pmatrix}
 -T_1&I_{2W}&0&\cdots&0\\
 0&-T_2&I_{2W}&\ddots&\vdots\\
 \vdots&&\ddots&\ddots&0\\
 0&\cdots&0&-T_{N-1}&I_{2W}\\
 I_{2W}&0&\cdots&0&-T_N
 \end{pmatrix}.
\]
Eliminating $s_2,\ldots,s_N$ leaves
$(I_{2W}-P)s_1=0$, where $P=T_N\cdots T_1$; hence
\begin{equation}\label{eq:block-state-determinant}
 \det\mathscr L=\varepsilon_1\det(I_{2W}-P)
\end{equation}
for a deterministic sign $\varepsilon_1\in\{\pm 1\}$.

For the second evaluation, work on the probability-one event
\[
 \Omega_\beta:=\{\beta_i\ne0\text{ for every }i\in\T_N\}.
\]
Write $d:=2W$, and write $s_{i,k}$ for the $k$th coordinate of $s_i$, with
row indices understood modulo $N$.  Relabel the state coordinates by
\[
 v_{j,k}:=s_{j+W-k+1,k},
 \qquad j\in\T_N,\quad 1\le k\le d.
\]
For fixed $j$, all $v_{j,k}$ represent the same underlying scalar
coordinate $u_j$.  Indeed, the $k$th equation, $k<d$, in block row
$i=j+W-k$ is
\[
 s_{i+1,k}-s_{i,k+1}=v_{j,k}-v_{j,k+1}=0.
\]
Now set
\[
 w_{j,k}:=v_{j,k}-v_{j,k+1}\quad(1\le k<d),
 \qquad u_j:=v_{j,d}.
\]
For every $j$ this is a determinant-one triangular change of variables.
After fixed row and column permutations, the first $N(d-1)$ equations
are therefore the identity block $w_{j,k}=0$.

Modulo these identification equations, the last equation in block row
$i$ becomes
\begin{align*}
 s_{i+1,d}+\beta_i^{-1}\sum_{k=1}^d
     a_{i,-W+k-1}s_{i,k}
 &=\beta_i^{-1}\left(\beta_i u_{i+W}
     +\sum_{s=-W}^{W-1}a_{i,s}u_{i+s}\right)\\
 &=\beta_i^{-1}\bigl((X_N-zI_N)u\bigr)_i.
\end{align*}
Before imposing $w=0$, these last $N$ equations may contain additional
terms linear in $w$.  Adding suitable multiples of the identification
rows removes them without changing the determinant.  Thus, up to the
fixed row and column permutations, $\mathscr L$ reduces to
\[
 \begin{pmatrix}
  I_{N(d-1)}&0\\
  0&D_\beta^{-1}(X_N-zI_N)
 \end{pmatrix},
 \qquad D_\beta:=\operatorname{diag}(\beta_1,\ldots,\beta_N).
\]
Only the fixed permutations contribute a sign.  Hence, on $\Omega_\beta$,
\[
 \det\mathscr L=\varepsilon_2
 \left(\prod_{i=1}^N\beta_i^{-1}\right)\det(X_N-zI_N)
\]
for a deterministic $\varepsilon_2=\varepsilon_2(N,W)\in\{\pm1\}$.
Combining the two evaluations gives
\[
 \det(X_N-zI_N)=\varsigma_{N,W}\Big(\prod_{i=1}^N\beta_i\Big)
 \det(I_{2W}-P).
\]
Finally,
$\det(I_{2W}-P)=\sum_{r=0}^{2W}(-1)^r\tr(\wedge^rP)$ and exterior powers are
functorial.  Since
$\mathcal A_{[1,N]}^{(r)}=(\prod_i\beta_i)\wedge^rP$, the identity follows.
\end{proof}

The alternating sum in \eqref{eq:periodic-det} may exhibit cancellation.
We therefore reserve the final $2W$ rows and use their randomness to compare
the evaluated sum with
$\max_{0\le r\le2W}\|\mathcal A_{[1,N-2W]}^{(r)}\|$.
This is the purpose of the next two subsections.

\subsection{The full exterior reset word}
\label{sec:reset}

Let $E=\C^{2W}$ and use the orthonormal basis
$e_I=e_{i_1}\wedge\cdots\wedge e_{i_r}$ of $\wedge^rE$, indexed by
$I=\{i_1<\cdots<i_r\}\subset[2W]$.  It is useful to regard all exterior
degrees simultaneously in the full exterior space
$\mathscr F(E)=\bigoplus_{r=0}^{2W}\wedge^rE$.
Throughout this section, Hermitian inner products are linear in their second
argument.

Define the label-to-offset map
\[
 \operatorname{off}(\star):=W,\qquad
 \operatorname{off}(j):=-W+j-1\quad(j\in[2W]),
\]
and write $K_\ell^{(r)}:=K_{\operatorname{off}(\ell)}^{(r)}$ for
$\ell\in\{\star\}\cup[2W]$.  On the full exterior space, set
\[
 K_\ell:=\bigoplus_{r=0}^{2W}K_\ell^{(r)}.
\]
Up to a sign of modulus one,
$K_\star$ shifts every occupied site one step left and kills a state
containing site $1$.  The reset $K_j$ removes the particle at $j$, shifts
all remaining particles left, and inserts a particle at $2W$; it is zero if
the contraction or one of the shifts is impossible.

The following combinatorial lemma selects one coefficient from the resulting
matrix polynomials by using this explicit shift--reset action.
\begin{lemma}[Singleton-domain word]\label{lem:singleton-word}
For every $I,J\subset[2W]$ with $|I|=|J|=r$, there is a word
$\omega=(\ell_1,\ldots,\ell_{2W})
\in(\{\star\}\cup[2W])^{2W}$ such that
\begin{equation}\label{eq:singleton-word}
 K_{\ell_{2W}}^{(r)}\cdots K_{\ell_1}^{(r)}
 =\varsigma_{I,J}|e_I\rangle\langle e_J|,\qquad
 K_{\ell_{2W}}^{(q)}\cdots K_{\ell_1}^{(q)}=0\quad(q\ne r),
\end{equation}
where $|\varsigma_{I,J}|=1$ and the operator $|e_I\rangle\langle e_J|$ is the rank-one operator mapping $v$ to $e_I\langle e_J,v\rangle$.
\end{lemma}

\begin{proof}
For any $A\subset[2W]$, the support action is
\[
 K_\star e_A=
 \begin{cases}
  e_{\{a-1:a\in A\}},&1\notin A,\\
  0,&1\in A,
 \end{cases}
 \qquad
 K_je_A=
 \begin{cases}
  \pm e_{\{a-1:a\in A\setminus\{j\}\}\cup\{2W\}},
       &j\in A,\ 1\notin A\setminus\{j\},\\
  0,&\text{otherwise}.
 \end{cases}
\]
The sign in the second formula comes from exterior ordering and is absorbed
into $\varsigma_{I,J}$.

\textbf{The prefix step.} Write $J=\{j_1<\cdots<j_r\}$ and first use the $r$ reset labels
\[
 \ell_t:=p_t=j_t-t+1,\qquad 1\le t\le r,
\]
as the prefix.  Thus the original particle at $j_t$ is at $p_t$ after the
first $t-1$ left shifts.
For the input $J$, after $t$ steps the remaining $r-t$ sites are moved leftward, and the first $t$ sites are sent to the rightmost, so that the sites are now at
\[
 \{j_{t+1}-t,\ldots,j_r-t\}
 \cup\{2W-t+1,\ldots,2W\};
\]
in particular, the prefix sends $J$ to
$H_r=\{2W-r+1,\ldots,2W\}$.  Conversely, survival through this prefix
forces the original input to contain every $j_t$.  Thus every surviving
input is of the form $J$ together with extra particles, and after the
prefix its state is $H_r\cup E_0$ for some $E_0\subset[2W-r]$.

\textbf{The second phase: sending to coordinates $I$.} To send $H_r$ to $I$, define
\[
 I_-=I\cap[r],\qquad K_I=[r]\setminus I_-,\qquad
 T_I=\{t\in[2W-r]:r+t\in I\}.
\]
Since $|K_I|=|T_I|$, choose a bijection $\phi:T_I\to K_I$.  During the
remaining $2W-r$ steps use
\[
 \ell_{t+r}:=q_t=\begin{cases}
 \star,&t\notin T_I,\\
 2W-r+\phi(t)-t+1,&t\in T_I.
 \end{cases},\quad 1\le t\le 2W-r.
\]
A particle $k\in[r]$, initially at $2W-r+k$, is left unreset precisely when
$k\in I_-$ and then finishes at $k$.  If $\phi(t)=k$, it is reset at time
$t$ and finishes at $r+t$.  Hence the suffix sends $H_r$ to $I$.

It remains to check the map is nonzero only on $e_J$.  If $E_0\ne\varnothing$, take an
extra particle initially at $x\in E_0$.  Until it dies, its position at time
$t$ is $x-t+1$.  A planned reset is strictly to its right because
\[
 q_t-(x-t+1)=2W-r+\phi(t)-x\ge1.
\]
At time $t=x$ the extra particle is therefore at site $1$.  A shift kills
it, while a planned reset has label $q_x\ge2$ and also leaves the site-$1$
factor to be killed by the shift.  Thus every strict superset of $J$ dies.
Inputs of size less than $r$ cannot contain $J$; inputs of size $r$ survive
only when equal to $J$; and inputs of larger size have an extra particle.
This proves both assertions in \eqref{eq:singleton-word}.
\end{proof}
The next lemma concerns a product of exactly $2W$ operators, the length of
the reset word in Lemma~\ref{lem:singleton-word}.  Later we split the full
transfer into an outside product and a final product of $2W$ fresh rows.
\begin{lemma}[An isolated full monomial]\label{lem:isolated-monomial}
Let $Q^{(r)}=\mathcal A_{2W}^{(r)}\cdots\mathcal A_1^{(r)}$ be formed from
$2W$ fresh rows with full randomness.  For arbitrary deterministic operators $B^{(r)}$ on $\wedge^rE$ defined for each $0\le r\le 2W$, set
\[
 Z_B:=\sum_{r=0}^{2W}(-1)^r\tr(B^{(r)}Q^{(r)}),\qquad
 M_B:=\max_{0\le r\le2W}\|B^{(r)}\|.
\]
If $M_B>0$, then, among all the $(2W+1)\times 2W$ row variables in these $2W$ rows, $Z_B$ has a square-free monomial
of degree $2W$ whose coefficient $c_B$ satisfies
\begin{equation}\label{eq:isolated-coefficient}
 |c_B|\ge M_B\exp\{-C W\log(eW)\}.
\end{equation}
\end{lemma}

Here and below, $\E_Q$ denotes expectation over the $2W$ fresh rows
defining $Q$, conditionally on all other randomness when present.

\begin{proof}
Choose $r$ attaining $M_B$ and coordinate sets $I,J\subset[2W]$ such that
 \begin{equation}\label{isolateds}
 |\langle e_J,B^{(r)}e_I\rangle|
 \ge\binom{2W}{r}^{-1}\|B^{(r)}\|\ge2^{-2W}M_B.
\end{equation}
Expand each $\mathcal A_i^{(r)}$ by
\eqref{eq:row-linear-exterior}, and use the word $\omega$ from
Lemma~\ref{lem:singleton-word} associated with the coordinate sets $I,J$.
In row $t$, select the coefficient with label $\ell_t$.  The resulting
monomial contains one distinct atom from each of the $2W$ rows, and its
coefficient in $Z_B$ is
\[
 \left(\prod_{t=1}^{2W}b_{\operatorname{off}(\ell_t)}\right)
 \sum_{q=0}^{2W}(-1)^q
 \operatorname{tr}\!\left(
 B^{(q)}K_{\ell_{2W}}^{(q)}\cdots K_{\ell_1}^{(q)}\right).
\]
By \eqref{eq:indicator-weights}, the weight factor is at least
$\exp(-CW\log(eW))$.  By \eqref{eq:singleton-word}, the same word vanishes
in every exterior degree $q\ne r$.  In degree $r$, use
$K_{\ell_{2W}}^{(r)}\cdots K_{\ell_1}^{(r)}
=\varsigma_{I,J}|e_I\rangle\langle e_J|$ and \eqref{isolateds}.
If a selected coefficient is diagonal, the deterministic translation by
$-z$ changes only lower-degree monomials.  Therefore the full degree-$2W$
coefficient obeys \eqref{eq:isolated-coefficient}.
\end{proof}

\subsection{Density estimates and one fresh closure}
\label{sec:density-scan}
Lemma~\ref{lem:isolated-monomial} extracts a large coefficient of full degree
$2W$ in $Z_B$.  We now derive small-ball estimates from that coefficient,
without estimating the remaining coefficients separately.
\begin{lemma}[Multiaffine density bound]\label{lem:multiaffine}
Let $P(x_1,\ldots,x_k)$ be a complex-valued polynomial which is affine in
each of $k$ independent variables, and suppose that the coefficient of
$x_1\cdots x_k$ is $c\ne0$.  If the variables are real with densities
bounded by $L$, then, for $0<\rho<1$,
\begin{equation}\label{eq:multiaffine-smallball}
 \Pp\{|P|\le |c|\rho^k\}\le 2Lk\rho.
\end{equation}
If instead they are complex with planar densities bounded by $L$, then
\begin{equation}\label{eq:multiaffine-complex-smallball}
 \Pp\{|P|\le |c|\rho^k\}\le \pi Lk\rho^2.
\end{equation}
In either case,
\begin{equation}\label{eq:multiaffine-log}
 \E\left(\log\frac{|c|}{|P|}\right)_+
 \le C_L k\log(ek).
\end{equation}
The same logarithmic conclusion holds under alternative~\textnormal{(iii)}
of Assumption~\ref{ass:indicator-density}.
\end{lemma}

\begin{proof}
We use induction on $k$; for $k=1$, integrate the density over an interval
in the real case and over a disk in the complex case.  Write
$P=x_kP_1+P_0$.  On $|P_1|\ge|c|\rho^{k-1}$, conditioning on all variables
except $x_k$ bounds the probability of $|P|\le|c|\rho^k$ by $2L\rho$ in
the real case and by $\pi L\rho^2$ in the complex case.  The induction
hypothesis controls the complement and proves
\eqref{eq:multiaffine-smallball} and
\eqref{eq:multiaffine-complex-smallball}.  Letting $\rho\downarrow0$ in
either bound gives $\Pp\{P=0\}=0$.  Taking $\rho=e^{-t/k}$ and
integrating $\min\{1,C_Lke^{-dt/k}\}$, with $d=1$ or $2$, proves
\eqref{eq:multiaffine-log}.

Under the directional alternative, condition on all orthogonal
coordinates $V_1,\ldots,V_k$.  Because the pairs $(U_j,V_j)$ are
independent across $j$, the variables $U_j$ remain conditionally
independent.  The polynomial is multiaffine in them, their conditional
densities are bounded by $L$, and the modulus of the top coefficient remains
$|c|$.  Apply the real estimate conditionally and then average.
\end{proof}

\begin{proposition}[One fresh closure]\label{prop:fresh-closure}
For the $2W$ fresh rows in Lemma~\ref{lem:isolated-monomial}, uniformly over
all deterministic families $B=(B^{(r)})_{r=0}^{2W}$ with $M_B>0$,
\begin{equation}\label{eq:fresh-closure-L1}
 \E_Q\left|\log|Z_B|-\log M_B\right|
 \le C W\log(eW).
\end{equation}
The same estimate holds conditionally when $B$ is measurable with respect to
a sigma-field independent of the fresh rows.
\end{proposition}

\begin{proof}
For the negative part, condition on every fresh variable except the $2W$
atoms in the monomial from Lemma~\ref{lem:isolated-monomial}. Note that $Z_B$ is separately affine in each random entry. Applying
\eqref{eq:multiaffine-log} and then
\eqref{eq:isolated-coefficient} gives
\[
 \E_Q(\log M_B-\log|Z_B|)_+\le CW\log(eW).
\]
For the positive part, the trace bound in every exterior degree gives
\[
 |Z_B|\le2^{2W}M_B\max_{0\le r\le2W}\|Q^{(r)}\|.
\]
By row-linearity and \(\|K_s^{(r)}\|\le1\),
\[
 \max_r\log_+\|Q^{(r)}\|
 \le\sum_{i=1}^{2W}\log_+\Big(\sum_{s=-W}^W|a_{i,s}|\Big).
\]
The expectation of each summand is $O(\log(eW))$ by
\eqref{eq:atom-normalization} and Jensen's inequality.  This proves the
positive half of \eqref{eq:fresh-closure-L1}.
\end{proof}

We also need a lower bound on the action of a fresh transfer product on one
prescribed starting vector.

\begin{lemma}[Projective observability]\label{lem:projective}
Fix $r$, a deterministic unit vector $v\in\wedge^r\C^{2W}$, and a
deterministic nonzero operator $B$ on $\wedge^r\C^{2W}$.  For a fresh
product $Q^{(r)}$ defined in Lemma \ref{lem:isolated-monomial},
\begin{equation}\label{eq:projective}
 \E_Q\log\|BQ^{(r)}v\|\ge\log\|B\|-CW\log(eW).
\end{equation}
The conclusion remains valid conditionally when $B$ and $v$ are measurable
with respect to the past, while the fresh rows in $Q^{(r)}$ are independent of that past.
\end{lemma}

\begin{proof}
Let $u,x$ be top left and right singular vectors of $B$.  Choose coordinates
$I,J$ with
$|\langle e_I,x\rangle|,|\langle e_J,v\rangle|
\ge\binom{2W}{r}^{-1/2}$.
As in the proof of Lemma \ref{lem:isolated-monomial}, the singleton word determined by Lemma \ref{lem:singleton-word} which sends $J$ to $I$ isolates in
$\langle u,BQ^{(r)}v\rangle$ a full monomial of coefficient at least
$\|B\|\binom{2W}{r}^{-1}(c_*W^{-1/2})^{2W}$, where
$c_*:=\sqrt{c_0/3}$.
Lemma~\ref{lem:multiaffine} gives the
lower bound for the logarithm of this scalar pairing, which is bounded
above by $\log\|BQ^{(r)}v\|$.
\end{proof}

\subsection{Concentration inequality for the pressure and transfer products}
\label{sec:concentration}
Finally, we prove concentration of
$\log\|\mathcal A_{[1,n]}^{(r)}\|$ around its mean.  In the
Efron--Stein resampling step, freezing all factors except one reduces the
needed estimate to the following single-row affine lemma.

\begin{lemma}[Operator-valued affine logarithm]\label{lem:operator-affine}
Let $W\ge1$, let $b_{-W},\ldots,b_W$ be deterministic, and assume that
\[
 |b_0|\ge c_bW^{-1/2}
\]
for some fixed $c_b>0$.  Let $(\xi_s)_{s=-W}^W$ be independent copies of
a centered atom $\xi$ satisfying $\E|\xi|^2=1$.  Assume either that $\xi$
is real with a Lebesgue density bounded by $L$, or that $\xi$ is complex
with a density bounded by $L$ with respect to planar Lebesgue measure.
Alternative~\textnormal{(iii)} of
Assumption~\ref{ass:indicator-density} is also admissible.  For arbitrary
deterministic square matrices $M_s$ of the same size, denote
\[
 G=\sum_{s=-W}^W b_s\xi_sM_s-zM_0,\qquad
 \mu=\max\left\{\|M_0\|,
       \max_{-W\le s\le W}|b_s|\|M_s\|\right\}>0.
\]
Then
\begin{equation}\label{eq:operator-affine}
 \E|\log\|G\|-\log\mu|^2\le C_{z,c_b,L}\log^2(eW).
\end{equation}
\end{lemma}

\begin{proof}
If \(\mu=\max_s|b_s|\|M_s\|\), choose \(s_*\) attaining this maximum;
otherwise take \(s_*=0\).  Since \(|b_0|\geq c_bW^{-1/2}\), in either case
\[
 |b_{s_*}|\|M_{s_*}\|\geq c_bW^{-1/2}\mu.
\]
Test $G$ between unit vectors which approximately attain
$\|M_{s_*}\|$.  In the real case, condition on all variables except
$\xi_{s_*}$ and integrate its one-dimensional density.  In the complex
case, integrate the planar density over the resulting disk.  This gives,
respectively,
\[
 \Pp\{\|G\|\le t\mu\}\le C_{c_b,L}\sqrt W\,t,
 \qquad
 \Pp\{\|G\|\le t\mu\}\le\min\{1,C_{c_b,L}Wt^2\}.
\]
Since $\min(1,at^2)\le\sqrt a\,t$, the first displayed bound, after
changing its constant, holds in both cases for $0<t<1$.  Under the
directional alternative, condition also on all orthogonal coordinates
and use the bounded conditional density; the same real bound results.
The center created by the conditioned variables is irrelevant.
This controls the squared negative logarithm by
\(C_{c_b,L}\log^2(eW)\).
For the positive part,
\[
 \frac{\|G\|}{\mu}
 \le\sum_{s=-W}^W|\xi_s|+|z|.
\]
We upper bound this sum using only the second moment. Denote by
$S_W=\sum_{s=-W}^W|\xi_s|$.  Cauchy--Schwarz gives
\[
 \E S_W^2\le(2W+1)\sum_{s=-W}^W\E|\xi_s|^2\le C W^2.
\]
Moreover,
\[
 \log_+^2(S_W+|z|)
 \le C_z\log^2(eW)+C\log_+^2(1+S_W/W).
\]
The expectation of the last term is bounded because
$\log(1+x)\le x$ and $\E(S_W/W)^2\le C$.  This proves the required
$O(\log^2(eW))$ squared positive logarithm using only
\eqref{eq:atom-normalization}.
\end{proof}

For the indicator weights in \eqref{eq:indicator-weights},
\[
 b_0=\sqrt{q_0}\ge \sqrt{c_0/(2W+1)}
     \ge \sqrt{c_0/3}\,W^{-1/2},
\]
so the preceding lemma applies with a fixed $c_b>0$.

Let $n$ be any positive integer.
For a stationary open product of $n$ independent rows define
\begin{equation}\label{eq:pressure-def}
 Y_r(n):=\log\|\mathcal A_{[1,n]}^{(r)}\|,\qquad
 F_r(n):=\E Y_r(n),\qquad 0\le r\le2W.
\end{equation}
Then we have the following maximal concentration:
\begin{proposition}[Pressure concentration]\label{prop:pressure-concentration}
Uniformly in $0\le r\le2W$,
\begin{equation}\label{eq:pressure-variance}
 \operatorname{Var}Y_r(n)\le Cn\log^2(eW),
\end{equation}
and
\begin{equation}\label{eq:max-pressure-concentration}
 \E\max_{0\le r\le2W}|Y_r(n)-F_r(n)|
 \le C\sqrt{Wn}\log(eW).
\end{equation}
\end{proposition}

\begin{proof}
For each $i$ we resample row $i$, and write $Y_r^{(i)}(n)$ for the log norm
after replacing that row by an independent copy.  Conditional on all other
rows, we use \eqref{eq:row-linear-exterior} to rewrite the term
$\mathcal A_i^{(r)}$; then both the original and
resampled full products have the form
\[
 L_i\mathcal A_i^{(r)}R_i
 =\sum_{s=-W}^W a_{i,s}M_{i,s}
\]
with the same deterministic coefficient operators $M_{i,s}$.  The scale
\[
 \mu_i=\max\{\|M_{i,0}\|,\max_s|b_s|\|M_{i,s}\|\}
\]
is therefore the same in the two copies.
It is strictly positive almost surely.  Indeed,
$|\det T_j|=|\alpha_j/\beta_j|>0$, so the left and right histories are
invertible; and the family $(K_s^{(r)})_s$ is not identically zero in any
exterior degree.

Since every $\mathcal A_j^{(r)}$ is invertible almost surely,
\[
 |Y_r(n)|
 \le \sum_{j=1}^n\left(
 \log_+\|\mathcal A_j^{(r)}\|
 +\log_+\|(\mathcal A_j^{(r)})^{-1}\|\right).
\]
Row-linearity, the inverse identity
\[
 \|(\mathcal A_j^{(r)})^{-1}\|
 =\frac{\|\mathcal A_j^{(2W-r)}\|}{|\alpha_j\beta_j|},
\]
the second moment of $\xi$, and the bounded-density estimates
\[
 \Pp\{|\xi|\le t\}\le 2Lt\quad\text{in the real case},
 \qquad
 \Pp\{|\xi|\le t\}\le \pi Lt^2\quad\text{in the complex case},
\]
with the first bound also valid under the directional alternative, show
that every summand belongs to
$L^2$.  Hence $Y_r(n)\in L^2$, as required by Efron--Stein.

Applying Lemma~\ref{lem:operator-affine} separately to the original and
resampled rows yields
\[
 \E\big[(Y_r(n)-\log\mu_i)^2\mid\text{all other rows}\big]
 +\E\big[(Y_r^{(i)}(n)-\log\mu_i)^2
       \mid\text{all other rows}\big]
 \le C\log^2(eW).
\]
Since the two copies have the same conditional scale $\mu_i$, the
inequality
\[
 (Y_r(n)-Y_r^{(i)}(n))^2
 \le 2(Y_r(n)-\log\mu_i)^2
     +2(Y_r^{(i)}(n)-\log\mu_i)^2
\]
therefore gives
\[
 \E\big[(Y_r(n)-Y_r^{(i)}(n))^2\mid\text{all other rows}\big]
 \le C\log^2(eW).
\]
The Efron--Stein inequality \cite{EfronStein1981} proves
\eqref{eq:pressure-variance}.  Finally apply
Cauchy--Schwarz to the sum of the $2W+1$ variances:
\[
\E\max_{0\le r\le 2W}|Y_r(n)-F_r(n)|
 \le\Big(\sum_{r=0}^{2W}\E|Y_r(n)-F_r(n)|^2\Big)^{1/2}
 \le C\sqrt{Wn}\log(eW).
\]
\end{proof}

\section{Mesoscopic calibration and global pressure lifting}
This section completes the proof of Theorem~\ref{thm:indicator} by the
pressure-uplifting strategy described above; its final subsection also
proves Corollary~\ref{cor:tapered-indicator}.  We first calibrate the mean
exterior pressure on a mesoscopic ring, then lift that calibration to the
target length and perform one final cyclic closure.
\subsection{Mesoscopic calibration}
\label{sec:calibration}
Our first step is to turn the circular-law input in
Proposition~\ref{prop:finite-moment-anchor} into a deterministic limit for
\(\max_{0\leq r\leq2W}F_r\).
Recall $m_0(W)=\lceil W^{1+\delta}\rceil$ defined in \eqref{eq:scale-choice}.
Let \(m\in[m_0(W),2m_0(W)]\) and set \(\ell:=m-2W\).  On an auxiliary cyclic
ring of length $m$, reserve its first $2W$ rows and write
\begin{equation}\label{eq:cell-decomposition}
 Q^{(r)}=\mathcal A_{[1,2W]}^{(r)},\qquad
 B^{(r)}=\mathcal A_{[2W+1,m]}^{(r)},\qquad
 \mathcal A_{[1,m]}^{(r)}=B^{(r)}Q^{(r)}.
\end{equation}
We use the random and mean exterior pressures \(Y_r(\cdot)\) and
\(F_r(\cdot)\) defined in \eqref{eq:pressure-def} in the preceding
section.  The two products are independent.  The density hypothesis makes
the edge coefficients nonzero almost surely, so the companion formula
makes every \(\mathcal A_i^{(r)}\), and hence
\(B^{(r)}\) and \(Q^{(r)}\), invertible.  In particular, all
normalizations below are well defined.  By
Lemma~\ref{lem:periodic-det} and Proposition~\ref{prop:fresh-closure},
conditional on the bulk,
\begin{equation}\label{eq:det-versus-random-pressure}
 \E_Q\left|
 \log|\det(H_{m,W}-zI_m)|-\max_{0\le r\le2W}Y_r(\ell)
 \right|\le CW\log(eW).
\end{equation}
Proposition~\ref{prop:pressure-concentration} also gives
\begin{equation}\label{eq:random-versus-mean-pressure}
 \E\left|\max_{0\leq r\leq2W}Y_r(\ell)
 -\max_{0\leq r\leq2W}F_r(\ell)\right|
 \le C\sqrt{Wm}\log(eW).
\end{equation}
Since $m\asymp W^{1+\delta}$, both errors in
\eqref{eq:det-versus-random-pressure}--
\eqref{eq:random-versus-mean-pressure} are $o(m)$.

The auxiliary ring lies in the range verified in
\eqref{eq:scale-choice}.  Its normalized logarithmic determinant therefore
converges to $U_{\cir}(z)$ by
Proposition~\ref{prop:finite-moment-anchor}.  Since the quantity
\(\max_{0\leq r\leq2W}F_r(\ell)\) is deterministic, the preceding
\(L^1\) comparison forces
\begin{equation}\label{eq:mesoscopic-calibration}
 \sup_{m_0(W)\le m\le2m_0(W)}
 \left|\frac1m\max_{0\le r\le2W}F_r(m-2W)-U_{\cir}(z)\right|
 \longrightarrow0.
\end{equation}
Indeed, if the supremum did not tend to zero, one could choose a violating
sequence $m(W)$; the high-band input and
\eqref{eq:det-versus-random-pressure}--
\eqref{eq:random-versus-mean-pressure} would then give two different limits
in probability for the same scalar random variable.

\subsection{Global pressure lifting}
\label{sec:pressure-lifting}

The calibration determines the largest mean pressure on one mesoscopic
cell.  We next concatenate arbitrarily many such cells.  This is where the
projective form of the fresh-row estimate is used.

Fix \(m\in[m_0(W),2m_0(W)]\).  For \(0\leq r\leq2W\), define a cell by
\begin{equation}\label{eq:cell-transfer}
 C^{(r)}=B^{(r)}Q^{(r)},
\end{equation}
where $Q^{(r)}$ consists of the product of $\mathcal A_i^{(r)}$ from $2W$ fresh rows and $B^{(r)}$ contains the product of $\mathcal A_i^{(r)}$ from the following
$\ell=m-2W$ bulk rows.  Lemma~\ref{lem:projective}, first conditional on
$B^{(r)}$, implies that for every deterministic unit vector $v$,
\begin{equation}\label{eq:cell-lower}
 \E\log\|C^{(r)}v\|
 \ge F_r(\ell)-CW\log(eW).
\end{equation}
In the other direction, submultiplicativity and the estimate used in the
positive half of Proposition~\ref{prop:fresh-closure} give
\begin{equation}\label{eq:cell-upper}
 \E\log\|C^{(r)}\|
 \le F_r(\ell)+\E\log_+\|Q^{(r)}\|
 \le F_r(\ell)+CW\log(eW).
\end{equation}

For any $q\in\mathbb{N}_+$, we take $q$ independent cells $C_1^{(r)},\ldots,C_q^{(r)}$.  For the lower
bound, fix a unit vector $v_0$ and define successively
\[
 v_j=\frac{C_j^{(r)}v_{j-1}}{\|C_j^{(r)}v_{j-1}\|}.
\]
The vector $v_{j-1}$ is measurable with respect to the preceding cells, so
the conditional version of \eqref{eq:cell-lower} applies to cell $j$.
The cells are independent copies, and the concatenated product
$C_q^{(r)}\cdots C_1^{(r)}$ has the same law as
$\mathcal A_{[1,qm]}^{(r)}$; its expected log norm is therefore $F_r(qm)$.
Telescoping and taking expectations yields the lower bound below; the upper
bound follows by multiplying \eqref{eq:cell-upper}.  Thus, for every
\(q\geq1\) and every \(0\leq r\leq2W\),
\begin{equation}\label{eq:pressure-lift-r}
 q\{F_r(\ell)-CW\log(eW)\}
 \le F_r(qm)
 \le q\{F_r(\ell)+CW\log(eW)\}.
\end{equation}
Maximizing is now harmless: use a maximizing $r$ on the lower side and
maximize the upper side separately.  We obtain
\begin{equation}\label{eq:pressure-lift-max}
 \left|\frac1{qm}\max_{0\le r\le 2W}F_r(qm)
 -\frac1m\max_{0\le r\le 2W}F_r(m-2W)\right|
 \le C\frac{W\log(eW)}m.
\end{equation}
Together with \eqref{eq:mesoscopic-calibration}, this proves, uniformly in
\(q\geq1\) and \(m\in[m_0(W),2m_0(W)]\) (so that
\(m\sim W^{1+\delta}\)),
\begin{equation}\label{eq:global-pressure-multiples}
 \frac1{qm}\max_{0\leq r\leq2W}F_r(qm)=U_{\cir}(z)+o(1).
\end{equation}
No event is intersected over the $q$ cells: the loss in
\eqref{eq:cell-lower} is already integrated before the cells are summed.

\subsection{Remainders, the final closure, and the indicator theorem}
\label{sec:global-closure}
We now remove the divisibility restriction, extend
\eqref{eq:global-pressure-multiples} to every \(N\), and complete the
proof of the circular law.
We first control a terminal segment of small length:

\begin{lemma}[One-row forward and inverse cost]\label{lem:inverse-row}
Almost surely, every $\mathcal A_i^{(r)}$ is invertible, and
\begin{equation}\label{eq:inverse-exterior-identity}
 \|\big(\mathcal A_i^{(r)}\big)^{-1}\|
 =\frac{\|\mathcal A_i^{(2W-r)}\|}{|\alpha_i\beta_i|}.
\end{equation}
Moreover, uniformly in \(0\leq r\leq2W\),
\begin{equation}\label{eq:one-row-log-cost}
 \E\log_+\|\mathcal A_i^{(r)}\|
 +\E\log_+\|\big(\mathcal A_i^{(r)}\big)^{-1}\|
 \le C\log(eW).
\end{equation}
\end{lemma}

\begin{proof}
The companion formula gives
$|\det T_i|=|\alpha_i/\beta_i|$.  If the singular values of an invertible
$2W\times2W$ matrix $T$ are $s_1\ge\cdots\ge s_{2W}$, then
\[
 \|\big(\wedge^rT\big)^{-1}\|
 =\prod_{j=2W-r+1}^{2W}s_j^{-1}
 =\frac{\|\wedge^{2W-r}T\|}{|\det T|}.
\]
The product is interpreted as empty when \(r=0\).
Substituting $\mathcal A_i^{(r)}=\beta_i\wedge^rT_i$ proves
\eqref{eq:inverse-exterior-identity}.  Row-linearity bounds the forward
logarithm by
$\log_+(\sum_s|a_{i,s}|)$, whose expectation is $O(\log(eW))$ by
\eqref{eq:atom-normalization}.  Formula
\eqref{eq:inverse-exterior-identity} reduces
the inverse cost to the forward cost and
$\log_+|\alpha_i\beta_i|^{-1}$.  The real and directional hypotheses give
$\Pp\{|\xi|\le t\}\le2Lt$, while the planar-density hypothesis gives
$\Pp\{|\xi|\le t\}\le\pi Lt^2$.  Thus
$\E\log_+|\xi|^{-1}<\infty$, and the two edge weights contribute only
$O(\log W)$.
\end{proof}

We next use a perturbation argument to identify the maximum mean pressure
at every length.
Let $R$ be a product of $\mathcal A_i^{(r)}$ from $s$ consecutive rows and $P$ an independent product
next to it.  Applying submultiplicativity both to $RP$ and to $R^{-1}RP$
gives the pathwise inequality
\begin{equation}\label{eq:pathwise-remainder}
 |\log\|RP\|-\log\|P\||
 \le\log_+\|R\|+\log_+\|R^{-1}\|.
\end{equation}
By Lemma~\ref{lem:inverse-row}, the expected cost of appending or removing
$s$ rows is at most $Cs\log(eW)$.  Hence the same bound holds for the
change of every $F_r$.

Return to the long target \eqref{eq:long-regime} and reserve \(2W\) rows
for one final closure.  Set \(n:=N-2W\) and define
\begin{equation}\label{eq:balanced-cell-division}
 q=\left\lfloor\frac n{m_0(W)}\right\rfloor,\qquad
 m=\left\lfloor\frac nq\right\rfloor,\qquad
 s=n-qm.
\end{equation}
Since \(\delta<\gamma\), we have \(q\to\infty\),
\(m\in[m_0(W),2m_0(W)]\), and \(0\leq s<q\).  From
\eqref{eq:pathwise-remainder},
\begin{equation}\label{eq:mean-pressure-remainder}
 \max_{0\leq r\leq2W}|F_r(n)-F_r(qm)|\le Cs\log(eW),
 \qquad
 \frac{s\log(eW)}N\le C\frac{\log(eW)}{m_0(W)}=o(1).
\end{equation}
Equations \eqref{eq:global-pressure-multiples} and
\eqref{eq:mean-pressure-remainder}, together with $W=o(N)$, yield
\[
 \frac{qm}{N}=1-\frac{2W+s}{N}\longrightarrow1,
\]
and hence
\begin{equation}\label{eq:target-mean-pressure}
 \frac1N\max_{0\le r\le2W}F_r(N-2W)\longrightarrow U_{\cir}(z).
\end{equation}

Now we can complete the proof of Theorem~\ref{thm:indicator}:
\begin{proof}[\proofname\ of Theorem~\ref{thm:indicator}]
We only need to consider the case \(N>W^{1+\gamma}\).
Use the reserved $2W$ rows as the fresh product $Q$ and the other $N-2W$ rows
as the outside product $B$.  Proposition~\ref{prop:fresh-closure} and
Proposition~\ref{prop:pressure-concentration} give the following estimates.
If the reserved rows occur at the end of the written product, we use
$\tr(Q^{(r)}B^{(r)})=\tr(B^{(r)}Q^{(r)})$ before applying the fresh-closure
proposition:
\begin{align}
 \E\left|\log|\det(X_N-zI_N)|-\max_{0\le r\le 2W}Y_r(N-2W)\right|
 &\le CW\log(eW),                                           \label{eq:final-seam}\\
 \E\left|\max_{0\le r\le 2W}Y_r(N-2W)-\max_{0\le r\le 2W}F_r(N-2W)\right|
 &\le C\sqrt{WN}\log(eW).                                  \label{eq:final-pressure-fluctuation}
\end{align}
In the given regime of $(N,W)$,
\[
 \frac{W\log(eW)}N+\sqrt{\frac WN}\log(eW)
 \le C\big(W^{-\gamma}\log(eW)+W^{-\gamma/2}\log(eW)\big)=o(1).
\]
Combining this with \eqref{eq:target-mean-pressure} proves
\eqref{eq:indicator-logdet-goal}.  Notice that there is only one final
closure; no union bound over interfaces is present.
Taking \(z=0\) in \eqref{eq:indicator-logdet-goal} gives
\(\Pp\{\det X_N=0\}\to0\), since
\(U_{\cir}(0)=-1/2\) is finite whereas the left-hand side equals
\(-\infty\) on the singularity event.

Finally,
\begin{equation}\label{eq:HS-tightness-indicator}
 \frac1N\E\|X_N\|_{\HS}^2=1.
\end{equation}
Thus the normalized Hilbert--Schmidt norms are tight.  The Ginibre matrix
has the same limiting logarithmic potential, and the replacement principle
\cite[Theorem~4.1]{Han2410} applies to
\eqref{eq:indicator-logdet-goal} and \eqref{eq:HS-tightness-indicator}.
It identifies the limiting empirical eigenvalue measure with
\(\mu_{\cir}\) and completes the proof of
Theorem~\ref{thm:indicator}.
\end{proof}

\subsection{Polynomially controlled endpoint taper}
Essentially the same proof generalizes to profiles $f$ vanishing polynomially at the boundary, as defined in Corollary~\ref{cor:tapered-indicator}.
\begin{proof}[Proof of Corollary~\ref{cor:tapered-indicator}]
Let $H^{\rm tap}_{M,W}$ be the fresh cyclic $M\times M$ matrix defined as
in \eqref{eq:auxiliary-ring}, with $q_s$ replaced by $q_s^{\rm tap}$.
Its variance profile is doubly stochastic.  The $O(W)$ indices with
$|s|\le W/2$ satisfy
$1-|s|/(W+1)\ge1/2$, whereas for every $|s|\le W$ this quantity is at least
$(W+1)^{-1}$.  Together with the global upper bound on $f$, this gives
\begin{equation}\label{eq:tapered-discrete-bounds}
 Z_W^{\rm tap}\asymp W,\qquad
 \max_{|s|\le W}q_s^{\rm tap}\le\frac CW,\qquad
 \min_{|s|\le W}q_s^{\rm tap}\ge cW^{-1-\kappa},
\end{equation}
and $q_s^{\rm tap}\ge c/W$ when $|s|\le W/2$.

Denote by $W'=\lfloor W/2\rfloor$.  On the inner band,
\[
 \max_{i,j}\E|H^{\rm tap}_{ij}|^2\le\frac CW\le\frac{C'}{W'},
 \qquad
 \E|H^{\rm tap}_{ij}|^2\ge\frac cW\ge\frac{c'}{W'}
 \quad\text{if }|i-j|_M\le W'.
\]
Moreover, $\mathfrak b(H^{\rm tap}_{M,W})\asymp W$.  If
$W\ge M^{8/9+\omega}$, then
$W'\ge M^{8/9+\omega/2}$ for all large $M$.  Thus
Theorem~\ref{thm:high-band-lsv} applies through the inner band, while
Proposition~\ref{prop:mesoscopic-counting} and
Lemma~\ref{lem:local-finite-moment-bulk} use the global upper bound in
\eqref{eq:tapered-discrete-bounds}.  The Hilbert--Schmidt cutoff uses only
the row-variance normalization.  Repeating the proof of
Proposition~\ref{prop:finite-moment-anchor} therefore proves the required
short-ring input under the same condition $W\ge M^{8/9+\omega}$.

For the long transfer, the deterministic weight of the monomial isolated
in Lemma~\ref{lem:isolated-monomial} is now bounded below by
\[
 \big(cW^{-(1+\kappa)/2}\big)^{2W}
 \ge \exp\{-C_\kappa W\log(eW)\}.
\]
Consequently, the same word proofs of
Lemmas~\ref{lem:isolated-monomial} and~\ref{lem:projective} hold with loss
$C_\kappa W\log(eW)$.  Also
$b_0^{\rm tap}=\sqrt{q_0^{\rm tap}}\ge c_bW^{-1/2}$ for some fixed
$c_b>0$, as required in Lemma~\ref{lem:operator-affine}, while
$q_{\pm W}^{\rm tap}\ge cW^{-1-\kappa}$ makes the inverse cost in
Lemma~\ref{lem:inverse-row} at most $C_\kappa\log(eW)$ per row.  Thus every
estimate from
Proposition~\ref{prop:fresh-closure} through
Subsection~\ref{sec:global-closure} remains valid with $C$ replaced by
$C_\kappa$.  Since $m_0(W)=W^{1+\delta}+O(1)$ and
$N>W^{1+\gamma}$ in the long branch, all normalized additional losses are
still $o(1)$.  The final Hilbert--Schmidt identity is unchanged because
\(\sum_s q_s^{\rm tap}=1\).  The proof of
Theorem~\ref{thm:indicator} now
applies without any further change, including its \(z=0\)
nonsingularity conclusion.
\end{proof}

\section{Gaussian compact cores and profiles without compact support}
\label{sec:gaussian-general}

In this section we prove Theorem~\ref{thm:gaussian-profile}, the circular
law for \(G_{N,f}\) when \(f\) is not compactly supported.  We truncate
\(X_N\) at half-width \(RW\), import the circular law for the resulting
compact core, and show that the tail outside the cutoff vanishes in the
limit.

The key step is \eqref{eq:tail-Jensen-lower}, where Jensen's formula
discards the tail after the band cutoff.  This step critically uses the
rotational invariance of the entries \(g_{i,s}\).

\subsection{Fixed-cutoff compact profile input}
\label{sec:fixed-cutoff-core}

The noncompact Gaussian argument will need a fixed singular-value cutoff.
We record that conclusion separately because it does not require a hard-edge
estimate at the large target size.

For an \(N\times N\) matrix \(A\) and \(a>0\), define the raw and
truncated logarithmic singular-value functionals
\begin{equation}\label{eq:La-def}
 L_0(A):=\frac1N\log|\det A|,\qquad
 L_a(A):=\frac1N\sum_{j=1}^N\log(s_j(A)\vee a).
\end{equation}
Let $\nu_{v,z}$ be the limiting singular-value law of
$\sqrt v\,G_N-zI_N$, where $G_N$ is normalized circular Ginibre matrix
(with each entry of variance \(N^{-1}\)), and write
\begin{equation}\label{eq:Lv-def}
 U_v(z):=\int\log s\,d\nu_{v,z}(s)
 =\frac12\log v+U_{\cir}(z/\sqrt v),\qquad
 \mathcal L_{v,z}(a):=\int\log(s\vee a)\,d\nu_{v,z}(s).
\end{equation}

We recall the definition of our random band matrix:
\begin{definition}\label{CYCLIC5.1}[Circular Gaussian cyclic scalar band] Write $\mathbb T_N:=\mathbb Z/N\mathbb Z$ and choose the complete set of
cyclic displacements
\[
 \mathcal D_N:=
 \left\{-\left\lfloor\frac N2\right\rfloor,\ldots,
 \left\lceil\frac N2\right\rceil-1\right\}.
\]
We define the active band of length $2H+1$ as the following subset of $\mathcal D_N$:
\[
 \mathcal S_{N,H}:=\{-H,\ldots,H\}\subset\mathcal D_N,
 \qquad 2H+1\le N,
\]
and let $(q_s)_{s\in\mathcal S_{N,H}}$ be deterministic nonnegative
weights.  A circular Gaussian cyclic scalar band with weights
$(q_s)$ is the $N\times N$ matrix $A_N$ defined by
\[
 (A_N)_{i,i+s}:=\sqrt{q_s}\,g_{i,s},
 \qquad i\in\mathbb T_N,\quad s\in\mathcal S_{N,H},
\]
where addition in the column index is taken modulo $N$ and the
$g_{i,s}$ are independent standard circular complex Gaussian
variables.  All remaining entries are zero.  We call the matrix
normalized if
\[
 \sum_{s\in\mathcal S_{N,H}}q_s=1.
\]
Its diagonal variance is positive when $q_0>0$.
\end{definition}

\begin{lemma}[Gaussian raw-potential concentration]
\label{lem:compact-Gaussian-concentration}
Let \(\mathcal S_N\subseteq\mathcal D_N\) be any deterministic active
set containing \(0\), and let
\[
 (A_N)_{i,i+s}:=\sqrt{q_s}\,g_{i,s},
 \qquad i\in\mathbb T_N,\quad s\in\mathcal S_N,
\]
with all other entries zero.  If \(q_0\geq c/N\), then, for every fixed
\(r>0\) and \(z\in\C\),
\begin{equation}\label{eq:compact-Gaussian-variance}
\operatorname{Var}\log|\det(rA_N-zI_N)|
 \le C_{r,z,c}N\log^2(eN).
\end{equation}
\end{lemma}

\begin{proof}
Condition on all rows except row $i$.  Cofactor expansion has the form
\[
 \det(rA_N-zI_N)
 =\sum_{s\in\mathcal S_N}
   r\sqrt{q_s}\,g_{i,s}C_{i,i+s}-zC_{i,i}.
\]
The minor $C_{i,i}$ is nonzero almost surely: as a polynomial in the
remaining Gaussian entries, it contains the product of their diagonal
variables with nonzero coefficient.  Hence
\[
 \mu_i:=\max_{s\in\mathcal S_N} r\sqrt{q_s}|C_{i,i+s}|
 \ge r\sqrt{q_0}|C_{i,i}|>0.
\]
After division by $\mu_i$, the random part is circular Gaussian with
variance between $1$ and $N$ since there are at most $N$ Gaussians, while the translation has modulus at most
$|z|/(r\sqrt{q_0})\le C_{r,z,c}\sqrt N$.  Gaussian small-ball and upper-tail
estimates therefore give, conditionally,
\[
 \E\left[
 \left|\log|\det(rA_N-zI_N)|-\log\mu_i\right|^2
 \,\middle|\,\text{all rows except }i\right]
 \le C_{r,z,c}\log^2(eN).
\]
The same Gaussian logarithmic estimate gives
\(\log|\det(rA_N-zI_N)|\in L^2\).  Let \(Y\) be this logarithm and
\(Y^{(i)}\) its value after independently resampling row \(i\).
Conditionally on all other rows, \(Y\) and \(Y^{(i)}\) share the same
center \(\log\mu_i\), and hence
\[
 (Y-Y^{(i)})^2
 \leq
 2(Y-\log\mu_i)^2+2(Y^{(i)}-\log\mu_i)^2.
\]
Taking the conditional expectation, using the preceding bound for both
terms, and summing the Efron--Stein inequality over \(i\) proves
\eqref{eq:compact-Gaussian-variance}.
\end{proof}

\begin{proposition}[Compact Gaussian core]\label{prop:compact-gaussian-core}
Let $A_N$ be a normalized circular Gaussian cyclic scalar band whose active
offsets form a centered interval containing zero of length comparable with
$W_N$, and whose active variances are bounded above and below by fixed multiples of
\(W_N^{-1}\).  In the notation of
Definition~\ref{CYCLIC5.1}, this means that \(H\asymp W_N\) and
\(q_s\asymp W_N^{-1}\) uniformly for \(s\in\mathcal S_{N,H}\).
If \(W_N\to\infty\), then, for fixed \(r>0\), \(a>0\), and every fixed
\(z\in\C\),
\begin{align}
\E L_0(rA_N-zI_N)&\longrightarrow U_{r^2}(z),                \label{eq:compact-core-raw}\\
 \E L_a(rA_N-zI_N)&\longrightarrow\mathcal L_{r^2,z}(a).      \label{eq:compact-core-cut}
\end{align}
\end{proposition}

\begin{proof}
Write \(W:=W_N\).
After replacing $W_N$ by a fixed multiple, this is a profile of the form
\eqref{eq:indicator-weights}.  Circular Gaussian atoms satisfy
Assumption~\ref{ass:indicator-density}, so
Theorem~\ref{thm:indicator} applies; scaling gives
\(L_0(rA_N-zI_N)=\log r+L_0(A_N-(z/r)I_N)\).  Thus the random
quantity \(L_0(rA_N-zI_N)\) converges in probability to
\(U_{r^2}(z)\).  In the long regime \(N>W^{1+\gamma}\), the
$L^1$ conclusion follows directly from
\eqref{eq:final-seam}--\eqref{eq:final-pressure-fluctuation}.  In the direct
high-band regime $N\le W^{1+\gamma}$, Lemma~\ref{lem:compact-Gaussian-concentration} gives
\[
 \frac1N\log|\det(rA_N-zI_N)|
 -\E L_0(rA_N-zI_N)\longrightarrow0
 \quad\text{in }L^2.
\]
Together with the convergence in probability in
Proposition~\ref{prop:finite-moment-anchor}, this forces the deterministic
mean to converge to $U_{r^2}(z)$ and proves
\eqref{eq:compact-core-raw} in expectation in all regimes.

We next prove \eqref{eq:compact-core-cut}.  Because of the fixed lower
cutoff \(a\), the least singular value no longer matters; only the
squared-singular-value comparison in
\cite[Theorem~3.6]{Han2410} is needed.  That theorem requires the matrix
size not to be too large relative to the bandwidth, so in the
complementary regime we apply it after decomposing the matrix into
smaller pieces.  In the direct regime \(N\leq W^{1+\gamma}\), use
\begin{equation}\label{eq:fixed-scale-La}
 L_a(rA_N-zI_N)=\log r+L_{a/r}(A_N-(z/r)I_N)
\end{equation}
and apply the squared-singular-value comparison
\cite[Theorem~3.6]{Han2410} to the normalized matrix.  In the complementary
regime partition
$[N]$ into consecutive intervals of lengths
$m_j\in[m_0(W),2m_0(W)]$.  On each interval replace the band by its cyclic
version, using the same entries away from the two endpoints and independent
entries on the boundary rows.  Denote the resulting block-diagonal matrix
by $\widehat A_N$.  If $H\asymp W$ is the half-width, only $O(H)$ rows per
cut are changed; hence the coupling satisfies
\begin{equation}\label{eq:periodicization-HS}
 \frac1N\E\|A_N-\widehat A_N\|_{\HS}^2
 \le C\frac H{m_0(W)}=o(1).
\end{equation}
Mirsky's singular-value inequality \cite{Mirsky1960} and the
$a^{-1}$-Lipschitz property of
$s\mapsto\log(s\vee a)$ give
\begin{equation}\label{eq:periodicization-La}
 \E|L_a(rA_N-zI_N)-L_a(r\widehat A_N-zI_N)|
 \le\frac r{a\sqrt N}\E\|A_N-\widehat A_N\|_{\HS}=o(1).
\end{equation}

Every diagonal block has length in
\([m_0(W),2m_0(W)]\) and effective bandwidth
comparable with $W$, so \cite[Theorem~3.6]{Han2410}, together with
\eqref{eq:fixed-scale-La}, applies to each normalized block.
It gives convergence of its squared-singular-value empirical law to the
shifted-Ginibre law.  Push this law forward by \(x\mapsto\sqrt x\), and
then integrate \(\log(s\vee a)\) on a fixed compact singular-value
interval.  Since the integrand is bounded there, convergence
in probability also gives convergence of expectations.  The convergence is
uniform over the block lengths by the usual subsequence contradiction.  For
the upper tail, use
$\log(s\vee a)\1_{\{s>K\}}\le\varepsilon_Ks^2$, where
$\varepsilon_K\downarrow0$, together with the uniform second-moment bound
for the shifted and scaled matrices.  This is an expectation argument and
requires no union bound over the blocks.  The
weighted average over the blocks therefore converges to
$\mathcal L_{r^2,z}(a)$.  Combining this with
\eqref{eq:periodicization-La} proves \eqref{eq:compact-core-cut}.
\end{proof}

\subsection{The sparse branch and the compact core}

We now prove Theorem~\ref{thm:gaussian-profile}.  Write
$X_N=G_{N,f}$ for the matrix in \eqref{eq:gaussian-profile}.  The argument
splits according to whether $W/N$ tends to zero.

First pass to a subsequence on which
\begin{equation}\label{eq:sparse-profile-regime}
 \frac WN\longrightarrow0.
\end{equation}
Fix an integer $R\ge1$ and denote by $h_{R,W}:=\lfloor RW\rfloor$.  This is a
deterministic truncation half-width; it is unrelated to the auxiliary
matrix $H_{M,W}$ in \eqref{eq:auxiliary-ring}.  For all large $N$,
$2h_{R,W}<N$.  Split
\begin{equation}\label{eq:core-tail-split}
 X_N=Y_{R,N}+E_{R,N},
\end{equation}
where $Y_{R,N}$ contains the offsets $|s|\le h_{R,W}$ and $E_{R,N}$ the
remaining offsets.  The two matrices are independent.  Define their row
variance masses by
\begin{equation}\label{eq:core-tail-masses}
 v_{R,N}:=\frac1N\E\|Y_{R,N}\|_{\HS}^2
 =\sum_{|s|\le h_{R,W}}q_s^{(N,W)},\qquad
 t_{R,N}:=1-v_{R,N}.
\end{equation}

Recall that
\(Z_{N,W}:=\sum_{s\in\mathcal D_N}f(s/W)\).  We use the following
Riemann-sum approximation.
\begin{lemma}[Riemann masses]\label{lem:Riemann-masses}
For every fixed $R$, in the $N\to\infty$ limit,
\begin{equation}\label{eq:Riemann-mass-limits}
 \frac{Z_{N,W}}W\longrightarrow1,\qquad
 v_{R,N}\longrightarrow v_R:=\int_{-R}^R f(x)\,dx,\qquad
 t_{R,N}\longrightarrow t_R:=\int_{|x|>R}f(x)\,dx.
\end{equation}
In particular, in the $R\to\infty$ limit, we have $v_R\uparrow1$ and $t_R\downarrow0$.
\end{lemma}

\begin{proof}
For a BV function $f$, the error between a mesh-$W^{-1}$ Riemann sum and the
corresponding integral on an interval is at most
$W^{-1}\TV(f)$, up to the two endpoint cells.  Apply this first on
$[-R,R]$.  In the sparse regime $W=o(N)$ the sampled interval
$[-N/(2W),N/(2W)]$ exhausts $\R$; applying the same estimate on a large
fixed interval and then using the integrable tails of $f$ proves the first
limit.  The other two claims then follow by division.
\end{proof}

We now rescale $Y_{R,N}$ to turn it into a circular Gaussian scalar band with normalized variance. Normalize the core by
\begin{equation}\label{eq:normalized-core}
 r_{R,N}:=\sqrt{v_{R,N}},\qquad r_R:=\sqrt{v_R},\qquad
 \widetilde Y_{R,N}:=r_{R,N}^{-1}Y_{R,N}.
\end{equation}
The row and column variance sums of $\widetilde Y_{R,N}$ equal one.  Its
active weights are
\[
 \widetilde q_{R,s}^{(N,W)}
 =\frac{f(s/W)\1_{\{|s|\le h_{R,W}\}}}
 {\sum_{|u|\le h_{R,W}}f(u/W)}.
\]
Continuity and positivity of $f$ give, for any fixed $R$, constants $0<c_R<C_R<\infty$ such that
\begin{equation}\label{eq:normalized-core-comparability}
 \frac{c_R}{W}\le\widetilde q_{R,s}^{(N,W)}
 \le\frac{C_R}{W},\qquad |s|\le h_{R,W}.
\end{equation}
Thus Proposition~\ref{prop:compact-gaussian-core} applies to the normalized
core $\widetilde Y_{R,N}$.

\subsection{Log potential limit of the center core matrix}

The scalar $r_{R,N}$ still depends on $N$.  We first treat the raw
logarithm without requiring a uniform-in-$r$ compact theorem.  The matrix
\(\widetilde Y_{R,N}\) is almost surely invertible: its determinant,
viewed as a polynomial in the independent diagonal Gaussian variables,
contains their product with nonzero coefficient.  Consequently, for
every \(z\in\C\), the polynomial
\(w\mapsto\det(w\widetilde Y_{R,N}-zI_N)\) is almost surely nonzero,
including when \(z=0\).  For any \(A,z\) for which
\(w\mapsto\det(wA-zI_N)\) is nonzero, define
\begin{equation}\label{eq:phase-average}
 \mathfrak M_{A,z}(r):=
 \frac1{2\pi N}\int_0^{2\pi}
 \log|\det(re^{\mathrm i\theta}A-zI_N)|\,d\theta.
\end{equation}
When \(z=0\) and \(A\) is invertible,
\(\mathfrak M_{A,0}(r)=\log r+N^{-1}\log|\det A|\), so the radial
monotonicity used below remains valid without change.

\begin{lemma}[Jensen radial monotonicity]\label{lem:radial-monotonicity}
For every nonzero complex polynomial $p$, the function
\[
 \rho\longmapsto
 \frac1{2\pi}\int_0^{2\pi}
 \log|p(\rho e^{\mathrm i\theta})|\,d\theta
\]
is nondecreasing on $(0,\infty)$.  In particular, if $p(0)\ne0$, then
\[
 \frac1{2\pi}\int_0^{2\pi}\log|p(e^{\mathrm i\theta})|\,d\theta
 \ge\log|p(0)|.
\]
\end{lemma}

\begin{proof}
Write
$p(w)=c\prod_{j=1}^k(w-\zeta_j),c\ne 0$.  Jensen's formula gives
\[
 \frac1{2\pi}\int_0^{2\pi}\log|p(re^{\mathrm i\theta})|\,d\theta
 =\log|c|+\sum_{j=1}^k\log(r\vee|\zeta_j|),
\]
which is nondecreasing in \(r\).  Letting \(r\downarrow0\) and comparing
with \(r=1\) gives the second claim.
\end{proof}

The law of \(\widetilde Y_{R,N}\) is invariant under multiplication by a
common phase.  The Gaussian logarithmic estimate from
Lemma~\ref{lem:compact-Gaussian-concentration} gives the integrability
needed to exchange expectation and phase integration.  Hence
\begin{equation}\label{eq:phase-law}
 \E L_0(r\widetilde Y_{R,N}-zI_N)
 =\E\mathfrak M_{\widetilde Y_{R,N},z}(r).
\end{equation}
Choose \(r_-\) and \(r_+\) from a fixed countable dense subset of
\((0,\infty)\), with
\(r_-<r_R<r_+\).  Eventually
\(r_-<r_{R,N}<r_+\), and
Lemma~\ref{lem:radial-monotonicity} shows that
\(\mathfrak M_{A,z}(r)\) is nondecreasing in \(r\).  We may therefore
sandwich the
expected raw potential of $Y_{R,N}$ between the corresponding expected
potentials at $r_-$ and $r_+$.  Proposition~\ref{prop:compact-gaussian-core}
and then the countable limits $r_-\uparrow r_R$, $r_+\downarrow r_R$
give, for every fixed \(z\in\C\),
\begin{equation}\label{eq:core-raw-limit}
 \E L_0(Y_{R,N}-zI_N)\longrightarrow U_{v_R}(z).
\end{equation}

For a fixed cutoff $a>0$, the varying normalization is simpler.  The map
$s\mapsto\log(s\vee a)$ is $a^{-1}$-Lipschitz.  Mirsky's inequality and
Cauchy--Schwarz imply
\begin{align}
 &|L_a(r_{R,N}\widetilde Y_{R,N}-zI_N)
 -L_a(r_R\widetilde Y_{R,N}-zI_N)| \notag\\
 &\hspace{25mm}\le
 \frac{|r_{R,N}-r_R|}{a\sqrt N}\|\widetilde Y_{R,N}\|_{\HS}.
 \label{eq:varying-normalization-La}
\end{align}
The expected right-hand side tends to zero because
$N^{-1}\E\|\widetilde Y_{R,N}\|_{\HS}^2=1$.  The fixed-cutoff part of
Proposition~\ref{prop:compact-gaussian-core} therefore yields
\begin{equation}\label{eq:core-cut-limit}
 \E L_a(Y_{R,N}-zI_N)\longrightarrow\mathcal L_{v_R,z}(a),
 \qquad N\to\infty
 \quad(R,a\text{ fixed}).
\end{equation}

\subsection{Removing the tail in the variance-cutoff procedure}
We now bound the expected raw potential
\(\E L_0(X_N-zI_N)\) from above and below.
For the lower comparison, we use only phase invariance and no norm
estimate.  For \(w\in\C\), consider
\[
p_{R,N}(w)=\det(Y_{R,N}+wE_{R,N}-zI_N).
\]
The same diagonal-monomial argument gives
\(p_{R,N}(0)=\det(Y_{R,N}-zI_N)\ne0\) almost surely for every fixed
\(z\in\C\).  Thus \(p_{R,N}\) is almost surely a nonzero polynomial,
also when \(z=0\).
The conditional Gaussian logarithm estimate used in
Lemma~\ref{lem:compact-Gaussian-concentration} also shows that the positive
and negative logarithms below are integrable, uniformly for
$Y_{R,N}+e^{\mathrm i\theta}E_{R,N}$ and for $Y_{R,N}$.  Thus by Fubini's
theorem, we can exchange the expectation and phase integration.
Since
\(E_{R,N}\stackrel{\mathrm d}=e^{\mathrm i\theta}E_{R,N}\) and it is
independent of \(Y_{R,N}\),
Lemma~\ref{lem:radial-monotonicity} gives
\begin{align}
 \E L_0(X_N-zI_N)
 &=\E\frac1{2\pi N}\int_0^{2\pi}
 \log|p_{R,N}(e^{\mathrm i\theta})|\,d\theta \notag\\
 &\ge \E\frac1N\log|p_{R,N}(0)|
 =\E L_0(Y_{R,N}-zI_N).
 \label{eq:tail-Jensen-lower}
\end{align}

For the upper comparison, $L_0\le L_a$, and Mirsky's inequality gives, for
any two $N\times N$ matrices,
\begin{equation}\label{eq:Mirsky-La}
 |L_a(A)-L_a(B)|
 \le\frac1{aN}\sum_{j=1}^N|s_j(A)-s_j(B)|
 \le\frac{\|A-B\|_{\HS}}{a\sqrt N}.
\end{equation}
Consequently,
\begin{align}
 \E L_0(X_N-zI_N)
 &\le\E L_a(Y_{R,N}-zI_N)
 +\frac{\E\|E_{R,N}\|_{\HS}}{a\sqrt N} \notag\\
 &\le\E L_a(Y_{R,N}-zI_N)+\frac{\sqrt{t_{R,N}}}{a}.
 \label{eq:tail-Mirsky-upper}
\end{align}
Taking $N\to\infty$ with $R,a$ fixed in
\eqref{eq:tail-Jensen-lower}--\eqref{eq:tail-Mirsky-upper} gives
\begin{equation}\label{eq:mean-squeeze}
 U_{v_R}(z)
 \le\liminf_N\E L_0(X_N-zI_N)
 \le\limsup_N\E L_0(X_N-zI_N)
 \le\mathcal L_{v_R,z}(a)+\frac{\sqrt{t_R}}a.
\end{equation}

\subsection{The truncated logarithmic potential}
We now show that, for small \(a\) and large \(R\), both sides of
\eqref{eq:mean-squeeze} converge to \(U_{\cir}(z)\).
The cutoff error can be explicitly computed as follows.
Let
\[
 F_{v,z}(t):=\nu_{v,z}([0,t]).
\]
Then
\begin{equation}\label{eq:truncated-potential-identity}
 \mathcal L_{v,z}(a)-U_v(z)
 =\int_0^a\log\!\left(\frac as\right)\,d\nu_{v,z}(s)
 =\int_0^a\frac{F_{v,z}(t)}t\,dt.
\end{equation}
This cutoff measures exactly the singular-value mass below $a$, with logarithmic weight $t^{-1}$.

For fixed $z$, the shifted-Ginibre singular-value law obeys, uniformly for
$v$ in a neighborhood of one,
\begin{equation}\label{eq:hard-edge-linear-mass}
 F_{v,z}(t)\le C_z t\qquad(0<t<a_z).
\end{equation}
Indeed, the bounded-density estimate for the symmetrized limiting
Hermitization law in \cite[proof of Theorem~3.6, (3.15)]{Han2410} gives the
claim for $v=1$.  The scaling identity
$F_{v,z}(t)=F_{1,z/\sqrt v}(t/\sqrt v)$ gives the asserted local uniformity
in $v$.
Equations \eqref{eq:truncated-potential-identity} and
\eqref{eq:hard-edge-linear-mass} give
\begin{equation}\label{eq:truncated-potential-bound}
 0\le\mathcal L_{v,z}(a)-U_v(z)\le C_z a.
\end{equation}

Now take $a=t_R^{1/4}$ in \eqref{eq:mean-squeeze} and then let
$R\to\infty$.  Since
\[
 \frac{\sqrt{t_R}}a=t_R^{1/4}\to0,\qquad
 v_R\to1,\qquad U_{v_R}(z)\to U_{\cir}(z),
\]
we obtain, along the sparse subsequence,
\begin{equation}\label{eq:sparse-mean-limit}
 \E L_0(X_N-zI_N)\longrightarrow U_{\cir}(z),
 \qquad\text{for every fixed }z\in\C.
\end{equation}

We are only one step from the circular law for $X_N$.
Since
\[
 q_0^{(N,W)}=\frac{f(0)}{Z_{N,W}}
 \geq \frac{f(0)}{N\|f\|_\infty},
\]
the hypothesis of Lemma~\ref{lem:compact-Gaussian-concentration} holds
with a fixed constant depending only on \(f\).  Hence, by that lemma,
\begin{equation}\label{eq:Gaussian-logdet-variance}
 \operatorname{Var}\log|\det(X_N-zI_N)|\le C_{f,z}N\log^2(eN),
\end{equation}
and therefore
\begin{equation}\label{eq:Gaussian-logdet-concentration}
 L_0(X_N-zI_N)-\E L_0(X_N-zI_N)\xrightarrow{\Pp}0.
\end{equation}
Together with \eqref{eq:sparse-mean-limit}, this proves the required
logarithmic-potential convergence in the sparse branch.

\subsection{The dense branch and completion}
We have covered the regime \(W/N\to0\) and \(W\to\infty\).  Suppose
instead that, along a subsequence, \(W/N\geq c>0\).  For
\(s\in\mathcal D_N\), the points \(s/W\) lie in a fixed compact interval.
Positivity and continuity of \(f\) therefore give constants
\(0<c_{f,c}<C_{f,c}<\infty\) such that
\begin{equation}\label{eq:dense-profile-comparability}
 c_{f,c}N\le Z_{N,W}\le C_{f,c}N,\qquad
 \frac{c_{f,c}}N\le q_s^{(N,W)}\le\frac{C_{f,c}}N
 \quad(s\in\mathcal D_N).
\end{equation}
The profile is doubly stochastic and its effective bandwidth is comparable
with \(N\).  Thus the proof of
Proposition~\ref{prop:finite-moment-anchor} applies directly, with the
bandwidth parameter taken to be \(N\) and
\(\mathfrak b(X_N)\asymp N\): Theorem~\ref{thm:high-band-lsv},
Proposition~\ref{prop:mesoscopic-counting}, and
Lemma~\ref{lem:local-finite-moment-bulk} give, for every fixed
\(z\in\C\),
\begin{equation}\label{eq:dense-gaussian-logdet}
 L_0(X_N-zI_N)\xrightarrow{\Pp}U_{\cir}(z).
\end{equation}
The bulk inputs here are the fixed-\(z\) estimates
\cite[Proposition~3.4, Corollary~3.5, and Theorem~3.6]{Han2410}; no
spectral-parameter exceptional set is introduced.

Every subsequence contains a further subsequence on which either $W/N\to0$
or $W/N$ is bounded below.  In the sparse branch,
\[
 \frac1N\E\|X_N\|_{\HS}^2=1.
\]
This gives Hilbert--Schmidt tightness, while
\eqref{eq:sparse-mean-limit}--\eqref{eq:Gaussian-logdet-concentration}
identify the logarithmic potential for every fixed \(z\); normalized
Ginibre has the same Hilbert--Schmidt tightness.  Hence the replacement
principle \cite[Theorem~4.1]{Han2410}, which only requires this comparison
for almost every \(z\), completes that branch.  In the dense branch,
\eqref{eq:dense-profile-comparability} gives the same Hilbert--Schmidt
tightness and \eqref{eq:dense-gaussian-logdet} gives the fixed-\(z\)
limit directly, so the same replacement principle applies.  The
subsequence principle therefore proves both claims of
Theorem~\ref{thm:gaussian-profile} for the full sequence.
Taking \(z=0\), whose limiting potential is
\(U_{\cir}(0)=-1/2\), also gives
\(\Pp\{\det G_{N,f}=0\}\to0\), since the normalized log determinant is
\(-\infty\) on the singularity event.

\part{Cyclic full-block bands: discrete and bounded-density regimes}

This part is devoted exclusively to the periodic full-block band matrix
\(X_N\) defined in \eqref{eq:model}.  In block coordinates it is cyclic
block tridiagonal, and each of its three nonzero block diagonals consists
of independent full \(W\times W\) blocks.  We prove
Theorems~\ref{thm:main} and~\ref{thm:density-main}, respectively for the
possibly discrete subgaussian regime and the bounded-density regime.  This
geometry is distinct from the scalar indicator-profile band treated in
Part~I.

\section{Local three-block estimates for the discrete full-block model}
\label{sec:local-interface}

This section isolates the local probabilistic input for
Theorem~\ref{thm:main}.  We reserve three consecutive block sites and
condition on the complementary arc, leaving seven independent random
blocks in a three-block path.  The conditioned complement then appears
only as a deterministic deformation at the two outer sites, and the
local estimate is uniform in that deformation.

We use two descriptions of the same reserved packet.  The coordinate
description produces a determinant polynomial to which the local random
estimate applies.  The transfer description inserts that estimate into
the cyclic determinant, and the determinant identity below connects the
two descriptions.  The principal estimates are stated here and used in
the global argument of Section~\ref{sec:gb-global-proof}; their detailed
proofs are deferred to Section~\ref{supplementsection}.  The bounded-density
branch later reuses the deterministic packet identities with a different
local probabilistic input.  We state the short density-evaluation result
where the common packet polynomial is introduced, while keeping the two
global arguments separate.

\subsection{Setup and normalization}

Throughout this section we work under the hypotheses of
Theorem~\ref{thm:main}; in particular, the atom is real, centered,
variance one, and uniformly subgaussian, but it need not have a density
and may have an atom at zero.  Constants may depend on its subgaussian
constant and on the fixed spectral parameter \(z\).

Within this local section, we write \(A_j,B_j,C_j\) for the normalized
blocks \(\mathsf A_j,\mathsf B_j,\mathsf C_j\) in \eqref{eq:model}.  When
the three packet equations are temporarily multiplied by
\(\sqrt{3W}\), write
\begin{equation}
 \sigma_W:=\sqrt{3W},\qquad \zeta_W:=-\sigma_W z.
 \label{eq:local-zeta}
\end{equation}
In the model normalization, the equation at block site \(j\) is
\begin{equation}
 C_j\psi_{j-1}+(A_j-zI_W)\psi_j+B_j\psi_{j+1}=0.
 \label{eq:local-block-recurrence}
\end{equation}
Here \(A_j\) is the diagonal random block, \(B_j\) is the coefficient of the
vector at the right neighbour, and \(C_j\) is the coefficient of the vector
at the left neighbour.  Thus \(B_j\) and \(C_j\) are defined by the original
block equation; they are not constructed from a transfer matrix.

\subsection{The interface event}
We work on a high-probability event on which all off-diagonal interface
blocks are well behaved.  The precise bounds are as follows.
\begin{lemma}[Interface determinant and inverse bounds]
\label{lem:local-interface-control}
Let \(G=(\xi_{ij})_{i,j\leq W}\) have independent entries with law \(\xi\), and denote by
\(G_{\rm nor}=\sigma_W^{-1}G\).  Then there exist constants \(C_K,c_K>0\), depending only on the subgaussian constant \(K\), such
that
\begin{equation}
 \P\left\{
  \|G_{\rm nor}\|>C_K
  \ \text{or}\ 
  |\det G_{\rm nor}|\notin[e^{-C_KW},e^{C_KW}]
  \ \text{or}\ 
  \|G_{\rm nor}^{-1}\|>e^{C_KW}
 \right\}\leq e^{-c_KW}.
 \label{eq:local-interface-control}
\end{equation}
The inverse norm is \(+\infty\) when $G_{\rm nor}$ is singular.
\end{lemma}
The proof, using subgaussian concentration and Nguyen's overcrowding
theorem \cite{Nguyen}, is given in
Section~\ref{supplementsection}.

\subsection{The reserved three-block packet}
\label{sec:local-terminal}
\subsubsection{Coordinate model for the reserved packet}
From the system of block equations in
\eqref{eq:local-block-recurrence}, we expose all but three consecutive
block sites and reserve these three sites.  By cyclic relabelling, label
the three sites
\begin{equation}
 L=1,\qquad C=2,\qquad R=3,
 \label{eq:local-LCR}
\end{equation}
and write $-:=m$ and $+:=4$ for their immediate outside left and right
neighbours.
The seven unexposed blocks are 
\begin{equation}
 A_L,\ B_L,\ C_C,\ A_C,\ B_C,\ C_R,\ A_R,
 \label{eq:local-seven-blocks}
\end{equation}
where, in terms of the block equation \eqref{eq:local-block-recurrence},
\begin{equation}
 B_L=B_1,\quad C_C=C_2,\quad B_C=B_2,\quad C_R=C_3.
 \label{eq:local-internal-identification}
\end{equation}
The two endpoint blocks, which are denoted by
\begin{equation}
 C_L=C_1,\qquad B_R=B_3
 \label{eq:local-endpoints}
\end{equation}
have already been exposed.  Thus the three equations on the rows $L,C,R$ are
\begin{align}
 C_L\psi_-+(A_L-zI_W)\psi_L+B_L\psi_C&=0,
 \label{eq:local-packet-left}\\
 C_C\psi_L+(A_C-zI_W)\psi_C+B_C\psi_R&=0,
 \label{eq:local-packet-centre}\\
 C_R\psi_C+(A_R-zI_W)\psi_R+B_R\psi_+&=0.
 \label{eq:local-packet-right}
\end{align}

We use the following notation to explicitly mark the location where randomness still remains:

\begin{definition}[Row--column mask]
A row--column mask is a deterministic \(0\)-\(1\) matrix \(M=(m_{ij})\)
which records where fresh random entries are present.  In a matrix
\(Q+M\odot X\), the entry \(x_{ij}\) is used when \(m_{ij}=1\), and is
absent when \(m_{ij}=0\).  Zeros in the mask \(M\) mean only that no
fresh random part is present; the deterministic matrix \(Q\) may still
be nonzero there.
\end{definition}

With both rows and columns ordered as \((L,C,R)\), the packet mask related to \eqref{eq:local-packet-left}-\eqref{eq:local-packet-right} is
\begin{equation}
 M_{\partial}=
 \begin{pmatrix}
  1&1&0\\
  1&1&1\\
  0&1&1
 \end{pmatrix},
 \label{eq:local-mask}
\end{equation}
where each \(1\) denotes a complete \(W\times W\) random block.  Equivalently,
the only missing rectangles are \(L_{\rm row}\times R_{\rm col}\) and
\(R_{\rm row}\times L_{\rm col}\).

Define the scalar support set
\begin{equation}
 \mathcal I_{\partial}:=
 \{(i,j):\text{the block containing }(i,j)
                 \text{ is marked }1\text{ in }M_\partial\}.
 \label{eq:local-mask-support}
\end{equation}
Let \(\mathbf x=(x_e)_{e\in\mathcal I_{\partial}}\) be the vector of all
\(7W^2\) formal scalar variables, and let
\(\Xi=(\xi_e)_{e\in\mathcal I_{\partial}}\) denote their random evaluation by
independent copies of \(\xi\).  For every marked block \(M\), the notation
\(M(\mathbf x)\) means the normalized \(W\times W\) matrix whose
corresponding entries are \(\sigma_W^{-1}x_e\).\footnote{In the model of
Theorem~\ref{thm:main} every weight is
one.  The same argument permits deterministic positive
weights \(a_e\) with \(0<c_*\leq a_e\leq C_*\); the corresponding
normalized block-matrix entry is \(\sigma_W^{-1}a_e x_e\).}  

We denote by $\Delta(\mathbf x)$ the matrix preserving only the fresh randomness from \eqref{eq:local-packet-left}-\eqref{eq:local-packet-right}: 
\begin{equation}
 \Delta(\mathbf x)=
 \begin{pmatrix}
  A_L(\mathbf x)&B_L(\mathbf x)&0\\
  C_C(\mathbf x)&A_C(\mathbf x)&B_C(\mathbf x)\\
  0&C_R(\mathbf x)&A_R(\mathbf x)
 \end{pmatrix}.
 \label{eq:local-Delta}
\end{equation}

We denote by \(O=L\oplus R\) (the outer space).  For
\(Q_O\in\C^{2W\times2W}\), written in \(W\times W\) blocks as
\begin{equation}
 Q_O=
 \begin{pmatrix}Q_{LL}&Q_{LR}\\Q_{RL}&Q_{RR}\end{pmatrix},
 \label{eq:local-QO-blocks}
\end{equation}
define the embedding into the full \((L,C,R)\) space by
\begin{equation}
 \operatorname{Emb}_O(Q_O):=
 \begin{pmatrix}
  Q_{LL}&0&Q_{LR}\\
  0&0&0\\
  Q_{RL}&0&Q_{RR}
 \end{pmatrix}.
 \label{eq:local-Emb}
\end{equation}
Thus \(\operatorname{Emb}_O(Q_O)(v_L,v_C,v_R)\) equals
\((Q_{LL}v_L+Q_{LR}v_R,0,Q_{RL}v_L+Q_{RR}v_R)\).

\subsubsection{Conditioning on the complementary arc}
We isolate the seven fresh packet blocks from the randomness on the
complementary arc.
For an arbitrary, possibly very large, matrix
\(Q_O\) of size $2W$ (encoding the effect of the complementary arc on the
outer coordinates), define the terminal matrix, determinant polynomial, and
coefficient norm by
\begin{align}
 H(Q_O;\mathbf x)
 &:=\Delta(\mathbf x)-zI_{3W}+\operatorname{Emb}_O(Q_O),
 \label{eq:local-terminal-matrix}\\
 p_{Q_O}(\mathbf x)&:=\det H(Q_O;\mathbf x),
 \label{eq:local-terminal-polynomial}\\
 \mathfrak C(Q_O)
 &:=\bigl\|\Coeff_{\mathbf x} p_{Q_O}\bigr\|_{\ell^2}.
 \label{eq:local-terminal-coefficients}
\end{align}
Here  $\mathfrak C(Q_O)$ records the $\ell^2$ norm of the coefficients of the polynomial $p_{Q_O}$.

\subsection{Small-ball estimate for the packet polynomial}
The main probabilistic estimate we use is the following small ball estimate for $p_{Q_O}$. The most striking feature is that it holds for any $Q_O$ without any specific operator norm bound on $Q_O$.
\begin{proposition}[Three-block terminal estimate]
\label{prop:local-terminal}
Assume \(|\zeta_W|=\sqrt{3W}|z|\leq W^{K_z}\) for a fixed \(K_z\).  Then uniformly in
\(Q_O\in\C^{2W\times2W}\), for every \(T>0\),
\begin{equation}
 \E\min\left\{T,
  \log_+\frac{\mathfrak C(Q_O)}
                   {|p_{Q_O}(\Xi)|}\right\}
 \leq CW\log W+p_WT,
 \qquad
 p_W\leq C\sqrt{\frac{\log W}{W}}+e^{-cW}.
 \label{eq:local-terminal-capped}
\end{equation}
The capped loss is \(T\) at a zero of the determinant.  In particular,
\begin{equation}
 \P\{p_{Q_O}(\Xi)=0\}\leq p_W.
 \label{eq:local-terminal-zero}
\end{equation}
On the event
\begin{equation}
 \mathcal E_{\max}:=\left\{\max_{e\in\mathcal I_{\partial}}|\xi_e|
                 \leq C_0\sqrt W\right\},
 \qquad \P(\mathcal E_{\max}^c)\leq e^{-cW},
 \label{eq:local-max-event}
\end{equation}
where \(C_0=C_0(K)\) is sufficiently large,
one also has
\begin{equation}
 \log_+\frac{|p_{Q_O}(\Xi)|}
                  {\mathfrak C(Q_O)}
 \leq CW\log W.
 \label{eq:local-terminal-reverse}
\end{equation}
All assertions remain valid conditionally when \(Q_O\) is measurable with
respect to a sigma-field independent of the seven packet blocks.
\end{proposition}

The proof, combining Lemma~\ref{lem:local-cook-input} with a
linear-algebraic extraction of the large part of $Q_O$, is given in
Section~\ref{supplementsection}.

\begin{proposition}[Conditional evaluation of the packet polynomial]
\label{prop:dens-packet-coeff-evaluation}
Assume the packet atoms satisfy the real-or-complex bounded-density
hypothesis of Theorem~\ref{thm:density-main}.
Fix \(\Theta\in\GL_{2W}(\C)\) and fixed invertible endpoint blocks
\[
 E_\partial=\operatorname{diag}(C_L,B_R).
\]
Let \(K_{\Theta,E_\partial}(\mathbf x)\) denote the \(5W\times5W\) matrix
\(K_\Theta(\mathbf x)\) in \eqref{eq:local-KTheta}, with the two endpoint
blocks fixed by \(E_\partial\), and set
\[
 P_{\Theta,E_\partial}(\mathbf x)
 :=\det K_{\Theta,E_\partial}(\mathbf x),
 \qquad
 \mathscr C(\Theta;E_\partial)
 :=\|\Coeff_{\mathbf x}P_{\Theta,E_\partial}\|_2.
\]
Identity~\eqref{eq:local-K-first-elimination} identifies this determinant
with the boundary polynomial \(\mathscr D_\Theta\) in
\eqref{eq:local-DTheta}, with the endpoints fixed by \(E_\partial\).
If \(\Xi_7\) denotes the random evaluation of the seven blocks
\[
 A_L,B_L,C_C,A_C,B_C,C_R,A_R,
\]
then
\begin{equation}
 \E_{\Xi_7}\left|
  \log|P_{\Theta,E_\partial}(\Xi_7)|
  -\log\mathscr C(\Theta;E_\partial)
 \right|
 \leq C_{\rm dens}W\log(eW),
 \label{eq:dens-pathwise-boundary-volume}
\end{equation}
where \(C_{\rm dens}\) depends only on the density bound and the fixed model
parameters.
\end{proposition}

\begin{proof}
Divide the \(7W^2\) variables into the \(3W\) packet-row groups
\begin{align*}
 \mathbf x_{L,a}
 &:=
 \bigl((A_L)_{ab},(B_L)_{ab}\bigr)_{b=1}^W,\\
 \mathbf x_{C,a}
 &:=
 \bigl((C_C)_{ab},(A_C)_{ab},(B_C)_{ab}\bigr)_{b=1}^W,\\
 \mathbf x_{R,a}
 &:=
 \bigl((C_R)_{ab},(A_R)_{ab}\bigr)_{b=1}^W,
 \qquad 1\leq a\leq W.
\end{align*}
Each group contains at most \(3W\) independent atoms.

By definition,
\(P_{\Theta,E_\partial}=\det K_{\Theta,E_\partial}\).
Every group \(\mathbf x_{\alpha,a}\) occurs only in the \(a\)-th scalar
row of the corresponding packet block row of
\(K_{\Theta,E_\partial}\).  Since the determinant is affine-linear in
each matrix row, \(P_{\Theta,E_\partial}\) is affine separately in every
one of the \(3W\) groups.

The pointwise comparison \eqref{eq:dens-seam-pathwise} shows that
\(\mathscr C(\Theta;E_\partial)>0\), so the polynomial is nonzero.  Applying
Corollary~\ref{cor:dens-coefficient-evaluation} with
\(n_{\rm grp}=3W\) and \(p\leq3W\) gives
\[
 \E_{\Xi_7}\left|
  \log|P_{\Theta,E_\partial}(\Xi_7)|
  -\log\mathscr C(\Theta;E_\partial)
 \right|
 \leq C_{\rm dens}(3W)\log(3eW),
\]
which is \eqref{eq:dens-pathwise-boundary-volume} after changing the
constant.
\end{proof}

The following lemma makes the coefficient norm $\mathfrak C(Q)$ concrete.
\begin{lemma}[All-minor coefficient estimate]
\label{lem:local-all-minor}
For the normalized terminal polynomial and coefficient norm defined in
\eqref{eq:local-terminal-matrix}--\eqref{eq:local-terminal-coefficients}, uniformly in
\(Q\in\C^{2W\times2W}\),
\begin{equation}
 e^{-CW\log W}\det(I_{2W}+Q^*Q)^{1/2}
 \leq\mathfrak C(Q)\leq
 e^{CW\log W}\det(I_{2W}+Q^*Q)^{1/2}.
 \label{eq:local-all-minor}
\end{equation}
\end{lemma}
The proof is given in Section~\ref{supplementsection}.

\subsection{Identifying the deformation via Schur complement of the outside chain}
\label{sec:local-open-complement}

We now derive the deformation matrix \(Q_O\) and its embedding
\eqref{eq:local-Emb} directly from the cyclic matrix.
Let the cyclic block indices be \(1,\dots,m\), with packet
\(\partial=\{1,2,3\}\), and let, recalling $m=\frac{N}{W}$,
\begin{equation}
 I_{\rm blk}:=\{4,5,\dots,m\},\qquad n_I:=(m-3)W.
 \label{eq:local-Iblk}
\end{equation}
Thus \(I_{\rm blk}\) is a set of block indices, while \(n_I\) is the number
of scalar coordinates in those blocks.

Let \(H_N(z):=X_N-zI_N\) denote the normalized cyclic matrix associated with
\eqref{eq:local-block-recurrence}.  After ordering the three packet coordinates
to the left of the other outside coordinates, we write
\begin{equation}
 H_N(z)=
 \begin{pmatrix}
  D_{\partial}(z)&E_{\partial I}\\
  E_{I\partial}&A_I(z)
 \end{pmatrix}.
 \label{eq:local-global-split}
\end{equation}
This formula defines every block in the display as the corresponding
submatrix of \(H_N(z)\).  Explicitly, the lower-right block is the
open chain
\begin{equation}
 A_I(z)=
 \begin{pmatrix}
  A_4-zI_W&B_4&&0\\
  C_5&A_5-zI_W&\ddots&\\
  &\ddots&\ddots&B_{m-1}\\
  0&&C_m&A_m-zI_W
 \end{pmatrix}
 \in\C^{n_I\times n_I}.
 \label{eq:local-open-chain}
\end{equation}
It is open because the two cyclic connections crossing the cuts
\(3\leftrightarrow4\) and \(m\leftrightarrow1\) occur in
\(E_{\partial I},E_{I\partial}\), not inside \(A_I(z)\).

Whenever \(A_I(z)\) is invertible, eliminating the outside coordinates
gives the Schur complement
\begin{equation}
 S_{\partial}(z)
 :=D_{\partial}(z)-E_{\partial I}A_I(z)^{-1}E_{I\partial}.
 \label{eq:local-packet-schur}
\end{equation}
The outside chain touches the retained packet only at positions \(L\) and
\(R\), not at the center $C$.  Therefore
the Schur correction in \eqref{eq:local-packet-schur} is supported on
\((L\oplus R)\times(L\oplus R)\), and the central block is untouched.
Thus there is an outside-measurable random matrix \(Q_O(z)\) such that
\begin{equation}
 S_{\partial}(z)=
 \Delta(\Xi)-zI_{3W}+\operatorname{Emb}_O(Q_O(z)).
 \label{eq:local-actual-terminal}
\end{equation}
The seven blocks in \(\Delta\) remain independent of the perturbation \(Q_O(z)\).

For $z$ outside a deterministic planar-null set, $A_I(z)$ is invertible
almost surely.  Indeed, for every realization of its entries,
\(\det A_I(z)\) is a polynomial in
the spectral parameter \(z\), of degree \(n_I\) and with leading
term \((-z)^{n_I}\).  It is not the zero polynomial and has only
finitely many roots.  Fubini's theorem therefore gives a
deterministic planar-null set \(\mathcal Z_I\) such that
\begin{equation}
 \P\{\det A_I(z)=0\}=0\qquad(z\notin\mathcal Z_I).
 \label{eq:local-open-ae}
\end{equation}
A countable union handles all matrix sizes and all retained packets used in
the proof.  No quantitative estimate for \(A_I(z)^{-1}\) is needed:
Proposition~\ref{prop:local-terminal} is uniform in the resulting
\(Q_O(z)\).

This Schur-complement paragraph only describes a convenient coordinate
chart; it does not restrict the spectral parameters in the main proof.
For every fixed \(z\), the proof below uses the cleared transfer identity
of Lemma~\ref{lem:local-block-floquet}, first on the open set of
invertible interface blocks and then everywhere by polynomial
continuation.  Thus neither \(A_I(z)^{-1}\) nor the exceptional set
\(\mathcal Z_I\) enters the fixed-\(z\) log-determinant argument.

\subsection{Packet transfer and the wedge expansion}
\label{sec:local-boundary}

We now formulate an exterior algebra version of the determinant, via the transfer matrices.

\subsubsection{Transfer through the retained three-block packet}

We denote by
\begin{equation}
 D_L:=A_L-zI_W,\qquad
 D_C:=A_C-zI_W,\qquad
 D_R:=A_R-zI_W.
 \label{eq:local-DLCR}
\end{equation}
On the algebraic set where \(B_L,B_C,B_R\) are invertible, define transfer operators
\begin{align}
 T_L&:=
 \begin{pmatrix}-B_L^{-1}D_L&-B_L^{-1}C_L\\I_W&0\end{pmatrix},
 &
 T_C&:=
 \begin{pmatrix}-B_C^{-1}D_C&-B_C^{-1}C_C\\I_W&0\end{pmatrix},
 \notag\\
 T_R&:=
 \begin{pmatrix}-B_R^{-1}D_R&-B_R^{-1}C_R\\I_W&0\end{pmatrix}.
 \label{eq:local-three-companions}
\end{align}
Recall the local coordinates
\(\psi_-,\psi_L,\psi_C,\psi_R,\psi_+\).  Solving
\eqref{eq:local-packet-left}--\eqref{eq:local-packet-right} in
order gives us
\begin{align*}
 \binom{\psi_C}{\psi_L}&=T_L\binom{\psi_L}{\psi_-},\\
 \binom{\psi_R}{\psi_C}&=T_C\binom{\psi_C}{\psi_L},\\
 \binom{\psi_+}{\psi_R}&=T_R\binom{\psi_R}{\psi_C}.
\end{align*}
Hence the three-step packet transfer is denoted as follows:
\begin{equation}
 \boxed{R_{\partial}:=T_RT_CT_L},\qquad
 \binom{\psi_+}{\psi_R}
 =R_{\partial}\binom{\psi_L}{\psi_-}.
 \label{eq:local-Rpartial}
\end{equation}
This formula also makes the earlier interface notation explicit:
the three right-interface matrices used in the transfer are
\(B_L=B_1\), \(B_C=B_2\), and \(B_R=B_3\).  They are the original
right-interface blocks entering the definition of \(R_{\partial}\), not
blocks extracted from \(R_{\partial}\).

\subsubsection{The actual outside relation and a general test relation}

We do the same for the outside arc.  Assume that its right-interface
blocks \(B_j\) are invertible.  Then define
\begin{equation}
 T_j:=
 \begin{pmatrix}
  -B_j^{-1}(A_j-zI_W)&-B_j^{-1}C_j\\I_W&0
 \end{pmatrix},\qquad 4\leq j\leq m.
\end{equation}
Starting at \((\psi_+,\psi_R)=(\psi_4,\psi_3)\) and propagating through
blocks \(4,\dots,m\) gives the closed form
\begin{equation}
 R_{\mathrm{out}}:=T_mT_{m-1}\cdots T_4,\qquad
 \binom{\psi_L}{\psi_-}
 =R_{\mathrm{out}}\binom{\psi_+}{\psi_R},
 \label{eq:local-Rout}
\end{equation}
because periodicity identifies \(\psi_{m+1}=\psi_1=\psi_L\) and
\(\psi_m=\psi_-\).  Thus \(R_{\mathrm{out}}\) is the boundary relation
arising from the actual complementary arc.  Conditional on the outside
blocks it is deterministic relative to the seven fresh packet blocks.  On the simultaneous
interface-regularity event, every \(B_j\) and \(C_j\) on this arc is invertible,
so each \(T_j\), and hence \(R_{\mathrm{out}}\), belongs to
\(\GL_{2W}(\C)\).  Outside that event we use the Schur-complement and
polynomial formulations, not companion coordinates.
\subsubsection{Local and global theorems, arbitrary frames and determinant computation}
We separate the outside randomness from the retained packet.  The
coordinate local theorem, Proposition~\ref{prop:local-terminal}, is uniform
in the deformation $Q_O$.  The boundary-polynomial conversion below yields
Corollary~\ref{corollary885}, uniformly in a deterministic matrix
\(\Theta\in\GL_{2W}(\C)\), representing the outside relation
\begin{equation}
 \boxed{
 \binom{\psi_L}{\psi_-}
 =\Theta\binom{\psi_+}{\psi_R}.}
 \label{eq:local-general-Theta}
\end{equation}
Only after this conversion do we substitute
\(\Theta=R_{\mathrm{out}}\) for the actual cyclic matrix.  A different,
artificial substitution will be used in Section~\ref{sec:local-penalty}.

Combining \eqref{eq:local-Rpartial} and
\eqref{eq:local-general-Theta}, a nonzero compatible state exists precisely
when
\begin{equation}
 (I_{2W}-\Theta R_{\partial})\binom{\psi_L}{\psi_-}=0.
 \label{eq:local-compatibility}
\end{equation}

The circular law proof further
needs the corresponding equality of determinants.  We use the following determinant computation result:

\begin{lemma}[Block Floquet determinant identity]
\label{lem:local-block-floquet}

For the packet \(\{1,2,3\}=\{L,C,R\}\), define
\begin{equation}
 b_{\partial}:=\det B_L\det B_C\det B_R,
 \qquad
 c_{\rm out}:=\prod_{j=4}^{m}\det B_j.
 \label{eq:local-floquet-clearing-factors}
\end{equation}
Then there is a deterministic sign
\(\varepsilon_{m,W}\in\{1,-1\}\), depending only on the chosen ordering
of block rows and columns, such that
\begin{equation}
 \boxed{\det H_{N}(z)
 =\varepsilon_{m,W}\,c_{\rm out}\,b_{\partial}
   \det(I_{2W}-R_{\rm out}R_{\partial}).}
 \label{eq:local-floquet-packet-split}
\end{equation}
The formula is interpreted by polynomial continuation when
a block $B_k$ is singular.
\end{lemma}
The proof is given in Section~\ref{supplementsection}.

Thus the actual cyclic closure is obtained by the substitution
\(\Theta=R_{\rm out}\).  In particular,
\(b_{\partial}\det(I_{2W}-R_{\rm out}R_{\partial})\) is the packet
boundary polynomial defined below, evaluated at the actual outside
relation, and \(c_{\rm out}\) supplies the remaining interface
determinants.  The arbitrary matrix \(\Theta\) is retained only so that the
local estimate can be proved uniformly before this substitution.

\subsubsection{From determinants to wedge powers: coefficient norms}
We write the clearing factor $b_\partial$ in its fresh variables; \(B_R\)
is exposed, whereas \(B_L,B_C\) remain fresh:
\begin{equation}
 b(\mathbf x):=\det B_L(\mathbf x)\det B_C(\mathbf x)\det B_R.
 \label{eq:local-b}
\end{equation}
The factor \(\det B_R\) is conditioned.  For \(0\leq k\leq 2W\), initially
on the set where \(B_L,B_C,B_R\) are invertible, define
\begin{equation}
 \mathcal Q^{(k)}(\mathbf x)
 :=b(\mathbf x)\,\bigwedge^kR_{\partial}(\mathbf x).
 \label{eq:local-cleared-compound}
\end{equation}
Clearing the denominators in the compound matrices gives a unique polynomial
extension of this expression through singular \(B_L,B_C,B_R\).

For a general boundary relation \(\Theta\in\GL_{2W}(\C)\), we define the
full-exterior boundary polynomial and its coefficient norm via the following notations:
\begin{align}
 \mathscr D_{\Theta}(\mathbf x)
 &:=\sum_{k=0}^{2W}(-1)^k
 \tr\!\left(\mathcal Q^{(k)}(\mathbf x)\bigwedge^k\Theta\right),
 \label{eq:local-DTheta}\\
 \mathscr C(\Theta)
 &:=\bigl\|\Coeff_{\mathbf x}\mathscr D_{\Theta}\bigr\|_{\ell^2}.
 \label{eq:local-CTheta}
\end{align}
The exposed endpoint blocks \(E_\partial=\operatorname{diag}(C_L,B_R)\)
are fixed parameters here; the notation \(\mathscr D_\Theta\) and
\(\mathscr C(\Theta)\) suppresses this dependence.
On the set where \(B_L,B_C,B_R\) are invertible, the standard exterior
identity gives
\begin{equation}
 \mathscr D_{\Theta}(\mathbf x)
 =b(\mathbf x)\det(I_{2W}-\Theta R_{\partial}(\mathbf x))
 =b(\mathbf x)\det(I_{2W}-R_{\partial}(\mathbf x)\Theta).
 \label{eq:local-DTheta-det}
\end{equation}
The second equality is Sylvester's determinant identity.

\subsubsection{Polynomial coefficient bound for general profiles}
We prove a two-sided estimate for $\mathscr C(\Theta)$ in terms of
$\Theta$.  For \(A>0\), use the convention
\begin{center}
\textbf{write \(X=e^{\pm A}Y\) for the pair of inequalities
\(e^{-A}Y\leq X\leq e^AY\).}
\end{center}

Let \(\mathcal E_{\partial}\) be the event that the two exposed endpoint
matrices \(C_L\) and \(B_R\) satisfy
Lemma~\ref{lem:local-interface-control}.  Since
\(E=\operatorname{diag}(C_L,B_R)\), on \(\mathcal E_{\partial}\),
\begin{equation}
 \|E\|\leq W^{O(1)},\qquad |\det E|^{-1}\leq e^{CW\log W}.
 \label{eq:local-E-control}
\end{equation}
\begin{lemma}
\label{lem:local-boundary-volume}
On the event \(\mathcal E_{\partial}\), there is \(C>0\), depending only
on \(K,K_z\) and, in the weighted variant, on \(c_*,C_*\), such that,
uniformly for every
\(\Theta\in\GL_{2W}(\C)\),
\begin{equation}
 \boxed{\;
 \mathscr C(\Theta)
 =e^{\pm CW\log W}\,
   \det(I_{2W}+\Theta^*\Theta)^{1/2}
 =e^{\pm CW\log W}
   \prod_{j=1}^{2W}\sqrt{1+s_j(\Theta)^2}.}
 \label{eq:local-coefficient-main}
\end{equation}
\end{lemma}
The proof is given in Section~\ref{supplementsection}.

Finally, we use the following simple linear algebra comparison:
\begin{lemma}
    Let $R$ be a \(2W\times2W\) complex matrix.  Then
    \begin{equation}\label{eq:local-boundary-volume-vs-exterior}
\max_{0\le r\le 2W}\|\wedge^rR\|\le \det(I_{2W}+R^*R)^{1/2}\le 2^W\max_{0\le r\le 2W}\|\wedge^rR\|.
    \end{equation}
\end{lemma}
\begin{proof}
Let \(s_1\geq s_2\geq\cdots\geq s_{2W}\geq0\) be the singular values
of \(R\).  The maximum on the left is
\(\prod_{i=1}^{2W}\max\{1,s_i\}\), while the middle term is
\(\prod_{i=1}^{2W}\sqrt{1+s_i^2}\).  The claim follows from
\(\sqrt{1+s^2}\leq\sqrt2\max\{1,s\}\).
\end{proof}

Thus the coefficient norm is, up to \(e^{O(W\log W)}\), exactly the
largest exterior growth of the boundary relation $\Theta$.  This is the link used
later in the proof of Theorem \ref{thm:local-complex-frame}.

\subsubsection{Boundary determinant estimates for every invertible
boundary relation}
We also need a small-ball estimate for the evaluation of this polynomial.
We work on the event \(\mathcal E_{\partial}\), where both endpoint
blocks \(C_L,B_R\) obey the required bounds.  The expectation is over the seven fresh
packet blocks.

Define, for any $T>0$, the following bounded function
\begin{equation}
 \mathcal L_T(c,w):=\min\{T,\log_+(c/|w|)\},
 \qquad \mathcal L_T(c,0):=T,
 \label{eq:local-LT}
\end{equation}
which is bounded and continuous on \((0,\infty)\times\C\).

\begin{corollary}\label{corollary885}
On \(\mathcal E_\partial\), uniformly over all
\(\Theta\in\GL_{2W}(\C)\), with $p_W$ defined in \eqref{eq:local-terminal-capped},
\begin{align}
 \E_\Xi\mathcal L_T\bigl(\mathscr C(\Theta),
        \mathscr D_{\Theta}(\Xi)\bigr)
 &\leq CW\log W+p_WT,
 \label{eq:local-boundary-capped}\\
 \boldsymbol 1_{\mathcal E_{\max}}
 \log_+\frac{|\mathscr D_{\Theta}(\Xi)|}
                  {\mathscr C(\Theta)}
 &\leq CW\log W.
 \label{eq:local-boundary-reverse}
\end{align}
\end{corollary}
The proof, including the approximation from the initial block-invertibility
domain to arbitrary
\(\Theta\in\GL_{2W}(\C)\), is given in
Section~\ref{supplementsection}.

\subsection{Small ball estimates for arbitrary frames}
For further use, we need the following arbitrary-frame comparison
between a coefficient norm and a polynomial evaluation.
Throughout this subsection, Hermitian inner products are linear in their
second argument.
Fix \(0\leq r\leq 2W\), and let
\(U,V:\C^r\to\C^{2W}\) be complex isometries.  Define
\begin{equation}
 \widehat U:=Ue_1\wedge\cdots\wedge Ue_r,\qquad
 \widehat V:=Ve_1\wedge\cdots\wedge Ve_r,
 \label{eq:local-UhatVhat}
\end{equation}
with the usual scalar convention when \(r=0\).  Define the polynomial
and its coefficient norm by
\begin{align}
 Z_{U,V}^{(r)}(\mathbf x)
 &:=\left\langle
      \widehat U,\mathcal Q^{(r)}(\mathbf x)\widehat V
    \right\rangle,
 \label{eq:local-ZUV}\\
 \Gamma_{U,V}^{(r)}
 &:=\left\|
      \Coeff_{\mathbf x}Z_{U,V}^{(r)}
    \right\|_{\ell^2}.
 \label{eq:local-GammaUV}
\end{align}
Here \(\mathbf x=(x_e)_{e\in\mathcal I_\partial}\) consists of the
\(7W^2\) formal scalar variables in the seven fresh packet blocks, and
\(\Coeff_{\mathbf x}\) denotes the coefficient vector in the monomial
basis in these variables.
\begin{theorem}[Subgaussian complex-frame extraction]
\label{thm:local-complex-frame}
On \(\mathcal E_{\partial}\), uniformly over
\(0\leq r\leq 2W\) and all complex isometries
\(U,V:\C^r\to\C^{2W}\),
\begin{equation}
 e^{-CW\log W}\leq\Gamma_{U,V}^{(r)}\leq e^{CW\log W}.
 \label{eq:local-frame-coefficients}
\end{equation}
For every \(T>0\),
\begin{equation}
 \E\min\left\{T,
  \log_+\frac{\Gamma_{U,V}^{(r)}}{|Z_{U,V}^{(r)}(\Xi)|}\right\}
 \leq CW\log W+p_WT.
 \label{eq:local-frame-capped}
\end{equation}
The expectation is over the seven fresh random packet blocks in $\Xi$. 
\begin{comment}
Moreover,
\begin{equation}
 \P\{Z_{U,V}^{(r)}(\Xi)=0\}\leq p_W,
 \label{eq:local-frame-zero}
\end{equation}
and, on \(\mathcal E_{\max}\),
\begin{equation}
 \log_+\frac{|Z_{U,V}^{(r)}(\Xi)|}{\Gamma_{U,V}^{(r)}}
 \leq CW\log W.
 \label{eq:local-frame-reverse}
\end{equation}
\end{comment}
\end{theorem}
The proof is given in Section~\ref{supplementsection}.

\section{Global identification of log-determinant limit}
\label{sec:gb-global-proof}

The main purpose of this section is to prove the subgaussian periodic
full-block circular-law result, Theorem~\ref{thm:main}, by combining the
local estimates from Section~\ref{sec:local-interface} with the global
pressure-lifting argument.  The bounded-density branch will later reuse
the same deterministic identification with different probabilistic input.

We use different symbols for the transfers that play different logical
roles:
\begin{center}
\begin{tabular}{c|p{0.73\textwidth}}
symbol & meaning \\ \hline
\(T_j(z)\) & the one-site companion transfer at site \(j\),\\
\(M_{\rm cyc}(z)\) & the full cyclic monodromy \(T_m(z)\cdots T_1(z)\),\\
\(R_{\partial}(z)\) & the transfer through one specified three-site packet,\\
\(R_{\rm out}(z)\) & the actual transfer through the complementary part (outside the three-site packet)\\
\(\Theta\) & an arbitrary deterministic boundary relation in the local theorem,\\
\(\Theta_{\lambda}^{(r;U,V)}\) &
an artificial boundary relation used only for exterior-degree extraction.
\end{tabular}
\end{center}

Suppose that the retained packet consists of sites \(a,a+1,a+2\).  With
\[
 v_{\rm in}:=\binom{\psi_a}{\psi_{a-1}},
 \qquad
 v_{\rm out}:=\binom{\psi_{a+3}}{\psi_{a+2}},
\]
define
\begin{equation}
 R_{\partial}:=T_{a+2}T_{a+1}T_a,
 \qquad
 v_{\rm out}=R_{\partial}v_{\rm in}.
 \label{eq:gb-packet-transfer-definition}
\end{equation}
Starting from \(v_{\rm out}\), propagation along the complementary arc
returns to the packet entrance.  The corresponding product is denoted by
\(R_{\rm out}\), so
\begin{equation}
 v_{\rm in}=R_{\rm out}v_{\rm out}.
 \label{eq:gb-outside-transfer-definition}
\end{equation}
Thus the actual cyclic closure condition at this cut is
\begin{equation}
 \det(I_{2W}-R_{\rm out}R_{\partial})=0.
 \label{eq:gb-cut-closure}
\end{equation}

The local terminal theorem is stated for the abstract relation
\begin{equation}
 v_{\rm in}=\Theta v_{\rm out},
 \qquad \Theta\in\GL_{2W}(\C).
 \label{eq:gb-abstract-boundary}
\end{equation}
To apply it to the cyclic matrix, we first expose and condition on the
complementary arc.  Relative to the still-fresh packet, \(R_{\rm out}\)
is then deterministic, and only at that point do we substitute
\begin{equation}
 \boxed{\Theta=R_{\rm out}.}
 \label{eq:gb-actual-theta-substitution}
\end{equation}

\subsection{Imported inputs}

We retain the model and notation of Section~\ref{sec:model}.  In
particular, \(\mathsf A_j=A_j/\sqrt{3W}\),
\(\mathsf B_j=B_j/\sqrt{3W}\), and
\(\mathsf C_j=C_j/\sqrt{3W}\).

We use the following results proved in the local section.

\begin{enumerate}
\item[(L1)] \emph{Interface determinant and inverse control.}
Lemma~\ref{lem:local-interface-control}, together with the standard
subgaussian norm bound for a diagonal block, supplies an event
\(\mathcal G_j\) such that
\begin{equation}
 \Prob(\mathcal G_j^c)\leq e^{-c_KW}.
 \label{eq:gb-one-site-failure}
\end{equation}
On \(\mathcal G_j\), the normalized blocks satisfy
\begin{equation}
 \|\mathsf A_j\|+\|\mathsf B_j\|+\|\mathsf C_j\|\leq C_K,
 \qquad
 e^{-C_KW}\leq|\det\mathsf B_j|,|\det\mathsf C_j|
 \leq e^{C_KW},
 \label{eq:gb-interface-det-control}
\end{equation}
and
\begin{equation}
 \|\mathsf B_j^{-1}\|+\|\mathsf C_j^{-1}\|
 \leq e^{C_KW}.
 \label{eq:gb-interface-inverse-control}
\end{equation}
The harmless factor \(3^{-1/2}\) between \(W^{-1/2}A_j\) and
\(\mathsf A_j\) is absorbed into the constants.

\item[(L2)] \emph{Three-block terminal and fixed-degree estimates.}
Proposition~\ref{prop:local-terminal},
Lemma~\ref{lem:local-boundary-volume}, and
Theorem~\ref{thm:local-complex-frame} give
\begin{equation}
 p_W\leq C_K\sqrt{\frac{\log W}{W}}+e^{-c_KW}
 \label{eq:gb-pW}
\end{equation}
and a coefficient-relative capped logarithm with base loss
\(C_KW\log W\).  These results are uniform in the predictable complex
outside deformation and in the two complex Grassmann frames.  Their
reverse evaluation estimate is valid on
\begin{equation}
 \mathcal M_{\rm fresh}:=
 \left\{\max_{\text{fresh scalar entries}}|\xi|\leq C_0\sqrt W\right\},
 \qquad
 \Prob(\mathcal M_{\rm fresh}^c)\leq e^{-c_KW},
 \label{eq:gb-fresh-max}
\end{equation}
for a sufficiently large \(C_0=C_0(K)\).  Thus the reverse estimate is
high probability, not deterministic, for an unbounded atom.

\item[(L3)] \emph{Self-contained high-band anchor.}
Proposition~\ref{prop:subgaussian-block-high-band} proves
\eqref{eq:gb-high-band-logdet} for the cyclic full-block model whenever
\(W\ge N^{8/9+\omega}\), with any fixed \(\omega\in(0,1/9)\), and for
every fixed \(z\in\C\).  Its hard-edge step uses the repaired
subgaussian least-singular-value input
Proposition~\ref{prop:jjlo-block-lsv} exactly once; its
upper-edge step uses only a Hilbert--Schmidt tail estimate, not an
operator-norm bound.  Thus (L3) is independent of the circular-law proof
in that paper and of any earlier high-band circular-law theorem.  It is
used only for the subgaussian law in Theorem~\ref{thm:main}.  The
bounded-density branch below instead uses
Proposition~\ref{prop:density-block-high-band}.
\end{enumerate}

\subsection{Cleared transfer and the full-Fock identity}

Fix \(z\in\C\).  On the event \(\det\mathsf B_j\neq0\), the block
recurrence \((X_N-zI_N)\psi=0\) is
\begin{equation}
 \binom{\psi_{j+1}}{\psi_j}
 =T_j(z)\binom{\psi_j}{\psi_{j-1}},\qquad
 T_j(z):=
 \begin{pmatrix}
  -\mathsf B_j^{-1}(\mathsf A_j-zI_W)&
  -\mathsf B_j^{-1}\mathsf C_j\\
  I_W&0
 \end{pmatrix}.
 \label{eq:gb-transfer}
\end{equation}
For \(0\leq r\leq2W\), recall that we have defined
\begin{equation}
 \mathcal A_j^{(r)}(z):=
 (\det\mathsf B_j)\,\bigwedge\nolimits^rT_j(z).
 \label{eq:gb-cleared-transfer}
\end{equation}
Exterior functoriality gives the relation
\begin{equation}
 \mathcal A_b^{(r)}\cdots\mathcal A_a^{(r)}
 =\left(\prod_{j=a}^b\det\mathsf B_j\right)
  \bigwedge\nolimits^r(T_b\cdots T_a).
 \label{eq:gb-wedge-product}
\end{equation}

For $m=N/W$, denote the cyclic monodromy by
\begin{equation}
 M_{\rm cyc}(z):=T_m(z)\cdots T_1(z).
 \label{eq:gb-cyclic-monodromy}
\end{equation}
Periodicity is \(M_{\rm cyc}(z)v=v\).
Lemma~\ref{lem:local-block-floquet} and the standard exterior identity
give
\begin{align}
 \det(X_N-zI_N)
 &=\varepsilon_{m,W}
   \left(\prod_{j=1}^m\det\mathsf B_j\right)
   \det(I_{2W}-M_{\rm cyc}(z))
 \notag\\
 &=\varepsilon_{m,W}\sum_{r=0}^{2W}(-1)^r
   \operatorname{tr}\!\left(
    \mathcal A_m^{(r)}\cdots\mathcal A_1^{(r)}
   \right),
 \qquad |\varepsilon_{m,W}|=1,
 \label{eq:gb-full-fock}
\end{align}
where the last display follows from the following exterior identity:
\begin{equation}
 \det(I_{2W}-M_{\rm cyc})=\sum_{r=0}^{2W}(-1)^r
   \operatorname{tr}(\bigwedge\nolimits^rM_{\rm cyc}),
 \qquad
 \bigwedge\nolimits^r(ST)=
 (\bigwedge\nolimits^rS)(\bigwedge\nolimits^rT).
 \label{eq:gb-exterior-identities}
\end{equation}
Both sides of \eqref{eq:gb-full-fock} are polynomial in the normalized
block-matrix entries, so the identity holds without assuming interface
invertibility.

It is useful to record now what happens after a three-site packet is cut
from the ring.  Write the monodromy as
\(M_{\rm cyc}=R_{\rm out}R_{\partial}\), let \(b_{\partial}\) and
\(c_{\rm out}\) be the products of the right-interface determinants in
the packet and in its complement, and define
\(\mathcal Q^{(r)}:=b_{\partial}\bigwedge^rR_{\partial}\).  Then
\begin{equation}
 \det(X_N-zI_N)
 =\varepsilon_{m,W}c_{\rm out}
  \sum_{r=0}^{2W}(-1)^r\operatorname{tr}\!\left(
   \mathcal Q^{(r)}\bigwedge\nolimits^rR_{\rm out}\right).
 \label{eq:gb-cut-exterior-sum}
\end{equation}
This is an alternating sum and cannot be bounded below by its largest
degree because cancellation is possible.  The local theorem instead
compares its random value with the norm of its complete coefficient
vector.  The all-minor identity identifies that norm, up to
\(e^{O(W\log W)}\), with
\(\det(I_{2W}+R_{\rm out}^*R_{\rm out})^{1/2}\); this exterior volume is
within a factor \(2^W\) of
\(\max_r\|\bigwedge^rR_{\rm out}\|\).  This is the precise bridge from
the local packet estimate to the global exterior pressure.

\subsubsection{The block size estimates for the interface events}

The companion-minor expansion of \eqref{eq:gb-cleared-transfer}, followed
by \eqref{eq:gb-interface-det-control}, gives on \(\mathcal G_j\),
uniformly in \(0\leq r\leq2W\),
\begin{equation}
 \|\mathcal A_j^{(r)}\|\leq e^{C_{K,z}W\log W}.
 \label{eq:gb-one-step-upper}
\end{equation}
Also
\begin{equation}
 \det T_j=\pm\frac{\det\mathsf C_j}{\det\mathsf B_j}.
 \label{eq:gb-transfer-det}
\end{equation}
Hodge duality,
\begin{equation}
 \|\bigwedge\nolimits^rT_j^{-1}\|
 =|\det T_j|^{-1}
  \|\bigwedge\nolimits^{2W-r}T_j\|,
 \label{eq:gb-hodge}
\end{equation}
therefore yields the exact comparison
\begin{equation}
 \|(\mathcal A_j^{(r)})^{-1}\|
 \leq
 \frac{\|\mathcal A_j^{(2W-r)}\|}
 {|\det\mathsf B_j\,\det\mathsf C_j|}
 \leq e^{C_{K,z}W\log W}.
 \label{eq:gb-one-step-inverse}
\end{equation}
As a result, if \(J\) contains \(s\) consecutive block sites, so that
the corresponding scalar length is \(L=sW\), then on the good event
\(\mathcal G_J:=\bigcap_{j\in J}\mathcal G_j\),
\begin{equation}
 \left|\log\left\|
 \prod_{j\in J}^{\leftarrow}\mathcal A_j^{(r)}
 \right\|\right|
 \leq C_{K,z}L\log W.
 \label{eq:gb-product-control}
\end{equation}
Here the arrow denotes chronological multiplication, with later sites on
the left.  This cell-level bound is the only boundedness input used for
concentration.

\subsection{Mesoscopic cells and clipped pressure}

Now we use the idea of computing the log determinant of a band matrix-like system where the size of the matrix is $M_W\sim W^{1.01}$, so that the circular law result for large bandwidth can be used.

To specify the size of the smaller subsystem, we fix the following integer scales
\begin{equation}
 s_W:=\lceil W^{1/200}\rceil,\qquad
 c_W:=s_W+3,\qquad
 \ell_W:=c_WW,
 \label{eq:gb-cell-scales}
\end{equation}
and
\begin{equation}
 K_W:=\lceil W^{1/200}\rceil,\qquad
 M_W:=W(K_Wc_W+3).
 \label{eq:gb-anchor-scales}
\end{equation}
Thus \(\ell_W=W^{201/200+o(1)}\) and
\(M_W=W^{101/100+o(1)}\).

A \textbf{complete cell} consists of \(c_W=s_W+3\) consecutive block
sites and therefore has scalar size \(\ell_W=c_WW\).  It begins with
three reset sites and ends with \(s_W\) core sites.  The labeling is
circular, so that if its first site is denoted by \(a_k\), then we set
\begin{align}
 \mathscr R_{{\rm reset},k}^{(r)}
 &:=\mathcal A_{a_k+2}^{(r)}
    \mathcal A_{a_k+1}^{(r)}
    \mathcal A_{a_k}^{(r)},
 \label{eq:gb-cell-reset}\\
 \mathscr R_{{\rm core},k}^{(r)}
 &:=\mathcal A_{a_k+c_W-1}^{(r)}\cdots
    \mathcal A_{a_k+3}^{(r)},
 \label{eq:gb-cell-core}\\
 \mathscr M_{{\rm cell},k}^{(r)}
 &:=\mathscr R_{{\rm core},k}^{(r)}\mathscr R_{{\rm reset},k}^{(r)}.
 \label{eq:gb-cell-product}
\end{align}
On the interface-invertibility event where all $B_j,C_j$ invertible, we also write
\begin{equation}
 R_{{\rm core},k}:=
 T_{a_k+c_W-1}\cdots T_{a_k+3},
 \qquad
 c_{{\rm core},k}:=
 \prod_{j=a_k+3}^{a_k+c_W-1}\det\mathsf B_j.
 \label{eq:gb-core-transfer-and-scalar}
\end{equation}
Then we will use the following exact identity
\begin{equation}
 \mathscr R_{{\rm core},k}^{(r)}
 =c_{{\rm core},k}\bigwedge\nolimits^rR_{{\rm core},k}.
 \label{eq:gb-core-cleared-form}
\end{equation}
Under the original
unconditioned product law, we first introduce a cutoff function
\begin{equation}
 \clip_{[a,b]}(t):=\min\{b,\max\{a,t\}\},
 \qquad \clip_{[a,b]}(-\infty):=a,
 \label{eq:gb-clip-definition}
\end{equation}
and then define the following truncated properties, where $C_{K,z}$ was defined in \eqref{eq:gb-product-control},
\begin{align}
 Y_{k,r}&:=\log\|\mathscr R_{{\rm core},k}^{(r)}\|,
 \label{eq:gb-core-log}\\
 \widetilde Y_{k,r}
 &:=\clip_{[-C_{K,z}\ell_W\log W,\,C_{K,z}\ell_W\log W]}Y_{k,r},
 \label{eq:gb-clipped-log}\\
 \Phi_{W,r}(z)&:=\E\widetilde Y_{k,r},
 \label{eq:gb-pressure}\\
 r_*(W,z)&:=\min\arg\max_{0\leq r\leq2W}\Phi_{W,r}(z).
 \label{eq:gb-deterministic-degree}
\end{align}
  On the cell
interface event where all $\mathcal G_j$ holds, clipping is inactive by \eqref{eq:gb-product-control} so clipping returns $Y_{k,r}$ itself.
The degree \(r_*\) is deterministic and is chosen before any anchor or
target matrix is sampled.

\begin{lemma}[Independent-cell concentration]
\label{lem:gb-cell-concentration}
Let \(K_c\geq1\) be an integer.  For any \(K_c\) disjoint complete cells
and every \(u>0\),
\begin{align}
 &\Prob\left\{
 \max_{0\leq r\leq2W}
 \left|\sum_{k=1}^{K_c}
   (\widetilde Y_{k,r}-\Phi_{W,r})\right|
 >
 C_{K,z}\ell_W\log W\sqrt{K_c(\log W+u)}
 \right\}
 \leq2e^{-u}.
 \label{eq:gb-hoeffding}
\end{align}
\end{lemma}

\begin{proof}
For fixed \(r\), the variables \(\widetilde Y_{k,r}\) are independent.
Each lies in an interval of length \(2C_{K,z}\ell_W\log W\).  Hoeffding's inequality
\cite[Theorem~2.2.6]{vershynin2019high}, followed by a union bound over
only \(2W+1\) exterior degrees, proves
\eqref{eq:gb-hoeffding}.  No interface event is conditioned upon.
\end{proof}

\subsubsection{Roadmap of the global argument}
\label{subsec:gb-roadmap}

All model-specific anti-concentration estimates have already been proved.
The purpose of this section is only to prove the normalized log determinant converges to the desired limit, via comparing the log determinant to a deterministic normalized pressure and then showing the pressure arises in another system with much larger bandwidth, where circular law is already proven.

\begin{proposition}[Global pressure lifting]
\label{prop:gb-roadmap}
For every fixed \(z\in\C\),
\[
 \frac1N\log|\det(X_N-zI_N)|
 \xrightarrow{\Prob}U_{\cir}(z).
\]
\end{proposition}

\begin{proof}
We only outline the argument, and the detailed verification is scattered in later subsections.

First, for \(K_c\) complete mesoscopic cells where $K_c\to\infty$,
\eqref{eq:gb-pressure-sandwich} proves that
\[
 \frac1{K_c\ell_W}
 \max_{0\le r\le2W}
 \log\left\|
   \mathscr M_{{\rm cell},K_c}^{(r)}\cdots
   \mathscr M_{{\rm cell},1}^{(r)}
 \right\|
 =
 \frac1{\ell_W}\max_{0\le r\le2W}\Phi_{W,r}(z)
 +o_{\Prob}(1).
\]

Second, we apply the same formula to another block band matrix with bandwidth $W$ but matrix size roughly $W^{1.01}$.  The terminal
comparison \eqref{eq:gb-terminal-pressure} identifies its logarithmic
determinant with its maximal exterior growth.  Hence
\eqref{eq:gb-anchor-pressure} holds.  Since the size of this auxiliary
matrix is $\sim W^{1.01}$,
Proposition~\ref{prop:subgaussian-block-high-band} applies and gives
\eqref{eq:gb-high-band-logdet}.  We use it to obtain the deterministic calibration
\[
 \frac1{\ell_W}\max_{0\le r\le2W}\Phi_{W,r}(z)
 \longrightarrow U_{\cir}(z),
\]
which is \eqref{eq:gb-pressure-calibration}.

Finally, suppose \(N\ge M_W\).  The already-derived anti-concentration yields \eqref{eq:gb-target-terminal-pressure}:
$$
\log|\det(X_N-zI_N)|
 =
 \max_{0\leq r\leq2W}
 \log\left\|
 c_{\rm out}^{\rm tar}\bigwedge\nolimits^rR_{\rm out}^{\rm tar}
 \right\|
 +o_{\Prob}(N).
$$

Combining these, we get
\[
 \frac1N\log|\det(X_N-zI_N)|
 \xrightarrow{\Prob}U_{\cir}(z).
\]
If
\(N<M_W\), then \(W/N^{0.99}\to\infty\), so
Proposition~\ref{prop:subgaussian-block-high-band} applies directly to
the target ring and gives the same displayed target limit.
The probability estimates needed in both branches are collected in the
final probability ledger.

The rest of the section justifies each claim in this sketched proof.

\end{proof}

\subsection{Removing frame dependence across cells by throwing three packets}

We now consider how different complete cells can be concatenated, by specifying an $r$-frame of the recursion that carries the $r$ largest singular values of each step.  Expose
all cell cores and the two outer endpoint interfaces used by each
three-block reset.  For a fixed degree \(r\), choose a deterministic
decomposable unit wedge \(\widehat v_{0,r}\) (i.e., it is a decomposable vector in the wedge space with unit norm).  Having exposed the first
\(k-1\) complete cells, set
\begin{equation}
 \widehat v_{k,r}:=
 \frac{\mathscr M_{{\rm cell},k}^{(r)}
       \widehat v_{k-1,r}}
 {\|\mathscr M_{{\rm cell},k}^{(r)}
       \widehat v_{k-1,r}\|}
 \label{eq:gb-recursive-wedge}
\end{equation}
when the denominator is nonzero, and set it equal to a fixed reference
wedge otherwise.  Let \(\mathcal F_{k-1}^{(r)}\) be the sigma-field
generated by all cores, all exposed endpoint interfaces, and the seven
internal blocks of resets \(1,\ldots,k-1\).  Then
\(\widehat v_{k-1,r}\) is
\(\mathcal F_{k-1}^{(r)}\)-measurable, while the seven internal blocks of
reset \(k\) are independent of \(\mathcal F_{k-1}^{(r)}\).  On the
all-interface-good event every factor is invertible, so the denominator
in \eqref{eq:gb-recursive-wedge} is nonzero and all these wedges remain
decomposable.

\textbf{Forming a frame.} For each complete cell $k$, we choose singular pairs of the core transfer
\[
 R_{{\rm core},k}w_j=s_ju_j,\qquad R_{{\rm core},k}^*u_j=s_jw_j,\qquad 1\leq j\leq r,
\]
corresponding to the largest \(r\) singular values, and, for each cell
\(k\), define
\[
 \widehat w=w_1\wedge\cdots\wedge w_r,\qquad
 \widehat u_0=u_1\wedge\cdots\wedge u_r,\qquad
 \widehat u=\frac{c_{{\rm core},k}}
 {|c_{{\rm core},k}|}\widehat u_0.
\]
Then \(\widehat w,\widehat u\) form a top singular pair for the cleared
core operator \(\mathscr R_{{\rm core},k}^{(r)}\) and its adjoint:
\begin{equation}
 \mathscr R_{{\rm core},k}^{(r)}\widehat w
 =\|\mathscr R_{{\rm core},k}^{(r)}\|\widehat u,\qquad
 (\mathscr R_{{\rm core},k}^{(r)})^*\widehat u
 =\|\mathscr R_{{\rm core},k}^{(r)}\|\widehat w.
 \label{eq:gb-top-singular-pair}
\end{equation}
In particular, for any other vector \(y\),
\begin{equation}
 |\langle\widehat u,\mathscr R_{{\rm core},k}^{(r)}y\rangle|
 =\|\mathscr R_{{\rm core},k}^{(r)}\|\,|\langle\widehat w,y\rangle|.
 \label{eq:gb-adjoint-test}
\end{equation}

The local complex-frame theorem and
\(\mathscr R_{{\rm reset},k}^{(r)}\) now use the same normalized blocks.
Thus the local cleared packet operator \(\mathcal Q^{(r)}\) is exactly
\(\mathscr R_{{\rm reset},k}^{(r)}\), with no conversion of notation.

\textbf{Getting rid of the frame dependence.}
We wish to apply Theorem~\ref{thm:local-complex-frame} conditionally with the
incoming frame \(\widehat v_{k-1,r}\), to show that the dependence on the frame $\widehat{v}_{k-1,r}$ is negligible in the large sum. To fix notations (the reader can ignore these notations for a first reading), we define the pairing and its
coefficient norm by
\begin{equation}
 Z_{k,r}:=
 \langle\widehat w,\mathscr R_{{\rm reset},k}^{(r)}
       \widehat v_{k-1,r}\rangle,
 \qquad
 \Gamma_{k,r}:=
 \|\operatorname{Coeff}_\mathbf{x}Z_{k,r}\|_2.
 \label{eq:gb-reset-pairing}
\end{equation}
To make the event used here explicit, set
\begin{align}
 \mathcal E_{\partial,k}:=\bigl\{\,
 &\|\mathsf C_{a_k}\|,\|\mathsf B_{a_k+2}\|\leq C_K,\notag\\
 &e^{-C_KW}\leq
   |\det\mathsf C_{a_k}|,|\det\mathsf B_{a_k+2}|
   \leq e^{C_KW},\notag\\
 &\|\mathsf C_{a_k}^{-1}\|,
   \|\mathsf B_{a_k+2}^{-1}\|\leq e^{C_KW}
 \,\bigr\}.
 \label{eq:gb-endpoint-regularity}
\end{align}
These are precisely the two endpoint interfaces \(C_L\) and \(B_R\) of
the three-site reset packet.  They are exposed before the seven internal
blocks, and \(\Prob(\mathcal E_{\partial,k}^c)\leq2e^{-c_KW}\).
On \(\mathcal E_{\partial,k}\),
\begin{equation}
 e^{-CW\log W}\leq\Gamma_{k,r}\leq e^{CW\log W},
 \label{eq:gb-reset-coefficients}
\end{equation}
and \eqref{eq:gb-adjoint-test} gives the exact inequality
\begin{equation}
 \|\mathscr R_{{\rm core},k}^{(r)}
      \mathscr R_{{\rm reset},k}^{(r)}
      \widehat v_{k-1,r}\|
 \geq
 \|\mathscr R_{{\rm core},k}^{(r)}\|\,|Z_{k,r}|.
 \label{eq:gb-reset-test}
\end{equation}

Set \(T_W:=C_2\ell_W\log W\), with \(C_2\) larger than the constant in
\eqref{eq:gb-product-control}, and define the clipped actual loss
\begin{equation}
 L_{k,r}:=
 \min\left\{T_W,
 \left[
  \log\|\mathscr R_{{\rm core},k}^{(r)}\|
  -\log\|\mathscr R_{{\rm core},k}^{(r)}
           \mathscr R_{{\rm reset},k}^{(r)}
           \widehat v_{k-1,r}\|
 \right]_+\right\}.
 \label{eq:gb-actual-loss}
\end{equation}
The value is \(T_W\) if a logarithm on the second line is
\(-\infty\).  Define
\begin{equation}
 \mathcal E_{{\rm core},k}:=
 \bigcap_{j=a_k+3}^{a_k+c_W-1}\mathcal G_j,
 \qquad
 H_k:=\mathcal E_{{\rm core},k}\cap\mathcal E_{\partial,k}.
 \label{eq:gb-Hk}
\end{equation}
Thus \(H_k\) says exactly that the core sites and the two exposed reset
interfaces satisfy the determinant, inverse, and norm bounds in (L1).
It is measurable before the seven fresh internal blocks are revealed.
Combining
\eqref{eq:gb-reset-coefficients}--\eqref{eq:gb-reset-test} with the capped
local estimate in Theorem~\ref{thm:local-complex-frame} gives the corrected conditional inequality
\begin{align}
 \mathbf1_{H_k}\,
 \E\!\left[
  L_{k,r}\mid
  \mathcal F_{k-1}^{(r)}
 \right]
 \leq
 \mathbf1_{H_k}\bigl(CW\log W+p_WT_W\bigr).
 \label{eq:gb-conditional-loss}
\end{align}
Finally we average over the core.  Since
\(\Prob(H_k^c)\leq e^{-cW}\) and \(L_{k,r}\leq T_W\),
\begin{equation}
 \E L_{k,r}
 \leq CW\log W+(p_W+e^{-cW})T_W.
 \label{eq:gb-unconditional-loss}
\end{equation}
On the eventual all-interface-good event where all $\mathcal{G}_j$ holds, we can verify that the two terms with the log in the definition of $L_{k,r}$ are bounded in absolute value by  \(T_W/2\); hence the clipping upper bound at $T_W$ in
\eqref{eq:gb-actual-loss} is inactive and \(L_{k,r}\) is the true reset loss, including when \(Z_{k,r}=0\).

For a deterministic degree \(r\), tower expectation and Markov's
inequality now imply
\begin{equation}
 \frac1{K_c\ell_W}\sum_{k=1}^{K_c}L_{k,r}
 =
 O_{\Prob}\left(
  \frac{W\log W}{\ell_W}
  +p_W\log W+e^{-cW}\log W
 \right)=o_{\Prob}(1).
 \label{eq:gb-total-reset-loss}
\end{equation}
The failure probability \(p_W\) in
Theorem~\ref{thm:local-complex-frame} has been integrated to give
\(o_{\Prob}(1)\), rather than summed over all cells.
\subsubsection{The deterministic pressure: comparison from many cells to one cell}
We combine the previous estimates to show that, for a transfer through
\(K_c\) complete cells of scalar length \(\ell_W\), the normalized
pressure equals the deterministic quantity
\(\max_r\Phi_{W,r}(z)\) up to a negligible error.  Thus the precise
number \(K_c\) of cells is irrelevant once it is sufficiently large.
For any positive integer $K_c$, form the following product of $K_c$ complete cells:
\begin{equation}\label{productcellus1}
 \mathscr M_{1:K_c}^{(r)}
 :=\mathscr M_{{\rm cell},K_c}^{(r)}\cdots\mathscr M_{{\rm cell},1}^{(r)}.
\end{equation}
We then apply \eqref{eq:gb-reset-test} successively to vectors in the recursive wedges
\eqref{eq:gb-recursive-wedge}.  Telescoping their norms gives, on the
global interface event where every \(\mathcal G_j\) holds,
\begin{equation}
 \log\|\mathscr M_{1:K_c}^{(r_*)}\|
 \geq
 \sum_{k=1}^{K_c}\widetilde Y_{k,r_*}
 -\sum_{k=1}^{K_c}L_{k,r_*}.
 \label{eq:gb-splice-lower}
\end{equation}
Only the one deterministic degree \(r_*\) is used here.
Submultiplicativity, \eqref{eq:gb-product-control}, and the three reset sites
per cell give, simultaneously for every \(0\le r\le 2W\),
\begin{equation}
 \log\|\mathscr M_{1:K_c}^{(r)}\|
 \leq
 \sum_{k=1}^{K_c}\widetilde Y_{k,r}
 +CK_cW\log W.
 \label{eq:gb-splice-upper}
\end{equation}
Combining Lemma~\ref{lem:gb-cell-concentration},
\eqref{eq:gb-total-reset-loss}, and
\eqref{eq:gb-deterministic-degree}, we find
\begin{align}
 \frac1{K_c\ell_W}
 \max_{0\leq r\leq2W}\log\|\mathscr M_{1:K_c}^{(r)}\|
 &=
 \frac1{\ell_W}\max_{0\leq r\leq2W}\Phi_{W,r}(z)
 \notag\\
 &\quad+
 O_{\Prob}\left(
  \frac{W\log W}{\ell_W}
  +p_W\log W
  +\frac{\log^{3/2}W}{\sqrt{K_c}}
 \right).
 \label{eq:gb-pressure-sandwich}
\end{align}
This estimate is sufficient when a short cycle has length \(K_c\ell_W\).
More generally, a block-site interval may be decomposed into complete
\(c_W\)-site cells and a final incomplete interval of fewer than \(c_W\)
block sites.  This incomplete interval changes either side of the
unnormalized sandwich by at most
\begin{equation}
 C\ell_W\log W
 \label{eq:gb-remainder-cost}
\end{equation}
on the simultaneous interface-regularity event $\cap_j\mathcal G_j$.  Indeed, if its cleared
exterior product is
\(R_{\rm rem}^{(r)}\), then both
\(\|R_{\rm rem}^{(r)}\|\) and
\(\|(R_{\rm rem}^{(r)})^{-1}\|\) are at most
\(e^{C\ell_W\log W}\) by \eqref{eq:gb-product-control}; applying these
two bounds gives the upper and lower comparisons, respectively.

\subsection{From pressure to circular law: calibration on the independent short ring}
\subsubsection{Determining the limit via a short cycle system}\label{todetermine5.5.1}
To determine the limit $\Phi_{W,r}(z)$, we need to consider a slightly larger system that contains a growing number of complete cells of size $\ell_W$.
For this, we take an independent periodic ring with
\begin{equation}
 m_M:=K_Wc_W+3,\qquad M_W=Wm_M.
 \label{eq:gb-anchor-sites}
\end{equation}
It consists of exactly \(K_W\) cells followed by three fresh terminal
sites.  For this cut, we again denote the packet and complementary transfers by
\begin{equation}
 R_{\partial}^{\rm anc}:=
 T_{m_M}T_{m_M-1}T_{m_M-2},
 \qquad
 R_{\rm out}^{\rm anc}:=T_{m_M-3}\cdots T_1.
 \label{eq:gb-anchor-two-transfers}
\end{equation}

The boundary determinant estimate and coefficient-norm comparison in (L2)
now give the following comparison, directly in the normalization of
\(X_{M_W}\). Denote by
\[
 c_{\rm out}^{\rm anc}:=
 \prod_{j=1}^{m_M-3}\det\mathsf B_j.
\]
Then the outside cleared operators are, recalling the product cell definition in \eqref{productcellus1},
\begin{equation}
 c_{\rm out}^{\rm anc}\bigwedge^rR_{\rm out}^{\rm anc}
 :=\mathscr M_{1:K_W}^{(r)}.
 \label{eq:gb-anchor-cleared-outside}
\end{equation}
For a deterministic sign \(\varepsilon_{\rm anc}\), the block-matrix
determinant and its packet coefficient vector are exactly
\[
 \varepsilon_{\rm anc}c_{\rm out}^{\rm anc}
 \mathscr D_{R_{\rm out}^{\rm anc}}
 \quad\hbox{and}\quad
 \varepsilon_{\rm anc}c_{\rm out}^{\rm anc}
 \operatorname{Coeff}_\mathbf{x}\mathscr D_{R_{\rm out}^{\rm anc}}.
\]
In
particular, the norm of the corresponding block-matrix coefficient vector is
\[
 |c_{\rm out}^{\rm anc}|\,
 \mathscr C(R_{\rm out}^{\rm anc})
 =e^{O(W\log W)}
   \max_{0\le r\le2W}
   \left\|c_{\rm out}^{\rm anc}
   \bigwedge^rR_{\rm out}^{\rm anc}\right\|.
\]
This is exactly the normalization of \(\mathscr M_{1:K_W}^{(r)}\), not an
additional sum of interface log-determinants.  
\begin{lemma}
For the periodic full-block band matrix with \(m_M\) block sites and
scalar size \(M_W=Wm_M\), one has
\begin{equation}
 \log|\det(X_{M_W}-zI_{M_W})|
 =
 \max_{0\leq r\leq2W}
 \log\|\mathscr M_{1:K_W}^{(r)}\|
 +o_{\Prob}(M_W).
 \label{eq:gb-terminal-pressure}
\end{equation}
\end{lemma}

\begin{proof}[Proof of \eqref{eq:gb-terminal-pressure}]
Denote by
\[
 \Theta:=R_{\rm out}^{\rm anc},
 \qquad
 c:=c_{\rm out}^{\rm anc}.
\]
On the simultaneous interface-regularity event, the linear algebraic wedge expansion and
the definition \eqref{eq:gb-anchor-cleared-outside} give, for some deterministic
\(|\varepsilon_{\rm anc}|=1\), where we recall the definition $\mathscr D_\Theta(\Xi)$ \eqref{eq:local-DTheta}:
\begin{equation}
 \det(X_{M_W}-zI_{M_W})
 =\varepsilon_{\rm anc}c\,\mathscr D_\Theta(\Xi),
 \qquad
 \mathscr M_{1:K_W}^{(r)}
 =c\bigwedge\nolimits^r\Theta.
 \label{eq:gb-anchor-two-identities}
\end{equation}

Let \(\mathcal F_{\rm anc}\) be the sigma-field generated by the outside
cells and the two exposed endpoint matrices of the terminal packet, and
condition on \(\mathcal F_{\rm anc}\).  The remaining seven packet blocks
are fresh.  Thus, for every fixed \(\varepsilon>0\), we apply
\eqref{eq:local-boundary-capped} with \(T=\varepsilon M_W\) and obtain
\[
 \Prob\left\{
  \log\frac{\mathscr C(\Theta)}
            {|\mathscr D_\Theta(\Xi)|}
  >\varepsilon M_W
  \,\middle|\,\mathcal F_{\rm anc}
 \right\}
 \le
 p_W+\frac{CW\log W}{\varepsilon M_W}.
\]
The reverse estimate \eqref{eq:local-boundary-reverse} gives, conditionally
on \(\mathcal F_{\rm anc}\), outside an event of probability at most
\(e^{-cW}\),
\[
 \log_+\frac{|\mathscr D_\Theta(\Xi)|}{\mathscr C(\Theta)}
 \le CW\log W.
\]
Since \(p_W=o(1)\) and \(W\log W=o(M_W)\), uniformly in the conditioned
outside data,
\begin{equation}
 \log|\mathscr D_\Theta(\Xi)|
 =\log\mathscr C(\Theta)+o_{\Prob}(M_W).
 \label{eq:gb-anchor-value-to-coeff}
\end{equation}

By Lemma~\ref{lem:local-boundary-volume} and
\eqref{eq:local-boundary-volume-vs-exterior}, we can measure the size of the coefficient by
\[
 \log\mathscr C(\Theta)
 =
 \max_{0\le r\le2W}\log\|\bigwedge\nolimits^r\Theta\|
 +O(W\log W).
\]
Combining this with \eqref{eq:gb-anchor-two-identities} and
\eqref{eq:gb-anchor-value-to-coeff} yields
\[
 \log|\det(X_{M_W}-zI_{M_W})|
 =
 \max_{0\le r\le2W}
 \log\left\|c\bigwedge\nolimits^r\Theta\right\|
 +o_{\Prob}(M_W)
 =
 \max_{0\le r\le2W}
 \log\|\mathscr M_{1:K_W}^{(r)}\|
 +o_{\Prob}(M_W).
\]
Finally, the complement of the simultaneous interface-regularity event $\cap_j\mathcal G_j$
has probability \(o(1)\), which completes the proof.\end{proof}

Then we compare to the stable coefficient $\Phi_{W,r}(z)$.
We apply \eqref{eq:gb-pressure-sandwich} with \(K_c=K_W\) to get
\begin{equation}
 \frac1{M_W}\log|\det(X_{M_W}-zI_{M_W})|
 =
 \frac{K_W}{M_W}\max_r\Phi_{W,r}(z)
 +o_{\Prob}(1).
 \label{eq:gb-anchor-pressure}
\end{equation}
This is a law-of-large-numbers step where we need $K_W\to\infty $ to replace the random sum by its mean $\Phi_{W,r}(z)$. Here, the explicit normalized errors are
\begin{equation}
 O_{\Prob}\left(
  W^{-1/200}\log W
  +\frac{\log^{3/2}W}{\sqrt W}
  +W^{-1/400}\log^{3/2}W
 \right)+o_{\Prob}(1).
 \label{eq:gb-anchor-errors}
\end{equation}
Since \(W=M_W^{100/101+o(1)}\), in particular
\(W/M_W^{0.99}\to\infty\).  Thus the high-band log-determinant
Proposition~\ref{prop:subgaussian-block-high-band} is used only in the
range \(W\gg M_W^{0.99}\).
\begin{corollary}Comparing
\eqref{eq:gb-high-band-logdet} and \eqref{eq:gb-anchor-pressure} yields the
deterministic pressure calibration
\begin{equation}
 \frac1{\ell_W}\max_{0\leq r\leq2W}\Phi_{W,r}(z)
 \longrightarrow U_{\cir}(z).
 \label{eq:gb-pressure-calibration}
\end{equation}
\end{corollary}
Here
$\frac{K_W\ell_W}{M_W}=1-\frac{3W}{M_W},
$
and the clipping bound gives the rounding error
\begin{equation}
 \left|
  \frac{K_W}{M_W}-\frac1{\ell_W}
 \right|\max_r|\Phi_{W,r}|
 \leq C\frac{W\log W}{M_W}=o(1).
 \label{eq:gb-anchor-rounding}
\end{equation}

\subsubsection{Uplifting to the original long cycle matrix}

Suppose first that \(N\geq M_W\).  Reserve the final three block sites for
the terminal seam and define
\begin{equation}
 K_N:=\left\lfloor\frac{m-3}{c_W}\right\rfloor.
 \label{eq:gb-target-cell-count}
\end{equation}
The first \(K_Nc_W\) sites form complete cells; the remaining open prefix
has fewer than \(c_W\) sites.  Since \(N\geq M_W\),
\begin{equation}
 K_N\geq K_W-1\longrightarrow\infty,\qquad
 \frac{K_N\ell_W}{N}=1+O\left(\frac{\ell_W}{N}\right).
 \label{eq:gb-target-rounding}
\end{equation}
For the final three block sites, define
\[
 R_{\partial}^{\rm tar}:=T_mT_{m-1}T_{m-2},
 \qquad
 R_{\rm out}^{\rm tar}:=T_{m-3}\cdots T_1.
\]
The first operator is the transfer through the terminal packet, while
the second is the chronologically ordered transfer through all remaining
sites, including the incomplete interval.  Condition on
those outside sites and also on the two packet endpoint blocks
\(\mathsf C_{m-2}\) and \(\mathsf B_m\).  On their
endpoint-regularity event, the seven
internal packet blocks remain fresh.  The extended boundary determinant
estimate applies with the actual outside boundary map
\begin{equation}
 \Theta=R_{\rm out}^{\rm tar}.
 \label{eq:gb-target-actual-substitution}
\end{equation}
Denote by
\[
 c_{\rm out}^{\rm tar}:=
 \prod_{j=1}^{m-3}\det\mathsf B_j.
\]
The same two-sided terminal comparison used in
\eqref{eq:gb-terminal-pressure} gives
\begin{equation}
 \log|\det(X_N-zI_N)|
 =
 \max_{0\leq r\leq2W}
 \log\left\|
 c_{\rm out}^{\rm tar}\bigwedge\nolimits^rR_{\rm out}^{\rm tar}
 \right\|
 +o_{\Prob}(N).
 \label{eq:gb-target-terminal-pressure}
\end{equation}
The cleared outside product in this display is the product of the
\(K_N\) complete cells and the final incomplete interval.  Apply
\eqref{eq:gb-pressure-sandwich} to the complete cells, absorb the
incomplete interval using \eqref{eq:gb-remainder-cost}, use the one
terminal packet, and invoke \eqref{eq:gb-pressure-calibration}.  This proves the claimed log-determinant limit
\begin{equation}
 \frac1N\log|\det(X_N-zI_N)|
 \xrightarrow{\Prob}U_{\cir}(z).
 \label{eq:gb-global-logdet}
\end{equation}
The normalized error before taking the limit is
$o_{\Prob}(1)$
for the scales \eqref{eq:gb-cell-scales}--\eqref{eq:gb-anchor-scales}.

If \(N<M_W\), then \(N\leq W^{101/100+o(1)}\), and consequently
\(W/N^{0.99}\to\infty\).  With, for example, $\omega_*=1/20$,
$0.99>8/9+\omega_*$; hence
Proposition~\ref{prop:subgaussian-block-high-band}, applied directly to
the target ring, gives \eqref{eq:gb-global-logdet} without a cell
decomposition.
An arbitrary sequence of dimensions may alternate between the two
branches.  Since the same limit holds on each of the two resulting
subsequences, it holds on the original sequence.

\subsection{Replacement and completion of the proof}

For every fixed \(z\in\C\), the preceding argument proves
\eqref{eq:gb-global-logdet}.
Taking \(z=0\) gives
\(\Prob\{\det X_N=0\}\to0\), since \(U_{\cir}(0)=-1/2\) is finite
whereas the normalized log determinant equals \(-\infty\) when \(X_N\)
is singular.

The Hilbert--Schmidt condition is immediate from the normalization:
\begin{equation}
 \frac1N\|X_N\|_{\HS}^2
 =
 \frac1{3WN}
 \sum_{\text{all }3NW\text{ displayed variance-one atoms}}\xi_{ij}^2
 \xrightarrow{\Prob}1.
 \label{eq:gb-HS}
\end{equation}
In particular, it is bounded in probability.  Let \(G_N\) be normalized
complex Ginibre.  By the Gaussian case of
\cite[Theorem~1.5]{Han2410}, its normalized log determinant converges to
\(U_{\cir}(z)\) on a full-measure set, and
\(N^{-1}\|G_N\|_{\HS}^2\) is bounded in probability.  Therefore
\begin{equation}
 \frac1N\log|\det(X_N-zI_N)|
 -\frac1N\log|\det(G_N-zI_N)|
 \xrightarrow{\Prob}0
 \label{eq:gb-replacement-logdet}
\end{equation}
for almost every \(z\).  The in-probability Tao--Vu replacement principle
\cite[Theorem 2.1]{TaoVuKrishnapur2010}, applied using
\eqref{eq:gb-HS}--\eqref{eq:gb-replacement-logdet}, gives
\[
 \mu_{X_N}-\mu_{G_N}\Longrightarrow0
 \quad\text{in probability}.
\]
Since \(\mu_{G_N}\Longrightarrow\mu_{\cir}\), this proves
Theorem~\ref{thm:main}.

\section{Proofs of technical results for the discrete periodic full-block case}
\label{supplementsection}
This section gives all proofs deferred from
Section~\ref{sec:local-interface}.
As permitted in the footnote introducing the formal packet variables,
we retain optional deterministic entry weights $a_e$ satisfying
$0<c_*\leq a_e\leq C_*<\infty$.  In the model stated in
Theorem~\ref{thm:main}, all these weights equal one.

\begin{proof}[\proofname\ of Lemma \ref{lem:local-interface-control}]
The subgaussian norm bound (cf. \cite{vershynin2019high}) gives
\(\P\{\|G\|>C_0\sqrt W\}\leq2e^{-c_0W}\) for two $K$-dependent constants $c_0,C_0$. This also gives the upper bound on the determinant.  For the lower bounds we use the following fixed-index and overcrowding
estimates for the bottom singular values of an iid matrix.  Denote the
singular values of $G$ in decreasing order by
\(s_1(G)\geq\cdots\geq s_W(G)\).  Nguyen's overcrowding theorem
\cite[Theorem~1.4, equations~(3)--(4)]{Nguyen} gives, for each fixed
\(0<\vartheta<1\), constants \(c,C,\gamma_0>0\) and the following
estimates.  Set \(k_0:=\lceil\gamma_0^{-1}\rceil+1\).  Equation~(4), with
\(u=k\varepsilon\), gives
\begin{equation}
 \P\left\{s_{W-k+1}(G)\leq \frac{k\varepsilon}{\sqrt W}\right\}
 \leq (C\varepsilon)^{(1-\vartheta)k^2}+e^{-cW}
 \label{eq:local-overcrowding}
\end{equation}
for \(k_0<k<\gamma_0W\), while for each fixed \(1\leq k\leq k_0\),
\begin{equation}
 \P\left\{s_{W-k+1}(G)\leq \frac{\varepsilon}{\sqrt W}\right\}
 \leq C_k\varepsilon^{k^2}+e^{-cW}.
 \label{eq:local-fixed-bottom}
\end{equation}
Choose \(a\) large and set \(\varepsilon_k=e^{-aW/k^2}\) up to
\(k=\lfloor\sqrt W\rfloor\); above \(\sqrt W\), use a sufficiently small
fixed \(\varepsilon_0\) up to \(\lfloor\gamma_0W/2\rfloor\).  A union bound
shows that, outside probability \(e^{-cW}\),
\begin{align*}
 \frac{s_{W-k+1}(G)}{\sqrt W}
 &\geq \frac{e^{-aW/k^2}}{W},&&1\leq k\leq k_0,\\
 \frac{s_{W-k+1}(G)}{\sqrt W}
 &\geq \frac{k e^{-aW/k^2}}{W},&&k_0<k\leq\sqrt W,\\
 \frac{s_{W-k+1}(G)}{\sqrt W}
 &\geq \frac{\varepsilon_0k}{W},&&
 \sqrt W<k\leq\gamma_0W/2.
\end{align*}
Monotonicity supplies a fixed lower bound for the remaining singular values.
Multiplication loses only \(O(W)\) in the logarithm, since
\(W\sum_{k\geq1}k^{-2}=O(W)\), the contribution below \(\sqrt W\) from
\(\log(k/W)\) is \(o(W)\), and the remaining sum is a Riemann sum of order
\(W\).  Since \(G_{\rm nor}=3^{-1/2}W^{-1/2}G\), the harmless factor
\(3^{-W/2}\) is absorbed in the constants, and
\(|\det G_{\rm nor}|\geq e^{-CW}\).  The case \(k=1\) likewise gives
\(s_W(G_{\rm nor})\geq e^{-CW}\), proving the inverse bound.
\end{proof}

\subsection{Discrete polynomial small-ball estimate: proof of
Proposition~\ref{prop:local-terminal}}

\subsubsection{Factorizing large singular values by strong RRQR, and polynomial LSV}

We recall the coordinate form of two-sided strong rank-revealing QR used in
the proof.

\begin{lemma}[Two-sided coordinate skeleton]
\label{lem:local-rrqr}
Let \(Q\in\C^{n\times n}\), let \(\tau\geq1\), and write
\(r:=\#\{j:s_j(Q)>\tau\}\).  There are row and column sets
\(I,J\subset[n]\), each of cardinality \(r\), such that
\begin{equation}
 K_{\mathrm{piv}}:=Q_{I,J}
 \label{eq:local-Kpiv}
\end{equation}
is invertible when \(r>0\).  Define
\begin{equation}
 X_{\mathrm{skel}}:=K_{\mathrm{piv}}^{-1}Q_{I,J^c},\quad
 Y_{\mathrm{skel}}:=Q_{I^c,J}K_{\mathrm{piv}}^{-1},\quad
 E_0:=Q_{I^c,J^c}-Y_{\mathrm{skel}}K_{\mathrm{piv}}X_{\mathrm{skel}}.
 \label{eq:local-skeleton-definitions}
\end{equation}
After a row-column permutation applied to $Q$, we can write
\begin{equation}
 Q=
 \begin{pmatrix}
  K_{\mathrm{piv}}&K_{\mathrm{piv}}X_{\mathrm{skel}}\\
  Y_{\mathrm{skel}}K_{\mathrm{piv}}&
  Y_{\mathrm{skel}}K_{\mathrm{piv}}X_{\mathrm{skel}}+E_0
 \end{pmatrix},
 \label{eq:local-skeleton-identity}
\end{equation}
and, with the universal choice \(C_{\rm RRQR}=4\), for all large $n$,
\begin{equation}
 s_j(K_{\mathrm{piv}})\geq n^{-C_{\rm RRQR}}s_j(Q),\quad
 \|X_{\mathrm{skel}}\|+\|Y_{\mathrm{skel}}\|
 \leq n^{C_{\rm RRQR}},\quad
 \|E_0\|\leq n^{C_{\rm RRQR}}\tau.
 \label{eq:local-skeleton-bounds}
\end{equation}
The first bound holds for \(j\leq r\).  For \(r=0\), the pivot is empty
and \eqref{eq:local-skeleton-identity} means \(Q=E_0\).
\end{lemma}

\begin{proof}
The case \(r=0\) is trivial, so assume \(r\geq1\).  We use the rank-revealing QR decomposition (RRQR) in
Gu--Eisenstat \cite[Theorem~3.2]{GuEisenstat} with the fixed parameter
\(f=2\), and denote by
\[
 a:=\sqrt{1+4r(n-r)},
 \qquad
 b:=2\sqrt{r(n-r)}.
\]
The first RRQR, applied to the columns of \(Q\), selects columns \(J\) and gives
matrices \(T,H\) such that
\begin{equation}
 Q_{:,J^c}=Q_{:,J}T+H,
 \qquad
 \|T\|\leq b,
 \qquad
 \|H\|\leq a\tau,
 \qquad
 s_j(Q_{:,J})\geq a^{-1}s_j(Q).
 \label{eq:first-rrqr}
\end{equation}
Here \(\|H\|\leq a\tau\) follows from
\(s_{r+1}(Q)\leq\tau\); when \(r=n\), the matrices \(T,H\) are empty.

The second RRQR is applied not independently to \(Q^*\), but to
\(Q_{:,J}^*\).  Thus it selects a row set \(I\) compatible with the
columns already selected in the first step.  Writing
\[
 K_{\mathrm{piv}}:=Q_{I,J},
\]
we obtain, denoting $Y_{\mathrm{skel}}$ as the matrix in the lemma statement,
\begin{equation}
 Q_{I^c,J}=Y_{\mathrm{skel}}K_{\mathrm{piv}},
 \qquad
 \|Y_{\mathrm{skel}}\|\leq b,
 \qquad
 s_j(K_{\mathrm{piv}})
 \geq a^{-1}s_j(Q_{:,J})
 \geq a^{-2}s_j(Q).
 \label{eq:second-rrqr}
\end{equation}
Equivalently, one may append \(n-r\) zero rows to \(Q_{:,J}^*\) before
this second application, thereby reducing it to the square formulation.

Restricting \eqref{eq:first-rrqr} to the rows in \(I\) gives
\[
 Q_{I,J^c}=K_{\mathrm{piv}}T+H_I,
\]
and hence
\begin{equation}
 X_{\mathrm{skel}}
 :=K_{\mathrm{piv}}^{-1}Q_{I,J^c}
 =T+K_{\mathrm{piv}}^{-1}H_I.
 \label{eq:X-from-two-rrqr}
\end{equation}
Restricting \eqref{eq:first-rrqr} to \(I^c\), and using
\eqref{eq:second-rrqr}, gives
\[
 Q_{I^c,J^c}
 =Y_{\mathrm{skel}}K_{\mathrm{piv}}T+H_{I^c}.
\]
Therefore
\begin{equation}
 E_0
 :=Q_{I^c,J^c}
   -Y_{\mathrm{skel}}K_{\mathrm{piv}}X_{\mathrm{skel}}
 =H_{I^c}-Y_{\mathrm{skel}}H_I.
 \label{eq:E-from-two-rrqr}
\end{equation}
These identities prove the asserted block decomposition.

Since \(s_r(Q)>\tau\), \eqref{eq:second-rrqr} implies
\[
 \|K_{\mathrm{piv}}^{-1}\|
 \leq \frac{a^2}{s_r(Q)}
 <\frac{a^2}{\tau}.
\]
Consequently,
\[
 \|X_{\mathrm{skel}}\|+\|Y_{\mathrm{skel}}\|
 \leq 2b+a^3,
 \qquad
 \|E_0\|\leq(1+b)a\tau.
\]
Finally, \(r(n-r)\leq n^2/4\), so
\[
 a\leq\sqrt{1+n^2},
 \qquad
 b\leq n.
\]
For \(n\geq2\), these estimates imply
\[
 s_j(K_{\mathrm{piv}})\geq n^{-4}s_j(Q),
 \qquad
 \|X_{\mathrm{skel}}\|+\|Y_{\mathrm{skel}}\|\leq n^4,
 \qquad
 \|E_0\|\leq n^4\tau.
\]
The case \(n=1\) is immediate.  Although Gu--Eisenstat state their
result for real matrices, the determinant-swap proof is unchanged over
\(\mathbb C\), with unitary matrices and adjoints replacing orthogonal
matrices and transposes.  Choosing the lexicographically first
admissible pair \((I,J)\) makes the selection measurable in \(Q\).
\end{proof}

We use the following consequence of Cook's super-regular-profile theorem
\cite[Theorem~1.24, equation~(1.28)]{Cook}.
\begin{lemma}[The deformed-square input used below]
\label{lem:local-cook-input}
Fix constants \(L<\infty\) and \(0<c_*\leq C_*<\infty\).  Let
\[
 Z_n=(a_{ij}\xi_{ij})_{i,j\leq n},
 \qquad c_*\leq a_{ij}\leq C_*,
\]
where the \(\xi_{ij}\) are independent copies of the atom in
Theorem~\ref{thm:main}.  If \(D_n\in\C^{n\times n}\) is deterministic
and \(\|D_n\|\leq n^L\), then there are constants
\(\beta_L,C_L,c_L>0\), depending only on
$L,K,c_*,C_*$, such that
\begin{equation}
 \P\{s_{\min}(Z_n+D_n)\leq n^{-\beta_L}\}
 \leq C_L\sqrt{\frac{\log n}{n}}+e^{-c_Ln}.
 \label{eq:local-cook-input}
\end{equation}
The same bound holds conditionally when \(D_n\) is measurable with
respect to an independent sigma-field and satisfies
$\|D_n\|\leq n^L$ almost surely.
\end{lemma}

\begin{proof}
After division by \(C_*\), the deterministic profile has entries in
\([0,1]\), and its threshold graph contains the complete bipartite graph
at a fixed threshold.  It is therefore uniformly super-regular in the
sense of \cite[Definition~1.23]{Cook}.  A centered variance-one fixed
subgaussian atom is spread with a fixed parameter.  Choose a fixed
\(\gamma>\max\{L,1/2\}\).  The cited theorem gives
\[
 \P\{s_{\min}(Z_n+D_n)\leq n^{-\beta},
       \ \|Z_n+D_n\|\leq n^\gamma\}
 \leq C\sqrt{\frac{\log n}{n}}.
\]
The standard subgaussian operator-norm tail
\cite[Theorem~4.4.5]{vershynin2019high}, together with
\(\|D_n\|\leq n^L\), bounds the omitted norm event by \(e^{-cn}\)
after increasing \(\gamma\) by a fixed amount.  This proves
\eqref{eq:local-cook-input}.
Here the numbering refers to arXiv:1608.07347v5; the theorem is titled
\emph{Matrix with super-regular profile}.
\end{proof}

\subsubsection{The full proof of Proposition~\ref{prop:local-terminal}}
For this proof only, we denote by
\(\widetilde\Delta:=\sigma_W\Delta\),
\(\widetilde Q_O:=\sigma_WQ_O\), and we define
\begin{equation}
 \widetilde H(\widetilde Q_O;\mathbf x)
 :=\widetilde\Delta(\mathbf x)+\zeta_WI_{3W}
   +\operatorname{Emb}_O(\widetilde Q_O),\quad
 \widetilde p_{\widetilde Q_O}:=\det\widetilde H(\widetilde Q_O),\quad
 \widetilde{\mathfrak C}(\widetilde Q_O)
 :=\|\Coeff\widetilde p_{\widetilde Q_O}\|_2.
\end{equation}
Since \(\widetilde H(\widetilde Q_O)=\sigma_WH(Q_O)\),
\begin{equation}
 p_{Q_O}=\sigma_W^{-3W}\widetilde p_{\widetilde Q_O},\qquad
 \mathfrak C(Q_O)
 =\sigma_W^{-3W}\widetilde{\mathfrak C}(\widetilde Q_O).
 \label{eq:local-terminal-scaling}
\end{equation}
Thus
every coefficient-to-value ratio, and every zero event, is unchanged.
\begin{proof}[Proof of Proposition~\ref{prop:local-terminal}]
By \eqref{eq:local-terminal-scaling}, the two logarithmic ratios and the
zero event are invariant under the scaling by $\sigma_W$.  Since
\(Q_O\mapsto\sigma_WQ_O\) is a bijection of
\(\C^{2W\times2W}\), it therefore suffices to prove the asserted bounds
for \(\widetilde p_{Q_O}\) and
\(\widetilde{\mathfrak C}(Q_O)\), with \(Q_O\) now denoting the
row-scaled deformation for the duration of this proof.

\textbf{Step 1: removal of $Q_O$ via linear algebra.}
Choose a large fixed \(K_0>K_z+C_{\rm RRQR}+10\), define
\(\tau_W:=W^{K_0}\), and
set
\begin{equation}
 r:=\#\{j:s_j(Q_O)>\tau_W\},\qquad n_{\rm res}:=3W-r.
 \label{eq:local-r-nres}
\end{equation}
Thus \(r\leq2W\) and \(n_{\rm res}\geq W\).  Apply
Lemma~\ref{lem:local-rrqr} to \(Q_O\), and extend its coordinate
decomposition by a zero block on the central coordinates.  We obtain
sets \(I,J\subset O\), \(|I|=|J|=r\), and an exact decomposition of the
deterministic matrix \(\operatorname{Emb}_O(Q_O)\) of the form
\eqref{eq:local-skeleton-identity}.  The rows of \(Y_{\mathrm{skel}}\) and
the columns of \(X_{\mathrm{skel}}\) indexed by \(C\) are zero.
When \(r=0\), the residual below is simply the
full matrix with a polynomial-norm deformation and the pivot matrix does not exist.  We write the remaining
formulas for \(r>0\), with this empty-pivot convention understood.

We absorb the deterministic shift into the row-scaled packet matrix:
\begin{equation}
 \Delta_z:=\widetilde\Delta+\zeta_WI_{3W}.
 \label{eq:local-Delta-tilde}
\end{equation}
Then we split \(\Delta_z\) in the same coordinates as $\Delta$:
\begin{equation}
 \Delta_z=
 \begin{pmatrix}(\Delta_z)_{11}&(\Delta_z)_{12}\\
                 (\Delta_z)_{21}&(\Delta_z)_{22}\end{pmatrix},
 \qquad
 K_\Delta:=K_{\mathrm{piv}}+(\Delta_z)_{11}.
 \label{eq:local-Delta-split}
\end{equation}
This is the pivot--residual split: \((\Delta_z)_{11}\) is the
\(r\times r\) pivot perturbation, whereas \((\Delta_z)_{22}\) is the
entire \(n_{\rm res}\times n_{\rm res}\) residual.  They are not the two
complete iid squares used below; those squares, denoted by \(Z_1,Z_2\),
will be selected inside the residual in \eqref{eq:local-two-cook-blocks}.

Let \(\mathcal F_{\rm piv}\) be generated by the packet entries occurring
in \((\Delta_z)_{11},(\Delta_z)_{12},(\Delta_z)_{21}\), and define
\begin{equation}
 \mathcal E_{\rm piv}:=
 \left\{\max_{(i,j)\in\{(1,1),(1,2),(2,1)\}}
 \|(\Delta_z)_{ij}\|\leq W^{C_\Delta}\right\},
 \qquad C_\Delta:=\max\{K_z,3\}+1.
 \label{eq:local-pivot-exposure-event}
\end{equation}
The standard subgaussian norm bound gives
\(\P(\mathcal E_{\rm piv}^c)\leq e^{-cW}\).  These variables lie outside
the residual, so \(\mathcal E_{\rm piv}\) is independent of the two
complete residual squares chosen below.  We use it only as a good event;
we do not condition those squares on it.  On \(\mathcal E_{\rm piv}\),
the choice of \(K_0\) gives
\(\|K_{\mathrm{piv}}^{-1}(\Delta_z)_{11}\|\leq1/2\), and hence
\begin{equation}
 |\det K_\Delta|
 \geq e^{-CW\log W}\prod_{j=1}^r s_j(Q_O).
 \label{eq:local-pivot-det}
\end{equation}
Indeed,
 \(\det K_\Delta=\det K_{\mathrm{piv}}
 \det(I+K_{\mathrm{piv}}^{-1}(\Delta_z)_{11})\);
the second determinant has modulus at least \(2^{-r}\), while strong
RRQR bounds the first by
\(W^{-C_{\rm RRQR}r}\prod_{j\leq r}s_j(Q_O)\).

To change the matrix to a further diagonal form, we denote by
\begin{align}
 G_{21}&:=(\Delta_z)_{21}
          -Y_{\mathrm{skel}}(\Delta_z)_{11},
 &G_{12}&:=(\Delta_z)_{12}
          -(\Delta_z)_{11}X_{\mathrm{skel}},
 \label{eq:local-Gs}\\
 F&:=E_0-Y_{\mathrm{skel}}(\Delta_z)_{12}
       -(\Delta_z)_{21}X_{\mathrm{skel}}\notag\\
  &\quad+Y_{\mathrm{skel}}(\Delta_z)_{11}X_{\mathrm{skel}}
       -G_{21}K_\Delta^{-1}G_{12}.
 \label{eq:local-F}
\end{align}
Here and below, an inverse is extended by zero on its singular set; this
is a Borel matrix-valued map.  Identities involving that inverse are used
only on the good event on which it is a genuine inverse.  In particular,
on \(\mathcal E_{\rm piv}\), direct multiplication gives the exact
factorization
\begin{align}
 \operatorname{Emb}_O(Q_O)+\Delta_z
 &={}
 \begin{pmatrix}
  I&0\\
  (Y_{\mathrm{skel}}K_{\mathrm{piv}}
    +(\Delta_z)_{21})K_\Delta^{-1}&I
 \end{pmatrix}
 \begin{pmatrix}K_\Delta&0\\0&(\Delta_z)_{22}+F\end{pmatrix}
 \notag\\
 &\quad\times
 \begin{pmatrix}
  I&K_\Delta^{-1}(K_{\mathrm{piv}}X_{\mathrm{skel}}
                    +(\Delta_z)_{12})\\
  0&I
 \end{pmatrix}.
 \label{eq:local-CUR}
\end{align}
The first
and third factors are determinant-one shears.  More importantly, every
unbounded occurrence of \(K_{\mathrm{piv}}\) cancels from \(F\).
Here is the promised norm check.  On \(\mathcal E_{\rm piv}\), the three
blocks of \(\Delta_z\) occurring in \eqref{eq:local-F} have operator norm
at most \(W^{C_\Delta}\).  The RRQR bounds (Lemma \ref{lem:local-rrqr}) and the
choice of \(K_0\) give, for fixed exponents independent of \(W\),
\begin{align}
 \|X_{\mathrm{skel}}\|+\|Y_{\mathrm{skel}}\|
 &\leq W^{C_{\rm RRQR}},&
 \|E_0\|&\leq W^{K_0+C_{\rm RRQR}},\notag\\
 \|K_\Delta^{-1}\|
 &\leq2\|K_{\mathrm{piv}}^{-1}\|
 \leq W^{C_{\rm RRQR}-K_0},&
 \|G_{12}\|+\|G_{21}\|
 &\leq W^{C_{\rm RRQR}+C_\Delta+1}.
 \label{eq:local-CUR-term-bounds}
\end{align}
Substituting these four estimates into the five terms of
\eqref{eq:local-F} shows, term by term, that on
\(\mathcal E_{\rm piv}\)
\begin{equation}
 \|F\|\leq W^{K_1}
 \label{eq:local-F-bound}
\end{equation}
for a fixed \(K_1=K_1(K_0,C_{\rm RRQR})\).  This eventwise bound will
show below that the global norm truncations are
inactive on \(\mathcal E_{\rm exp}\).

\textbf{Step 2: analyzing the remaining block via polynomial LSV.}
We first describe the structure of the matrix $(\Delta_z)_{22}+F$ in the original coordinates.
Let \(L_{\rm r},R_{\rm r}\) be the remaining outer row sets and
\(L_{\rm c},R_{\rm c}\) the remaining outer column sets, after removal of labels $I$ and $J$.  Denote by
\begin{equation}
 a=|L_{\rm r}|,\quad b=|R_{\rm r}|,\quad
 c=|L_{\rm c}|,\quad e=|R_{\rm c}|.
\end{equation}
Since equally many outer rows and columns were removed (no central rows
or columns are removed),
\begin{equation}
 a+b=c+e=:s,\qquad n_{\rm res}=W+s.
 \label{eq:local-counts}
\end{equation}
Choose an integer
\begin{equation}
 n_1\in[\max(a,c),W+\min(a,c)]
       \cap[n_{\rm res}/3-1,2n_{\rm res}/3+1].
 \label{eq:local-n1}
\end{equation}
The intersection is nonempty: its left endpoint is at most
\(s\leq2(W+s)/3=2n_{\rm res}/3\), while its right endpoint is at least
\(W\geq(W+s)/3=n_{\rm res}/3\), and integer rounding costs at most one.
Set \(n_2:=n_{\rm res}-n_1\), \(x:=n_1-a\), and \(y:=n_1-c\).  Select \(x\) central rows and \(y\) central
columns.  Together with the remaining left outer coordinates they form a
complete \(n_1\times n_1\) square.  Their complements, together with the
remaining right outer coordinates, form a complete
\(n_2\times n_2\) square.  Neither square meets a forbidden
rectangle of \eqref{eq:local-mask}; the first uses only \(L,C\) coordinates
and the second only \(R,C\) coordinates.  They are row-disjoint,
column-disjoint, and together cover every residual row and column.

Let \(\mathcal F_0\) be generated by all packet entries outside these two
squares.  Conditional on \(\mathcal F_0\), the two retained matrices
\(Z_1,Z_2\) are independent complete iid squares of respective dimensions
\(n_1,n_2\).  On \(\mathcal E_{\rm piv}\), write
\begin{equation}
 (\Delta_z)_{22}+F=
 \begin{pmatrix}
  Z_1+D_{11}&D_{12}\\D_{21}&Z_2+D_{22}
 \end{pmatrix},
 \label{eq:local-two-cook-blocks}
\end{equation}
where every \(D_{ij}\) is \(\mathcal F_0\)-measurable.  Extend the
\(D_{ij}\) measurably off \(\mathcal E_{\rm piv}\), for instance by zero.
By \eqref{eq:local-F-bound} and the subgaussian norm bound for the exposed
cross blocks, there is a fixed \(L_0\) such that the
\(\mathcal F_0\)-measurable event
\begin{equation}
 \mathcal E_{\rm exp}:=\mathcal E_{\rm piv}\cap
 \left\{\max_{i,j\leq2}\|D_{ij}\|\leq W^{L_0}\right\}
 \label{eq:local-full-exposure-event}
\end{equation}
satisfies \(\P(\mathcal E_{\rm exp}^c)\leq e^{-cW}\).

We now make the two conditional applications without conditioning on a
norm event involving \(Z_1\) or \(Z_2\).  For a matrix \(A\), let
\begin{equation}
 \mathcal T_R(A):=A\min\{1,R/\|A\|\},
 \qquad \mathcal T_R(0):=0,
 \label{eq:local-measurable-norm-truncation}
\end{equation}
and set \(\widehat D_{ij}:=\mathcal T_{W^{L_0}}(D_{ij})\).  These
truncated deformations are globally \(\mathcal F_0\)-measurable, have
polynomial norm, and agree with \(D_{ij}\) on
\(\mathcal E_{\rm exp}\).  Lemma~\ref{lem:local-cook-input}, applied
conditionally first to \(Z_1+\widehat D_{11}\), gives, for a fixed
\(\beta_1\), after defining
\[
 B_1:=Z_1+\widehat D_{11},
 \qquad \mathcal G_1:=\{s_{\min}(B_1)\geq W^{-\beta_1}\},
\]
\begin{equation}
 \P(\mathcal G_1^c\mid\mathcal F_0)
 \leq C\sqrt{\frac{\log W}{W}}+e^{-cW}.
 \label{eq:local-cook-one}
\end{equation}
Here both \(n_i\geq n_{\rm res}/3-1\geq W/3-1\), so, for all large
\(W\), \(W\leq4n_i\), while \(n_i\leq3W\).  Thus the estimates furnished
by Lemma~\ref{lem:local-cook-input} at size \(n_i\) imply the displayed
\(W\)-scale estimate: in particular, \(W^{L_0}\leq n_i^{L_0+1}\) for
large \(W\), and
\(e^{-cn_i}\leq e^{-cW/4}\), after which constants are renamed.

To make the second conditioning global, define the Borel truncated inverse
\begin{equation}
 \mathcal R_W(B):=
 \begin{cases}
  B^{-1},&s_{\min}(B)\geq W^{-\beta_1},\\
  0,&s_{\min}(B)<W^{-\beta_1}.
 \end{cases}
 \label{eq:local-truncated-inverse}
\end{equation}
It has norm at most \(W^{\beta_1}\).  After revealing \(Z_1\), define
the globally measurable Schur deformation and its corresponding square by
\begin{equation}
 \widehat S_2:=\widehat D_{22}
 -\widehat D_{21}\mathcal R_W(B_1)\widehat D_{12},
 \qquad \widehat B_2:=Z_2+\widehat S_2.
 \label{eq:local-second-schur}
\end{equation}
Conditional on \(\mathcal F_0\vee\sigma(Z_1)\), its deformation is
globally measurable, has fixed-polynomial norm, and is independent of the
untouched square \(Z_2\).  A second application of
Lemma~\ref{lem:local-cook-input} gives a fixed \(\beta_2\) such that, for
\(\mathcal G_2:=\{s_{\min}(\widehat B_2)\geq W^{-\beta_2}\}\),
\begin{equation}
 \P(\mathcal G_2^c\mid\mathcal F_0\vee\sigma(Z_1))
 \leq C\sqrt{\frac{\log W}{W}}+e^{-cW}.
 \label{eq:local-cook-two}
\end{equation}
To record the determinant identity without defining a Schur complement
only on \(\mathcal G_1\), introduce the artificial bottom block
\begin{equation}
 B_{22}^{\rm art}:=\widehat B_2+\widehat D_{21}
        \mathcal R_W(B_1)\widehat D_{12},
 \qquad
 M_{\rm art}:=
 \begin{pmatrix}B_1&\widehat D_{12}\\
                 \widehat D_{21}&B_{22}^{\rm art}\end{pmatrix}.
 \label{eq:local-artificial-bottom-block}
\end{equation}
On \(\mathcal G_1\), \(\mathcal R_W(B_1)=B_1^{-1}\), and Schur
complementation gives
\[
 \det M_{\rm art}=\det B_1\det\widehat B_2.
\]
On \(\mathcal E_{\rm exp}\), the four norm truncations are inactive;
cancellation in \eqref{eq:local-artificial-bottom-block} then shows that
\(M_{\rm art}\) equals the residual matrix in
\eqref{eq:local-two-cook-blocks}.  Therefore, on
\(\mathcal E_{\rm exp}\cap\mathcal G_1\cap\mathcal G_2\),
\[
 |\det((\Delta_z)_{22}+F)|
 \geq W^{-\beta_1n_1-\beta_2n_2}
 \geq e^{-CW\log W}.
\]
The union bound, \(\P(\mathcal E_{\rm exp}^c)\leq e^{-cW}\), and the two
Cook estimates consequently give, outside probability \(p_W\),
\begin{equation}
 |\det((\Delta_z)_{22}+F)|\geq e^{-CW\log W}.
 \label{eq:local-residual-det}
\end{equation}
Combining \eqref{eq:local-pivot-det}, \eqref{eq:local-CUR}, and
\eqref{eq:local-residual-det} gives
\begin{equation}
 |\widetilde p_{Q_O}(\Xi)|
 \geq e^{-CW\log W}\prod_{j=1}^r s_j(Q_O)
 \label{eq:local-value-lower}
\end{equation}
outside probability \(p_W\).

\textbf{Step 3: from product singular values to the coefficient norm.}
It remains to compare the right side with the coefficient norm.  First
set \(\zeta_W=0\), and denote the corresponding coefficient norm by
\(\widetilde{\mathfrak C}_0(Q_O)\).  An atom monomial is indexed by a
partial permutation \(\pi:I_0\to J_0\).  Its coefficient is
\begin{equation}
 \pm\left(\prod_{i\in I_0}a_{i,\pi(i)}\right)
 \det\bigl(\operatorname{Emb}_O(Q_O)\bigr)_{I_0^c,J_0^c}.
 \label{eq:local-partial-permutation}
\end{equation}
Every nonzero determinant here is a minor of \(Q_O\).  A
\(k\times k\) minor is at most \(\prod_{j\leq k}s_j(Q_O)\).
If \(k<r\), the omitted selected singular values are all larger than
\(\tau_W>1\).  If \(k>r\), the extra singular values are at most
\(\tau_W\) by \eqref{eq:local-r-nres}.  Thus this upper bound extracts only
the very large singular values of $Q_O$; those below
$\tau_W=W^{K_0}$ are absorbed into the polynomial factors.
Thus every coefficient in \eqref{eq:local-partial-permutation} is at
most
\[
 e^{CW\log W}\prod_{j=1}^r s_j(Q_O).
\]
The number of partial permutations of \(3W\) rows and columns is at most
\begin{equation}
 \sum_{t=0}^{3W}\binom{3W}{t}^2t!
 \leq 4^{3W}(3W)!
 \leq e^{CW\log W}.
 \label{eq:local-partial-permutation-count}
\end{equation}
Thus
\begin{equation}
 \widetilde{\mathfrak C}_0(Q_O)
 \leq e^{CW\log W}\prod_{j=1}^r s_j(Q_O).
 \label{eq:local-coeff-upper-zero}
\end{equation}
Conversely, leave precisely the pivot rows (I) and pivot columns (J)
unmatched, and choose a perfect matching in each of the two complete
residual squares constructed in Step~2.  The union of these matchings is
an atom monomial whose coefficient is the product of its deterministic
weights times
\(\det(Q_O)_{I,J}=\det K_{\mathrm{piv}}\).
Strong RRQR in Lemma \ref{lem:local-rrqr} therefore gives
\begin{equation}
 \widetilde{\mathfrak C}_0(Q_O)
 \geq e^{-CW\log W}\prod_{j=1}^r s_j(Q_O).
 \label{eq:local-coeff-lower-zero}
\end{equation}

Now we restore the shift into $\widetilde{\mathfrak C}_0(Q_O)$.  Translating each of the \(3W\) fresh diagonal
variables by \(\zeta_W/a_e\) sends the zero-shift polynomial to the
shifted polynomial.  On the two-dimensional coefficient space of one
multiaffine variable this is a triangular map; its inverse is the
opposite translation.  The norms of the product map and its inverse, generated by this scalar shift, are
at most
\[
 \prod_{e:\,e\text{ is diagonal}}(1+|\zeta_W|/a_e)
 \leq e^{CW\log W}.
\]
Consequently \(\widetilde{\mathfrak C}(Q_O)
=e^{\pm CW\log W}\widetilde{\mathfrak C}_0(Q_O)\), and
\begin{align}
 \widetilde{\mathfrak C}(Q_O)
 &\leq e^{CW\log W}\prod_{j=1}^r s_j(Q_O),
 \label{eq:local-coeff-upper}\\
 \widetilde{\mathfrak C}(Q_O)
 &\geq e^{-CW\log W}\prod_{j=1}^r s_j(Q_O).
 \label{eq:local-coeff-lower}
\end{align}
Equations \eqref{eq:local-value-lower} and
\eqref{eq:local-coeff-upper}, with the bad event charged at height \(T\),
prove \eqref{eq:local-terminal-capped}.  Dividing by \(T\) and letting
\(T\to\infty\) proves \eqref{eq:local-terminal-zero}.

Finally, there are at most \(e^{CW\log W}\) partial-permutation monomials.
On \(\mathcal E_{\max}\), every monomial has modulus at most
\((C_0\sqrt W)^{3W}\).
Cauchy--Schwarz over the coefficient vector gives
\eqref{eq:local-terminal-reverse}.  Centeredness and unit variance also
give the exact Parseval identity
\begin{equation}
 \bigl(\widetilde{\mathfrak C}(Q_O)\bigr)^2
 =\E|\widetilde p_{Q_O}(\Xi)|^2,
 \label{eq:local-parseval}
\end{equation}
since distinct multiaffine monomials are orthonormal in \(L^2\) norm.

By \eqref{eq:local-terminal-scaling}, we may now return to the original
notation $p_q$ and $\mathfrak C(q)$.  The preceding argument treated
deterministic \(Q_O\).  Let \(\mathcal G\)
be a sigma-field such that \(Q_O\) is \(\mathcal G\)-measurable and
\(\Xi\) is independent of \(\mathcal G\), and let \(\mu\) be the law of
\(\Xi\).  The maps \((q,x)\mapsto p_q(x)\) and
\(q\mapsto\mathfrak C(q)\) are Borel.  Thus, with
\[
 \Phi_T(q,x):=\min\left\{T,
   \log_+\frac{\mathfrak C(q)}{|p_q(x)|}\right\},
\]
interpreted as \(T\) when \(p_q(x)=0\), independence and Fubini give
\[
 \E[\Phi_T(Q_O,\Xi)\mid\mathcal G]
 =\int\Phi_T(Q_O,x)\,d\mu(x)
 \leq CW\log W+p_WT
 \quad\text{a.s.}
\]
The same argument with \(\mathbf1_{\{p_q(x)=0\}}\) proves the conditional
zero bound.  The reverse estimate is pointwise in \((q,x)\) on
\(\mathcal E_{\max}\), which is independent of \(\mathcal G\); this
proves the remaining conditional assertions.
\end{proof}

\subsubsection{The coefficient norm estimate}
\begin{proof}[\proofname\ of Lemma \ref{lem:local-all-minor}]
Set \(z=0\), and write $\mathfrak C_0(Q)$ for the resulting coefficient
norm.  A choice of atom entries in the determinant expansion of the
zero-shift polynomial $p_Q$ is a partial row--column matching in the mask
\eqref{eq:local-mask}.  Suppose that it leaves exactly the outer rows
\(I\subset O\) and outer columns \(J\subset O\) unmatched, where
\(|I|=|J|=k\).  Summing over the remaining \(Q\)-entries gives the minor
\(\det Q_{I,J}\).  Hence the coefficient of the corresponding atom
monomial is
\begin{equation}
 w_{\alpha}\det Q_{I,J},
 \label{eq:local-monomial-minor}
\end{equation}
where \(w_\alpha\) is the product of its deterministic entry weights, which are multiples of $\sigma_W^{-1}$.

Every pair \((I,J)\) occurs.  To see this, set
\[
 a:=|L\setminus I|,\quad b:=|R\setminus I|,\quad
 c:=|L\setminus J|,\quad e:=|R\setminus J|.
\]
These are the residual left/right row and column counts, and
$a+b=c+e$.  If
\(c\geq a\), add \(c-a\) central rows and no central columns to the left
subproblem; if \(c<a\), add no central rows and \(a-c\) central columns.
The resulting left-plus-central subproblem is square, and its
complementary right-plus-central subproblem is square as well.  Both squares lie
in the complete parts of the mask, so we can choose a perfect matching in
each.  Together they produce an atom monomial of the form
\eqref{eq:local-monomial-minor}.

For each pair \((I,J)\), the construction above gives at least one and at
most \(e^{CW\log W}\) matching monomials.  Each corresponding weight
product contains at most \(3W\) factors of size comparable to
\(\sigma_W^{-1}\), and hence lies between
\(e^{-CW\log W}\) and \(e^{CW\log W}\).  Squaring the coefficients and
summing first over the matching monomials and then over \((I,J)\) gives
\begin{equation}
 \mathfrak C_0(Q)^2
 =e^{\pm CW\log W}
  \sum_{k=0}^{2W}\ \sum_{|I|=|J|=k}|\det Q_{I,J}|^2,
 \label{eq:local-coefficients-to-minors}
\end{equation}
where \(\mathfrak C_0\) denotes the coefficient norm at zero shift.

We now use Cauchy--Binet.  Form the \(4W\times2W\) matrix
\[
 G(Q):=\begin{pmatrix}I_{2W}\\Q\end{pmatrix},
 \qquad G(Q)^*G(Q)=I_{2W}+Q^*Q.
\]
A \(2W\)-row minor of \(G(Q)\) selects identity rows indexed by
\(O\setminus J\) and \(Q\)-rows indexed by \(I\); its determinant is
\(\pm\det Q_{I,J}\).  Therefore
\begin{equation}
 \det(I_{2W}+Q^*Q)
 =\sum_{k=0}^{2W}\ \sum_{|I|=|J|=k}|\det Q_{I,J}|^2.
 \label{eq:local-all-minor-CB}
\end{equation}
Combining \eqref{eq:local-coefficients-to-minors} and
\eqref{eq:local-all-minor-CB} proves the lemma when \(z=0\).

Finally we take into account the nonzero shift.  A normalized diagonal entry is
\(\sigma_W^{-1}a_e x_e-z\), and
\[
 \sigma_W^{-1}a_e x_e-z
 =\sigma_W^{-1}a_e\bigl(x_e+\zeta_W/a_e\bigr).
\]
Thus translating the variable means composing the zero-shift polynomial
with the affine change \(x_e\mapsto x_e+\zeta_W/a_e\); it does not mean
altering a realized random entry.  On the two-dimensional coefficient
space spanned by \(1,x_e\), this change is a triangular \(2\times2\)
matrix, and the inverse change uses the opposite translation.  Taking the
tensor product over the \(3W\) diagonal variables costs at most
\(e^{O(W\log W)}\), because
\(|\zeta_W|\leq W^{O(1)}\) and \(|a_e|\geq c_*\).  This proves
\eqref{eq:local-all-minor} for the original shift.
\end{proof}

\subsection{Arbitrary-frame small ball probability estimate}\label{sec:local-penalty}
We will convert Theorem \ref{thm:local-complex-frame} as a corollary of Proposition \ref{prop:local-terminal}. The following linear algebra conversion procedure makes the deduction of Theorem \ref{thm:local-complex-frame} fairly precise.

We complete the $r$-frames $U,V:\C^r\to\C^{2W}$ to unitary matrices in the
form
\[
\widetilde U=(U,U_\perp),\quad \widetilde V=(V,V_\perp).
\]
Then we define the parameterized frame 
\begin{equation}\label{definetheframe}
\Theta_\lambda^{(r;U,V)}=\widetilde V\begin{pmatrix}
    \lambda I_r&0\\0&\lambda^{-1}I_{2W-r}
\end{pmatrix}\widetilde U^*,\quad \lambda>1.
\end{equation}
Suppress the parameters \((r;U,V)\), which are fixed for the moment.  We
introduce the frame vectors
\begin{equation}
 u_i:=\widetilde Ue_i,\qquad v_i:=\widetilde Ve_i,\qquad1\leq i\leq 2W.
\end{equation}
For a \(k\)-element set \(I=\{i_1<\cdots<i_k\}\), we write the wedge vector as
\begin{equation}
 u_I:=u_{i_1}\wedge\cdots\wedge u_{i_k},\quad
 v_I:=v_{i_1}\wedge\cdots\wedge v_{i_k},\quad
 a(I):=|I\cap\{1,\dots,r\}|.
\end{equation}
Since by definition we have \(\Theta_{\lambda}u_i=\lambda v_i\) for \(i\leq r\) and
\(\Theta_{\lambda}u_i=\lambda^{-1}v_i\) for \(i>r\), we have
\begin{equation}
 \boxed{
 \bigwedge^k\Theta_{\lambda}^{(r;U,V)}
 =\sum_{|I|=k}\lambda^{2a(I)-k}
   |v_I\rangle\langle u_I|.}
 \label{eq:local-wedge-expansion}
\end{equation}
The two families \((u_I)_{|I|=k}\) and \((v_I)_{|I|=k}\) are orthonormal
bases.  Therefore, the operator norm of this wedge product is the largest coefficient in
\eqref{eq:local-wedge-expansion}.

If \(k<r\), then \(a(I)\leq k\), so after multiplication by
\(\lambda^{-r}\) every exponent is at most \(k-r<0\).  If \(k>r\), then
\(a(I)\leq r\), so every exponent is at most \(r-k<0\).  Therefore
\begin{equation}
 \left\|\lambda^{-r}\bigwedge^k
 \Theta_{\lambda}^{(r;U,V)}\right\|_{\mathrm{op}}
 =\lambda^{-|k-r|}\longrightarrow0,\qquad k\neq r.
 \label{eq:local-other-k-limit}
\end{equation}
For \(k=r\), the exponent after division by \(\lambda^r\) is zero only for
\(I_0=\{1,\dots,r\}\); all other coefficients are at most \(\lambda^{-2}\).
Since \(u_{I_0}=\widehat U\) and \(v_{I_0}=\widehat V\),
\begin{equation}
 \left\|
  \lambda^{-r}\bigwedge^r\Theta_{\lambda}^{(r;U,V)}
  -|\widehat V\rangle\langle\widehat U|
 \right\|_{\mathrm{op}}\leq\lambda^{-2}.
 \label{eq:local-right-r-limit}
\end{equation}
Recall the definition of $\mathscr D_\Theta$ in
\eqref{eq:local-DTheta}.  Substituting \eqref{eq:local-other-k-limit}--
\eqref{eq:local-right-r-limit} into the finite sum
\eqref{eq:local-DTheta} and using
\(\tr(A|v\rangle\langle u|)=\langle u,Av\rangle\) gives coefficientwise
convergence (which means convergence of the coefficient of each monomial)
\begin{equation}
 \lambda^{-r}\mathscr D_{
    \Theta_{\lambda}^{(r;U,V)}}(\mathbf x)
 \longrightarrow(-1)^r Z_{U,V}^{(r)}(\mathbf x),\quad \lambda\to \infty.
 \label{eq:local-polynomial-limit}
\end{equation}

The limit is only taken with respect to $\lambda\to\infty$ but with $N$ and $W_N$ fixed.  Specifically, for each fixed \(N\) we fix
\(W_N,r_N,U_N,V_N\) and first let
\(\lambda\to\infty\) in the fixed finite-dimensional spaces
\(\bigwedge^k\C^{2W_N}\).  Only after the parameter \(\lambda\) has
disappeared do we take a large $N$.  

The same order of limit is enforced in the graph volume computations.  The singular values of
\(\Theta_{\lambda}^{(r;U,V)}\) are \(r\) copies of \(\lambda\) and
\(2W-r\) copies of \(\lambda^{-1}\), so exactly
\begin{equation}
 \lambda^{-r}\prod_{j=1}^{2W}
 \sqrt{1+s_j(\Theta_{\lambda}^{(r;U,V)})^2}
 =(1+\lambda^{-2})^{W}\longrightarrow1
 \label{eq:local-graph-limit}
\end{equation}
with \(2W\) fixed. 

\begin{proof}[\proofname\ of Theorem \ref{thm:local-complex-frame}]
For the coefficient norm $\Gamma_{U,V}^{(r)}$, we apply
Lemma~\ref{lem:local-boundary-volume} (which compares the coefficient norm
$\mathscr C(\Theta)$ with
$\prod_{j=1}^{2W}\sqrt{1+s_j(\Theta)^2}$) to the artificial relation
\(\Theta_{\lambda}^{(r;U,V)}\), divide by \(\lambda^r\), and use
\eqref{eq:local-polynomial-limit} to get $Z_{U,V}^{(r)}$ and \eqref{eq:local-graph-limit} for deterministic limit.  Since
the normalized polynomial space is finite-dimensional for fixed \(W\),
our coefficientwise convergence then implies convergence of coefficient norms:
\begin{equation}
 \lambda^{-r}\mathscr C
   (\Theta_{\lambda}^{(r;U,V)})
 \longrightarrow\Gamma_{U,V}^{(r)}.
 \label{eq:local-coeff-limit}
\end{equation}
This proves the coefficient two-sided estimate for $\Gamma_{U,V}^{(r)}$ in \eqref{eq:local-frame-coefficients}.

To prove the small ball probability estimate, we apply the uniform-in-\(\Theta\) estimate
\eqref{eq:local-boundary-capped} to
\(\Theta_{\lambda}^{(r;U,V)}\).  Then we divide both its polynomial and its
coefficient norm by \(\lambda^r\).  The bounded continuous function
\(\mathcal L_T\) from \eqref{eq:local-LT}, together with
the convergence in \eqref{eq:local-polynomial-limit} and \eqref{eq:local-coeff-limit}, permits
taking the \(\lambda\to\infty\) limit.  This proves the capped
logarithmic-loss estimate \eqref{eq:local-frame-capped}.
\end{proof}

\subsection{Linear algebra: the determinant of periodic block matrices}

\begin{proof}[\proofname\ of Lemma \ref{lem:local-block-floquet}]

Let \(H_{\rm cyc}\in\C^{mW\times mW}\) be the cyclic block matrix whose
equations (with all indices modulo \(m\)) are defined as
\begin{equation}
 C_j\psi_{j-1}+D_j\psi_j+B_j\psi_{j+1}=0,
 \qquad 1\leq j\leq m.
 \label{eq:local-floquet-recurrence}
\end{equation}

On the set where $B_j$ is invertible, define
\[
 T_j:=\begin{pmatrix}-B_j^{-1}D_j&-B_j^{-1}C_j\\I_W&0\end{pmatrix}.
\]
We also define the product matrix
\[
 M_{\rm cyc}:=T_mT_{m-1}\cdots T_1.
\]

We prove that, for some $\varepsilon_{m,W}\in\{\pm 1\}$,
\begin{equation}
 \det H_{\rm cyc}
=\varepsilon_{m,W}
\left(\prod_{j=1}^{m}\det B_j\right)
  \det(I_{2W}-M_{\rm cyc}).
\label{eq:local-block-floquet}
\end{equation}

\textbf{1.} We introduce separate auxiliary \(W\)-vector variables \(x_j,y_j\) and the
\(2W\)-state
\begin{equation}
 s_j:=\binom{x_j}{y_j},
 \qquad s_{m+1}:=s_1.
\end{equation}
Consider the \(2mW\times2mW\) cyclic first-order system
\begin{equation}
 \mathsf K_js_j+\mathsf L_js_{j+1}=0,
 \qquad
 \mathsf K_j:=
 \begin{pmatrix}D_j&C_j\\-I_W&0\end{pmatrix},
 \qquad
 \mathsf L_j:=
 \begin{pmatrix}B_j&0\\0&I_W\end{pmatrix}.
 \label{eq:local-floquet-augmented-system}
\end{equation}
Let \(\mathbb H_{\rm aug}\) denote its coefficient matrix.  Written in
components, its equations are
\begin{align}
 D_jx_j+C_jy_j+B_jx_{j+1}&=0,
 \label{eq:local-floquet-augmented-top}\\
 -x_j+y_{j+1}&=0.
 \label{eq:local-floquet-augmented-bottom}
\end{align}
The bottom equations give \(y_j=x_{j-1}\).  After permuting block rows and
columns so that the bottom equations and the \(y\)-variables are placed
last, their coefficient block is a permutation matrix.  Its determinant
has modulus one, and its Schur complement is exactly the original cyclic
system \eqref{eq:local-floquet-recurrence} in the variables
\(x_1,\dots,x_m\).  Consequently, we can compute
\begin{equation}
 \det\mathbb H_{\rm aug}=\varepsilon^{(1)}_{m,W}\det H_{\rm cyc},
 \qquad \varepsilon^{(1)}_{m,W}\in\{1,-1\}.
 \label{eq:local-floquet-first-elimination}
\end{equation}

\textbf{2.} We now eliminate the same augmented system in transfer order.  On the
invertible-\(B_j\) set,
\begin{equation}
 T_j=-\mathsf L_j^{-1}\mathsf K_j,
 \qquad
 \mathsf K_js_j+\mathsf L_js_{j+1}
 =\mathsf L_j(s_{j+1}-T_js_j).
\end{equation}
Factoring \(\mathsf L_j\) from the \(j\)-th \(2W\)-row block gives
\begin{equation}
 \det\mathbb H_{\rm aug}
 =\left(\prod_{j=1}^{m}\det\mathsf L_j\right)
   \det\mathbb D_T
 =\left(\prod_{j=1}^{m}\det B_j\right)
   \det\mathbb D_T,
 \label{eq:local-floquet-factor-L}
\end{equation}
where \(\mathbb D_T\) is the coefficient matrix of
\begin{equation}
 s_{j+1}-T_js_j=0,
 \qquad 1\leq j\leq m,\qquad s_{m+1}=s_1.
 \label{eq:local-floquet-transfer-system}
\end{equation}
Block Gaussian elimination with the identity coefficients in the first
\(m-1\) equations gives
\begin{equation}
 s_2=T_1s_1,\quad s_3=T_2T_1s_1,\quad\dots,\quad
 s_m=T_{m-1}\cdots T_1s_1.
\end{equation}
The last equation is therefore
\begin{equation}
 (I_{2W}-T_mT_{m-1}\cdots T_1)s_1=0.
\end{equation}
Equivalently, taking the Schur complement of those identity pivots,
\begin{equation}
 \det\mathbb D_T
 =\varepsilon^{(2)}_{m,W}\det(I_{2W}-M_{\rm cyc}),
 \qquad \varepsilon^{(2)}_{m,W}\in\{1,-1\}.
 \label{eq:local-floquet-second-elimination}
\end{equation}
Combining \eqref{eq:local-floquet-first-elimination},
\eqref{eq:local-floquet-factor-L}, and
\eqref{eq:local-floquet-second-elimination} proves
\eqref{eq:local-block-floquet}, after absorbing the two ordering signs
into \(\varepsilon_{m,W}\).

\textbf{3.} It remains to consider the case where some $B_i$ is singular.  For
\(0\leq r\leq2W\), we denote on the invertible-\(B_j\) set the wedge product
\begin{equation}
 \mathcal A_j^{(r)}:=(\det B_j)\bigwedge\nolimits^rT_j.
 \label{eq:local-floquet-cleared-one-step}
\end{equation}
Every entry of \(\mathcal A_j^{(r)}\) is a polynomial in
\(B_j,C_j,D_j\).  Indeed, after the identity rows of a minor of \(T_j\)
are removed, the remaining part is a minor of
\(B_j^{-1}(-D_j,-C_j)\); Cauchy--Binet followed by Jacobi's
complementary-minor identity shows that multiplication by \(\det B_j\)
clears it.  Exterior functoriality and
\(\det(I-A)=\sum_{r=0}^{2W}(-1)^r\operatorname{tr}(\bigwedge^rA)\)
now give
\begin{align}
 &\left(\prod_{j=1}^{m}\det B_j\right)
   \det(I_{2W}-M_{\rm cyc})
 \notag\\
 &\hspace{12mm}=
 \sum_{r=0}^{2W}(-1)^r
 \operatorname{tr}\left(
  \mathcal A_m^{(r)}\mathcal A_{m-1}^{(r)}\cdots
  \mathcal A_1^{(r)}
 \right).
 \label{eq:local-floquet-polynomial-extension}
\end{align}
The last expression is polynomial in all block entries.  It agrees with
\(\varepsilon_{m,W}^{-1}\det H_{\rm cyc}\) on the dense set on which all
\(B_j\) are invertible, and hence agrees everywhere.  Finally,
\(T_3T_2T_1=R_{\partial}\) and
\(T_m\cdots T_4=R_{\rm out}\), which proves
\eqref{eq:local-floquet-packet-split}.
\end{proof}

\subsection{From global determinant to local determinant}
\label{sec:local-double-elimination}

This section proves one linear algebra identity without probability
estimate.  The transfer construction of Section~\ref{sec:local-boundary}
produces the polynomial \(\mathscr D_{\Theta}\), whereas
Proposition~\ref{prop:local-terminal} controls coordinate determinants of
the form \(\det(\Delta-zI_{3W}+\operatorname{Emb}_O(Q))\).  The
proposition below identifies these two objects exactly.  Throughout this
section, \(\mathbf x\) denotes the full list of \(7W^2\) scalar variables
in the seven fresh packet blocks, with the normalization factors
\(\sigma_W^{-1}\) included in their deterministic coefficients.

Write the boundary map in the orders used in
\eqref{eq:local-general-Theta}: the source is \((\psi_+,\psi_R)\) and the
target is \((\psi_L,\psi_-)\).  Thus
\begin{equation}
 \Theta=
 \begin{pmatrix}\Theta_{11}&\Theta_{12}\\
                 \Theta_{21}&\Theta_{22}\end{pmatrix}.
 \label{eq:local-Theta-blocks}
\end{equation}
In components, this orientation means
\[
 \psi_L=\Theta_{11}\psi_++\Theta_{12}\psi_R,\qquad
 \psi_-=\Theta_{21}\psi_++\Theta_{22}\psi_R.
\]
Consider the following coordinate chart
\begin{equation}
 \mathcal U:=\{\Theta\in\GL_{2W}(\C):\det\Theta_{11}\neq0\}.
 \label{eq:local-U-domain}
\end{equation}

\begin{proposition}[Exact transfer--coordinate identity]
\label{prop:local-double-elimination}
Suppose first that \(\Theta\in\mathcal U\), and define
\begin{equation}
 S(\Theta):=
 \begin{pmatrix}
  \Theta_{21}\Theta_{11}^{-1}&
  \Theta_{22}-\Theta_{21}\Theta_{11}^{-1}\Theta_{12}\\
  \Theta_{11}^{-1}&-\Theta_{11}^{-1}\Theta_{12}
 \end{pmatrix}.
 \label{eq:local-STheta}
\end{equation}
We also define the diagonal matrix
\begin{equation}
 E:=\operatorname{diag}(C_L,B_R),
 \label{eq:local-E}
\end{equation}
and set
\begin{equation}
 Q_\Theta:=ES(\Theta).
\end{equation}
Then the coordinate terminal from
\eqref{eq:local-terminal-matrix} is
\begin{equation}
 H_{\Theta}(\mathbf x)
 :=H(Q_\Theta;\mathbf x)
 =\Delta(\mathbf x)-zI_{3W}+\operatorname{Emb}_O(Q_\Theta),
 \label{eq:local-HTheta}
\end{equation}
and the main identity is
\begin{equation}
 \boxed{
 \mathscr D_{\Theta}(\mathbf x)
 =\det\Theta_{11}\,
   p_{Q_\Theta}(\mathbf x)
 =\det\Theta_{11}\,\det H_{\Theta}(\mathbf x).}
 \label{eq:local-double-elimination}
\end{equation}

For every \(\Theta\in\GL_{2W}(\C)\), including those outside \(\mathcal U\),
define the following \(5W\times5W\) auxiliary matrix, with block columns
ordered as
\((\psi_L,\psi_C,\psi_R\mid\psi_-,\psi_+)\):
\begin{equation}
 K_{\Theta}(\mathbf x):=
 \left(
 \begin{array}{ccc|cc}
 D_L&B_L&0&C_L&0\\
 C_C&D_C&B_C&0&0\\
 0&C_R&D_R&0&B_R\\ \hline
 I_W&0&-\Theta_{12}&0&-\Theta_{11}\\
 0&0&-\Theta_{22}&I_W&-\Theta_{21}
 \end{array}
 \right).
 \label{eq:local-KTheta}
\end{equation}
Then
\begin{equation}
 \boxed{
 \det K_{\Theta}(\mathbf x)
 =\mathscr D_{\Theta}(\mathbf x).}
 \label{eq:local-K-first-elimination}
\end{equation}
\end{proposition}

Identity \eqref{eq:local-double-elimination} is the only output of this
section used in the local probability argument: it places
\(\mathscr D_\Theta\) directly within the scope of
Proposition~\ref{prop:local-terminal}.  The larger matrix
\(K_\Theta\) is only a proof device.  Its first three
block rows are the normalized packet equations and its last two are the boundary
relation.  Thus it keeps the three packet unknowns and two endpoint
unknowns long enough to eliminate them in either of two orders; no
probabilistic theorem is applied to this \(5W\times5W\) matrix.

\begin{proof}
We compute the determinant of \(K_{\Theta}\) in two
orders.

For the first computation, we initially suppose that \(B_L,B_C,B_R\) are
invertible.  Move the block columns
\((\psi_C,\psi_R,\psi_+)\) in front of
\((\psi_L,\psi_-)\).  There are four inversions of blocks of width \(W\),
so the permutation contributes \((-1)^{4W^2}=1\).  The first three pivot
columns are block lower triangular with diagonal blocks
\(B_L,B_C,B_R\).  Solving the corresponding three equations is exactly the
successive transfer calculation
\eqref{eq:local-Rpartial}.  The two remaining equations are written as
\[
 (I_{2W}-\Theta R_{\partial})
 \binom{\psi_L}{\psi_-}=0.
\]
Consequently, we can compute
\[
 \det K_{\Theta}
 =\det B_L\det B_C\det B_R\,
   \det(I_{2W}-\Theta R_{\partial})
 =\mathscr D_{\Theta},
\]
where the last equality is \eqref{eq:local-DTheta-det}.  Both extreme
expressions are polynomials in all packet and endpoint block entries (and,
after the endpoints are fixed, in the seven fresh blocks).  Since they
agree whenever the three right interfaces are invertible, polynomial
continuation makes them agree for all values of those interfaces.  This
proves \eqref{eq:local-K-first-elimination} without an invertibility
assumption.

For the second computation, suppose \(\Theta\in\mathcal U\).  The last two
block rows of \eqref{eq:local-KTheta} satisfy
\begin{align}
 \psi_+
 &=\Theta_{11}^{-1}(\psi_L-\Theta_{12}\psi_R),
 \label{eq:local-solve-plus}\\
 \psi_-
 &=\Theta_{21}\Theta_{11}^{-1}\psi_L
   +(\Theta_{22}-\Theta_{21}\Theta_{11}^{-1}\Theta_{12})\psi_R.
 \label{eq:local-solve-minus}
\end{align}
Equivalently, we can write
\begin{equation}
 \binom{\psi_-}{\psi_+}=S(\Theta)\binom{\psi_L}{\psi_R}.
 \label{eq:local-S-action}
\end{equation}
Substitute this relation into the first and third packet equations.  The
endpoint terms \((C_L\psi_-,B_R\psi_+)\) become
\(ES(\Theta)(\psi_L,\psi_R)^{\mathsf T}\); hence the remaining three
equations have coefficient matrix \(H_{\Theta}\).

Formally, this is the Schur complement of the endpoint block
\[
 G_{\Theta}:=
 \begin{pmatrix}0&-\Theta_{11}\\I_W&-\Theta_{21}\end{pmatrix}
\]
in \(K_{\Theta}\).  Since
\(\det G_{\Theta}=\det\Theta_{11}\), the Schur-complement determinant
formula gives
\[
 \det K_{\Theta}
 =\det\Theta_{11}\,\det H_{\Theta}.
\]
Combining this with \eqref{eq:local-K-first-elimination} proves
\eqref{eq:local-double-elimination}.
\end{proof}

\begin{corollary}[Temporary row scaling]
Let
\(\widetilde K_\Theta
 :=\operatorname{diag}(\sigma_WI_{3W},I_{2W})K_\Theta\).
Then the original normalized boundary polynomial and its coefficient norm
satisfy
\begin{equation}
 \mathscr D_{\Theta}
 =\sigma_W^{-3W}\det\widetilde K_\Theta,\qquad
 \mathscr C(\Theta)
 =\sigma_W^{-3W}
   \bigl\|\Coeff_{\mathbf x}\det\widetilde K_\Theta\bigr\|_{\ell^2}.
 \label{eq:local-boundary-scaling}
\end{equation}
\end{corollary}

\begin{proof}
The first \(3W\) rows acquire a factor \(\sigma_W\), while the last
\(2W\) boundary rows are unchanged.  Determinants and their complete
coefficient vectors therefore acquire the common factor
\(\sigma_W^{3W}\).
\end{proof}

The exponent is \(3W\), rather than the augmented dimension \(5W\),
because only the three packet equations are row-scaled.

\subsubsection{Removal of the endpoint factor \(E\)}
The coordinate polynomial contains the deformation
$Q_\Theta=ES(\Theta)$.  We next compare its exterior volume with that of
$S(\Theta)$ by proving \eqref{eq:local-remove-E}.

For every \(2W\times2W\) matrix \(A\), the compound-matrix identity is
\begin{equation}
 \boxed{\;
 \det(I_{2W}+A^*A)
 =\sum_{k=0}^{2W}\|\bigwedge\nolimits^kA\|_{\mathrm{HS}}^2.\;}
 \label{eq:local-compound-sum}
\end{equation}
Here is a direct derivation.  In the standard wedge bases, the entries of
\(\bigwedge^kA\) are the minors \(\det A_{I,J}\).  Therefore
\begin{equation}
 \|\bigwedge\nolimits^kA\|_{\mathrm{HS}}^2
 =\sum_{|I|=|J|=k}|\det A_{I,J}|^2.
 \label{eq:local-compound-minors}
\end{equation}
For each fixed \(J\), Cauchy--Binet identifies the sum over \(I\) with
\(\det((A^*A)_{J,J})\).  Summing over \(J\) and then over \(k\) gives the
principal-minor expansion of \(\det(I_{2W}+A^*A)\), proving
\eqref{eq:local-compound-sum}.  The identity is useful because one scalar
determinant records the squared Hilbert--Schmidt norms of all exterior
degrees at once.

From \eqref{eq:local-E-control}, the ordinary exterior-power bound, and
Hodge duality, on \(\mathcal E_{\partial}\) we get
\begin{equation}
 \|\bigwedge^kE\|+\|\bigwedge^kE^{-1}\|
 \leq e^{CW\log W},\qquad 0\leq k\leq 2W.
 \label{eq:local-E-compounds}
\end{equation}
Indeed,
\(\|\bigwedge^kE^{-1}\|
 =|\det E|^{-1}\|\bigwedge^{2W-k}E\|\).
Consequently, on \(\mathcal E_{\partial}\), for every \(S\),
\begin{align}
 \|\bigwedge\nolimits^k(ES)\|_{\mathrm{HS}}
 &\leq e^{CW\log W}
       \|\bigwedge\nolimits^kS\|_{\mathrm{HS}},\notag\\
 \|\bigwedge\nolimits^kS\|_{\mathrm{HS}}
 &\leq e^{CW\log W}
       \|\bigwedge\nolimits^k(ES)\|_{\mathrm{HS}}.
 \label{eq:local-ES-two-directions}
\end{align}
Square these inequalities, sum over \(k\), and use
\eqref{eq:local-compound-sum}.  This gives, on \(\mathcal E_{\partial}\),
\begin{equation}
 \det(I_{2W}+(ES)^*ES)^{1/2}
 =e^{\pm CW\log W}\det(I_{2W}+S^*S)^{1/2}.
 \label{eq:local-remove-E}
\end{equation}

\subsection{The coefficient norm for a boundary relation}

\begin{proof}[\proofname\ of Lemma \ref{lem:local-boundary-volume}]
First suppose \(\Theta\in\mathcal U\).  Parameterize its graph by the
source vector \((\psi_+,\psi_R)\).  In the reordered coordinates
\((\psi_L,\psi_R\mid\psi_-,\psi_+)\), the upper and lower graph frames are
\begin{equation}
 M(\Theta):=
 \begin{pmatrix}\Theta_{11}&\Theta_{12}\\0&I_W\end{pmatrix},
 \qquad
 N(\Theta):=
 \begin{pmatrix}\Theta_{21}&\Theta_{22}\\I_W&0\end{pmatrix}.
 \label{eq:local-MN}
\end{equation}
Direct multiplication gives
\(S(\Theta)=N(\Theta)M(\Theta)^{-1}\), and
\begin{equation}
 M(\Theta)^*\bigl(I_{2W}+S(\Theta)^*S(\Theta)\bigr)M(\Theta)
 =M(\Theta)^*M(\Theta)+N(\Theta)^*N(\Theta)
 =I_{2W}+\Theta^*\Theta.
 \label{eq:local-graph-Gram}
\end{equation}
Since \(\det M(\Theta)=\det\Theta_{11}\), taking determinants yields
\begin{equation}
 |\det\Theta_{11}|^2\det(I_{2W}+S(\Theta)^*S(\Theta))
 =\det(I_{2W}+\Theta^*\Theta).
 \label{eq:local-graph-identity}
\end{equation}
The exact identity \eqref{eq:local-double-elimination} implies
\[
 \mathscr C(\Theta)
 =|\det\Theta_{11}|\,\mathfrak C(ES(\Theta)).
\]
Apply Lemma~\ref{lem:local-all-minor},
\eqref{eq:local-remove-E}, and
\eqref{eq:local-graph-identity} in this order.  The factor
\(|\det\Theta_{11}|\) cancels exactly, and
\eqref{eq:local-coefficient-main} follows on \(\mathcal U\).

Now fix an arbitrary \(\Theta\in\GL_{2W}(\C)\), possibly with singular
\(\Theta_{11}\).  Let
\begin{equation}
 J_0:=\begin{pmatrix}I_W&0\\0&0\end{pmatrix},
 \qquad \Theta^{(n)}:=\Theta+\varepsilon_nJ_0,
 \label{eq:local-Theta-approx}
\end{equation}
where \(\varepsilon_n\to0\) is chosen so that
\(\det(\Theta_{11}+\varepsilon_nI_W)\neq0\).  Such a sequence exists
because this determinant is a monic polynomial in \(\varepsilon\) and has
only finitely many roots.  Since \(\GL_{2W}(\C)\) is open and
\(\Theta\in\GL_{2W}(\C)\), the sequence may also be chosen so that every
\(\Theta^{(n)}\) is invertible.  Thus \(\Theta^{(n)}\in\mathcal U\).

Each coefficient of \(\mathscr D_{\Theta}\) is polynomial in
the entries of \(\Theta\), because the entries of
\(\bigwedge^k\Theta\) are minors of \(\Theta\).  Hence
\(\mathscr C(\Theta^{(n)})\to
\mathscr C(\Theta)\).  The right-hand side of
\eqref{eq:local-coefficient-main} is continuous as well, so the estimate
passes to the limit.  Its lower bound is positive, and therefore the
limiting coefficient vector is never zero.  Notice that only
\(\mathscr D_\Theta\), its coefficient vector, and the estimate are
extended this way; \(S(\Theta)\) and \(H_\Theta\) remain defined only on
\(\mathcal U\).
\end{proof}

\begin{proof}[\proofname\ of Corollary \ref{corollary885}]
If \(\Theta\in\mathcal U\), set \(Q_\Theta=ES(\Theta)\).  By
\eqref{eq:local-double-elimination},
\[
 \mathscr D_\Theta=\det\Theta_{11}\,p_{Q_\Theta},
 \qquad
 \mathscr C(\Theta)=|\det\Theta_{11}|\,\mathfrak C(Q_\Theta).
\]
The common factor cancels from both logarithmic ratios, so
Proposition~\ref{prop:local-terminal} proves both assertions on
\(\mathcal U\), uniformly in \(\Theta\).

For arbitrary \(\Theta\in\GL_{2W}(\C)\), choose
\[
 \Theta^{(n)}=\Theta+\varepsilon_n
 \begin{pmatrix}I_W&0\\0&0\end{pmatrix}\in\mathcal U,
 \qquad \varepsilon_n\to0.
\]
Such a sequence exists because
\(\det(\Theta_{11}+\varepsilon I_W)\) has only finitely many zeros and
\(\GL_{2W}(\C)\) is open.  Since every coefficient of
\(\mathscr D_\Theta\) is polynomial in \(\Theta\),
\[
 \mathscr D_{\Theta^{(n)}}(\mathbf x)\to\mathscr D_\Theta(\mathbf x),
 \qquad
 \mathscr C(\Theta^{(n)})\to\mathscr C(\Theta).
\]
Lemma~\ref{lem:local-boundary-volume} gives
\(\mathscr C(\Theta)>0\).  Hence bounded convergence, applied to the
bounded continuous function \(\mathcal L_T\), proves the capped estimate;
continuity of
\((c,w)\mapsto\log_+(|w|/c)\) proves the reverse estimate.  Uniformity is
preserved because all preceding constants are independent of \(\Theta\).
\end{proof}

\section{The bounded-density case: circular law for growing bandwidth}
\label{sec:density-extension}

This section proves Theorem~\ref{thm:density-main}, which is a continuous-entry companion to
Theorem~\ref{thm:main}.  It reuses many linear algebra computations from the previous Bernoulli case. Only the probabilistic modules are replaced. Since the entries have a density, we can achieve any $W\to\infty$ without $W\gg\log N$.

The proof is therefore much shorter, and does not depend on the heavy anti-concentration estimates in the no-density setting.

\begin{proposition}[Finite-third-moment high-band input for full-block rings]
\label{prop:density-block-high-band}
Let $X_N$ be the cyclic full-block matrix \eqref{eq:model}, with $N=mW$,
and assume that $\xi$ is centered and normalized by $\E|\xi|^2=1$.
Assume either that $\xi$ is real with density bounded by $L$, or that it is
complex with a density, bounded by $L$, with respect to planar Lebesgue
measure on $\C\simeq\R^2$.  Alternative~\textnormal{(iii)} of
Assumption~\ref{ass:indicator-density} is also admissible.  Suppose that
$\E|\xi|^3<\infty$.  For every fixed
$\omega\in(0,1/9)$, if $N,W\to\infty$ and
\begin{equation}
 W\ge N^{8/9+\omega},
 \label{eq:density-block-high-band}
\end{equation}
then, for every fixed $z\in\C$,
\begin{equation}
 \frac1N\log|\det(X_N-zI_N)|
 \xrightarrow{\Prob}U_{\cir}(z).
 \label{eq:density-block-high-band-logdet}
\end{equation}
\end{proposition}

\begin{proof}
Index scalar coordinates by $(j,a)\in\mathbb T_m\times[W]$.  The variance
profile is
\begin{equation}
 \sigma_{(j,a),(k,b)}^2
 =\frac1{3W}\,\1_{\{k-j\in\{-1,0,1\}\}}.
 \label{eq:density-block-profile}
\end{equation}
It is doubly stochastic, $\max_{u,v}\sigma_{uv}^2=(3W)^{-1}$, and
$\sigma_{uv}^2=(3W)^{-1}$ whenever the scalar cyclic distance between
$u$ and $v$ is at most $W$.  Hence
Theorem~\ref{thm:high-band-lsv} applies, and
$\mathfrak b(X_N)=3W$.

Set $\beta=8/9+\omega$ and choose $\chi,\kappa,\tau>0$ so that
\begin{equation}
 \frac12+\chi<\beta,\qquad \kappa<\frac\chi4,\qquad
 \tau<\frac\beta8,\qquad
 \tau+3\kappa<\frac{9\beta}{8}-1.
 \label{eq:density-block-parameters}
\end{equation}
Apply Theorem~\ref{thm:high-band-lsv} with $t=N^{-2}$, remove its
Hilbert--Schmidt cutoff by \eqref{eq:lsv-remove-HS-cutoff}, and apply
Proposition~\ref{prop:mesoscopic-counting} at
\[
 a_N:=(3W)^{-1/8}N^\tau=o(1).
\]
With probability $1-o(1)$,
\begin{align*}
 s_{\min}(X_N-zI_N)&\ge e^{-L_N},\\
 \#\{j:s_j(X_N-zI_N)\le a_N\}&\le C_zN a_N,\\
 L_N&:=\frac{N^{1+3\kappa}}W+2\log N.
\end{align*}
Consequently,
\begin{equation}
 \begin{aligned}
 \frac1N\sum_{s_j(X_N-zI_N)\le a_N}|\log s_j(X_N-zI_N)|
 &\le C_za_NL_N\\
 &\le C_z\left(N^{1+\tau+3\kappa}W^{-9/8}
       +N^\tau W^{-1/8}\log N\right)=o(1).
 \end{aligned}
 \label{eq:density-block-hard-edge}
\end{equation}
where \eqref{eq:density-block-parameters} and $W\ge N^\beta$ were used.

Above $a_N$, Lemma~\ref{lem:local-finite-moment-bulk} and integration by
parts give the same truncated-log comparison with normalized Ginibre as
\eqref{eq:truncated-log-bulk-comparison}.  The upper tail is controlled by
the Hilbert--Schmidt estimate \eqref{eq:upper-log-tail-second-moment}, and
the Ginibre hard edge by \eqref{eq:ginibre-hard-edge}.  Letting first
$N\to\infty$ and then the upper cutoff tend to infinity proves
\eqref{eq:density-block-high-band-logdet}.
\end{proof}

\subsection{Local analytical smoothness estimates}

\subsubsection{Small-scale analysis for multi-affine logarithms}
\label{subsec:dens-affine-log}

The next lemma is the elementary continuous input used both for
row-resampling and for evaluating a terminal coefficient tensor.  Its
proof treats the whole affine coefficient vector.  In particular, it
does not select a largest monomial of a high-degree determinant.

\begin{lemma}[Scale-free concentration of random affine logarithms]
\label{lem:dens-affine-log}
Let \(E\) be a finite-dimensional real or complex normed space and let
\(x_1,\ldots,x_p\) be independent centered, variance-one random
variables.  Assume either that they are real and have densities bounded by
\(L\), or that they are complex, \(E\) is a complex normed space, and
their densities with respect to planar Lebesgue measure are bounded by
\(L\).  For fixed vectors
\(v_0,v_1,\ldots,v_p\in E\), consider
\[
 G(\mathbf{x})=v_0+\sum_{s=1}^p x_sv_s.
\]
Let \(\mathbf{x}'\) be an independent copy of \(\mathbf{x}\).  If \(G\) is not the zero function, then
\begin{equation}
 \E\left|\log\|G(\mathbf{x})\|-\log\|G(\mathbf{x}')\|\right|^2
 \leq C_L\log^2(ep).
 \label{eq:dens-affine-resampling}
\end{equation}
The constant $C_L$ is independent of \(\dim E\), of the norm on \(E\), and of
the deterministic center \(v_0\).

If we denote by
\begin{equation}
 \rho^2:=\|v_0\|^2+\sum_{s=1}^p\|v_s\|^2,
 \label{eq:dens-affine-rho}
\end{equation}
then, provided \(\rho>0\), we have a more specific estimate
\begin{equation}
 \E\left|\log\|G(\mathbf{x})\|-\log\rho\right|^2
 \leq C_L\log^2(ep).
 \label{eq:dens-affine-Hilbert}
\end{equation}
The same conclusions hold under alternative~\textnormal{(iii)} of
Assumption~\ref{ass:indicator-density}.
\end{lemma}
The important part of this lemma is that the estimate does not depend on the size of $v_0,v_1,\cdots,v_p$. The proof of Lemma \ref{lem:dens-affine-log} is deferred to Section \ref{remainingsection3841}.

\begin{corollary}[Evaluation of a row-multiaffine coefficient tensor]
\label{cor:dens-coefficient-evaluation}
Let \(P\) be a nonzero complex polynomial which is affine separately in
\(n_{\rm grp}\) independent groups, each containing at most \(p\) atoms satisfying
the assumptions of Lemma~\ref{lem:dens-affine-log}.  More explicitly,
write
\(\mathbf{x}_j=(x_{j,1},\ldots,x_{j,p_j})\), where the groups are mutually
independent and $p_j\le p$, and assume that, for every $j$,
\[
 P=Q_{j,0}+\sum_{s=1}^{p_j}x_{j,s}Q_{j,s}.
\]
Here the coefficients $Q_{j,s}$ are independent of $\mathbf{x}_j$ but may
depend on all the other groups.
Denote by
\(\|\Coeff P\|_2\) the Euclidean norm of its full coefficient tensor in
the monomial basis.  For complex atoms this means the monomial basis in the
formal complex atom variables.  Then
\begin{equation}
 \E\left|\log|P|-\log\|\Coeff P\|_2\right|
 \leq C_Ln_{\rm grp}\log(ep).
 \label{eq:dens-tensor-evaluation}
\end{equation}
In particular, \(P\neq0\) almost surely.
\end{corollary}

\subsubsection{Row-group multiaffinity}
\label{subsec:dens-row-affinity}

Recall the block transfer matrix $T_j$. We
write, for each $1\le j\le m$,
\[
 \mathsf R_j:=\bigl[\,\mathsf A_j-zI_W\ \ \mathsf C_j\,\bigr].
\]
For \(1\leq a\leq W\), define the \(a\)-th equation-row group at site
\(j\) by
\begin{equation}
 \mathfrak g_{j,a}:=\bigl(
  (\mathsf B_j)_{a,*},(\mathsf A_j-zI_W)_{a,*},
  (\mathsf C_j)_{a,*}\bigr),
 \label{eq:dens-row-group}
\end{equation}
where \((\cdot)_{a,*}\) denotes the entire \(a\)-th row.  The symbol
\(\mathfrak g_{j,a}\) denotes the \(3W\) underlying variance-one atom
coordinates generating these normalized entries; the diagonal coordinate
of \(\mathsf A_j-zI_W\) also contains the deterministic shift \(-z\).
There are \(W\) such row groups at each block site in both the real and
complex cases.

\begin{lemma}[Separate affinity in every $\mathfrak g_{j,a}$]
\label{lem:dens-row-affinity}
For every \(0\leq r\leq2W\), every entry of
\[
 \mathcal A_j^{(r)}=(\det\mathsf B_j)\bigwedge^rT_j
\]
is affine-linear separately in each group \(\mathfrak g_{j,a}\) (see Corollary \ref{cor:dens-coefficient-evaluation}).  Hence,
after all other rows are fixed, every product containing
\(\mathcal A_j^{(r)}\) is an operator-valued affine function of the
\(3W\) atoms in \(\mathfrak g_{j,a}\); that is, its coefficients are
matrices depending on \(\mathfrak g_{k,b}\) for
\((k,b)\neq(j,a)\).
\end{lemma}

\begin{proof}
An entry of \(\bigwedge^rT_j\) is a minor of \(T_j\).  Write
\[
 T_j=
 \begin{bmatrix}
  -\mathsf B_j^{-1}(\mathsf A_j-zI_W)&
  -\mathsf B_j^{-1}\mathsf C_j\\
  I_W&0
 \end{bmatrix}
 =\begin{bmatrix}-\mathsf B_j^{-1}\mathsf R_j\\ I_W\quad0\end{bmatrix}.
\]
We expand this minor of \(T_j\)
according to the rows taken from the deterministic lower block
\([I_W\ 0]\).  A nonconstant term left in the upper block is, up to a
sign, a minor of \(\mathsf B_j^{-1}\mathsf R_j\).  Thus each entry of the
cleared exterior matrix \(\mathcal A_j^{(r)}\) is either \(0\),
\(\pm\det\mathsf B_j\), or
\(\pm\det\mathsf B_j\det((\mathsf B_j^{-1}\mathsf R_j)_{I,K})\).
For subsets
\(I,K\) of the same size, Cauchy--Binet gives
\begin{align}
 &(\det\mathsf B_j)
 \det\bigl((\mathsf B_j^{-1}\mathsf R_j)_{I,K}\bigr)
 \notag\\
 &\qquad=
 \sum_{\substack{J\subset[W]\\|J|=|I|}}
  \pm\det(\mathsf B_j)_{J^c,I^c}
       \det(\mathsf R_j)_{J,K}.
 \label{eq:dens-row-CB-Jacobi}
\end{align}
Indeed, Jacobi's complementary-minor identity changes
\((\det\mathsf B_j)\det((\mathsf B_j^{-1})_{I,J})\) into the displayed
complementary minor of \(\mathsf B_j\).  In every summand of
\eqref{eq:dens-row-CB-Jacobi}, if $a\in J$ then the $R_j$-minor appears while the $B_j$ minor does not appear. If $a\notin J$ then instead the $B_j$-minor appears while the $R_j$ minor does not appear, and thus a row $a$ never appears in both factors of the same summand.  Since the determinant is linear, each summand is affine in the full row group $\mathfrak{g}_{j,a}$. Expanding \((\mathsf A_j-zI_W)_{aa}\) can replace its random
part by the constant \(-z\), but cannot create a second occurrence of the
row.  This proves separate affinity.  Multiplication on the left and
right by matrices independent of \(\mathfrak g_{j,a}\) preserves this separate affinity.
\end{proof}

\subsubsection{Row-group Efron--Stein concentration}
\label{subsec:dens-Efron-Stein}

For an interval \(I\) of block sites, let
\begin{equation}
 Y_r(I):=\log\left\|
   \prod_{j\in I}^{\leftarrow}\mathcal A_j^{(r)}
 \right\|,
 \qquad 0\leq r\leq2W.
 \label{eq:dens-Yr-I}
\end{equation}
Every interface block \(\mathsf B_j\) and \(\mathsf C_j\) is invertible
almost surely under the density hypothesis.  For a planar density this
follows from joint absolute
continuity; under the directional alternative, condition on the orthogonal
coordinates and use that the determinant is a nonzero polynomial in the
remaining continuously distributed coordinates.  Thus the cleared
exterior operators and their products are invertible almost surely, and
their logarithms are finite.

\begin{lemma}[Concentration over row groups]
\label{lem:dens-row-concentration}
If an interval \(I\) contains \(s\) consecutive block sites, then, uniformly in
\(0\leq r\leq2W\), we have the variance estimate
\begin{equation}
 \operatorname{Var}Y_r(I)
 \leq C_LsW\log^2(eW).
 \label{eq:dens-row-variance}
\end{equation}
Consequently, we have the maximal deviation estimate
\begin{equation}
 \E\max_{0\leq r\leq2W}
 |Y_r(I)-\E Y_r(I)|
 \leq C_L\sqrt{W(sW)}\log(eW).
 \label{eq:dens-all-degree-concentration}
\end{equation}
\end{lemma}

\begin{proof} The idea is to apply Efron--Stein to $Y_r(I)$ viewed as a
function of the $sW$ independent row groups $\mathfrak{g}_{j,a}$, and use
Lemma~\ref{lem:dens-affine-log} after each resampling.
Resample one group \(\mathfrak g_{j,a}\) and condition on all other row
groups.  By Lemma~\ref{lem:dens-row-affinity}, the full exterior
product is an operator-valued affine function of at most \(3W\) atoms.
For almost every realization of the conditioned rows this affine map is
not the constant-zero map.  Indeed, the joint event on which every
\(\mathsf B_i\) and \(\mathsf C_i\) is invertible has probability one.
Fubini therefore implies that, for almost every conditioning, the same
event has conditional probability one with respect to the resampled row.
On it every cleared exterior factor, and hence the full product in every
degree, is invertible.  Thus the conditioned affine map is nonzero (if
all slopes vanish, it is a fixed nonzero operator).  The exceptional
conditioned configurations may be discarded.
Apply Lemma~\ref{lem:dens-affine-log} in the finite-dimensional operator
space equipped with its operator norm.  Let \(\mu_{j,a}\) be the
coefficient scale \(\rho\) from \eqref{eq:dens-affine-rho} for this
conditioned affine map.  The old and resampled row groups share the same
\(\mu_{j,a}\).  Thus, if \(Y\) and \(Y'\) are the corresponding
logarithmic norms,
\[
 |Y-Y'|^2
 \leq
 2|Y-\log\mu_{j,a}|^2
 2|Y'-\log\mu_{j,a}|^2,
\]
and \eqref{eq:dens-affine-Hilbert} bounds the conditional expectation by
\(C_L\log^2(eW)\), uniformly in all conditioned left and right products.
There are \(sW\) row groups.
Efron--Stein therefore gives \eqref{eq:dens-row-variance}.

For complete formal justification, first apply Efron--Stein to the
one-Lipschitz truncation \((-T)\vee(Y_r\wedge T)\).  The integrated Hodge
bound proved below implies \(Y_r(I)\in L^2\) for every finite \(I,W\), so
one may let \(T\to\infty\).

Finally, by Cauchy--Schwarz over \(0\leq r\leq2W\),
\[
 \max_r|Y_r-\E Y_r|
 \leq\left(\sum_{r=0}^{2W}|Y_r-\E Y_r|^2\right)^{1/2}.
\]
Take expectations, use Jensen's inequality and
\eqref{eq:dens-row-variance}, and obtain
\eqref{eq:dens-all-degree-concentration}.
\end{proof}

\subsection{Integrated interface and Hodge bounds}
\label{subsec:dens-integrated-Hodge}

The next result replaces the event \(\mathcal G_j\) used in the discrete
proof.

\begin{lemma}[One-site integrated Hodge control]
\label{lem:dens-integrated-Hodge}
For every fixed \(z\), there is an integrable random variable
\(\mathfrak h_j=\mathfrak h_j(z)\), depending only on the three
normalized random blocks at site \(j\), such that
\begin{equation}
 \max_{0\leq r\leq2W}
 \left\{
  \log_+\|\mathcal A_j^{(r)}\|+
  \log_+\|(\mathcal A_j^{(r)})^{-1}\|
 \right\}
 \leq\mathfrak h_j,
 \qquad
 \E\mathfrak h_j\leq C_{L,z}W\log(eW).
 \label{eq:dens-one-site-Hodge-L1}
\end{equation}
The same variables have finite second moments.  Hence, for any interval
\(J\) consisting of \(s\) consecutive block sites (and therefore having
scalar length \(sW\)) and every integer \(0\leq r\leq2W\),
\begin{equation}
 \E\log_+\left\|
  \prod_{j\in J}^{\leftarrow}\mathcal A_j^{(r)}
 \right\|
 +
 \E\log_+\left\|
  \left(\prod_{j\in J}^{\leftarrow}\mathcal A_j^{(r)}
  \right)^{-1}\right\|
 \leq C_{L,z}sW\log(eW).
 \label{eq:dens-interval-Hodge-L1}
\end{equation}
Throughout, \(C_{L,z}\) depends only on the density bound \(L\) and the
complex shift \(z\in\C\).
\end{lemma}
The proof of Lemma \ref{lem:dens-integrated-Hodge} is deferred to Section \ref{remainingsection3841}.
\subsection{From the log determinant to maximal exterior growth}
\label{subsec:dens-reset-seam}
We use the same notation of in and out cells in Section \ref{sec:local-terminal}.
After a cyclic relabelling, reserve the three consecutive sites
\[
L=1,\qquad C=2,\qquad R=3,
\]
write \(-=m\) and \(+=4\) for their outside neighbours, and call
\(\{4,\ldots,m\}\) the outside arc.  The two transfers have opposite
orientations:
\begin{align}
R_{\rm in}:=R_\partial=T_RT_CT_L,
&\qquad
\binom{\psi_+}{\psi_R}
=R_{\rm in}\binom{\psi_L}{\psi_-},
\label{eq:dens-Rin-direction}\\
R_{\rm out}:=T_mT_{m-1}\cdots T_4,
&\qquad
\binom{\psi_L}{\psi_-}
=R_{\rm out}\binom{\psi_+}{\psi_R}.
\label{eq:dens-Rout-direction}
\end{align}
Thus the cyclic compatibility equation is
\[
(I_{2W}-R_{\rm out}R_{\rm in})
\binom{\psi_L}{\psi_-}=0.
\]

The test matrix \(\Theta\) in the local theorem has the same orientation
as \(R_{\rm out}\):
\begin{equation}
\binom{\psi_L}{\psi_-}
=\Theta\binom{\psi_+}{\psi_R}.
\label{eq:dens-Theta-direction}
\end{equation}
During the local argument \(\Theta\) is deterministic.  The actual
periodic closure is recovered by substituting
\(\Theta=R_{\rm out}\).

In the packet notation below, as in
Section~\ref{sec:local-terminal}, we suppress the sans-serif font and
write \(A_j,B_j,C_j\) for the normalized blocks
\(\mathsf A_j,\mathsf B_j,\mathsf C_j\).
We denote the two boundary matrices by
\[
E_\partial:=\operatorname{diag}(C_L,B_R),
\]
and let \(\mathbf x\) denote the atom coordinates in the remaining seven
packet blocks.
For fixed \((\Theta,E_\partial)\), define
\begin{align}
P_{\Theta,E_\partial}(\mathbf x)
&:=
\sum_{r=0}^{2W}(-1)^r
\operatorname{tr}\!\left(
\mathcal Q^{(r)}(E_\partial;\mathbf x)
\bigwedge\nolimits^r\Theta\right),
\label{eq:dens-explicit-boundary-polynomial}\\
\mathscr C(\Theta;E_\partial)
&:=
\bigl\|\operatorname{Coeff}_{\mathbf x}
P_{\Theta,E_\partial}\bigr\|_2,
\label{eq:dens-explicit-boundary-coefficients}
\end{align}
where
\[
\mathcal Q^{(r)}(E_\partial;\mathbf x)
:=
b_{\rm in}(E_\partial;\mathbf x)
\bigwedge\nolimits^rR_{\rm in}(E_\partial;\mathbf x),
\]
and
\[
b_{\rm in}(E_\partial;\mathbf x)
:=
\det B_L(\mathbf x)\,
\det B_C(\mathbf x)\,
\det B_R.
\]
On the companion-invertibility set,
\begin{equation}
P_{\Theta,E_\partial}(\mathbf x)
=
b_{\rm in}(E_\partial;\mathbf x)
\det\!\left(
I_{2W}-\Theta R_{\rm in}(E_\partial;\mathbf x)
\right).
\label{eq:dens-explicit-boundary-determinant}
\end{equation}
Both sides are polynomials in \(\mathbf x\), so this identity extends
everywhere by polynomial continuation.

Define the following filtration,
\[
\mathcal F_{\rm out}
:=
\sigma\{(\mathsf C_j,\mathsf A_j,\mathsf B_j):4\leq j\leq m\},
\qquad
c_{\rm out}:=\prod_{j=4}^{m}\det\mathsf B_j.
\]
Then \(R_{\rm out}\) and \(c_{\rm out}\) are
\(\mathcal F_{\rm out}\)-measurable, whereas the nine packet blocks are
independent of \(\mathcal F_{\rm out}\).  If \(\Xi_7\) denotes the random
evaluation of the seven-block family \(\mathbf x\), containing \(7W^2\)
scalar variables, the
block Floquet identity gives
\begin{equation}
\det(X_N-zI_N)
=
\varepsilon\,c_{\rm out}\,
P_{R_{\rm out},E_\partial}(\Xi_7),
\label{eq:dens-cut-exterior-sum}
\end{equation}
where \(|\varepsilon|=1\).  Equivalently,
\begin{align*}
P_{R_{\rm out},E_\partial}(\Xi_7)
&=
\sum_{r=0}^{2W}(-1)^r
\operatorname{tr}\!\left(
\mathcal Q^{(r)}(E_\partial;\Xi_7)
\bigwedge\nolimits^rR_{\rm out}\right)\\
&=
b_{\rm in}(E_\partial;\Xi_7)
\det\!\left(
I_{2W}-R_{\rm out}
R_{\rm in}(E_\partial;\Xi_7)\right).
\end{align*}

For every integrable function \(F\) of the nine packet blocks, define
\begin{equation}
 \E_{\rm pkt}F
 :=
 \E_{E_\partial}\!\left[
  \E_{\Xi_7}(F\mid E_\partial)
 \right].
 \label{eq:dens-packet-expectation}
\end{equation}
Thus \(\E_{\rm pkt}\) first integrates the seven internal blocks
conditionally on the two endpoint blocks and then integrates the
endpoints.

The main result of this section is the following:

\begin{proposition}
Define the outside exterior pressure by
\begin{equation}
\Pi_{\rm out}
:=
\max_{0\leq r\leq2W}
\log\left\|
c_{\rm out}\bigwedge\nolimits^rR_{\rm out}
\right\|
=
\log|c_{\rm out}|
+
\max_{0\leq r\leq2W}
\log\left\|\bigwedge\nolimits^rR_{\rm out}\right\|.
\label{eq:dens-outside-pressure}
\end{equation}
Then, conditionally on the outside arc,
\begin{equation}
\boxed{
\E_{\rm pkt}\left[
\left|
\log|\det(X_N-zI_N)|-\Pi_{\rm out}
\right|
\,\middle|\,
\mathcal F_{\rm out}
\right]
\leq C_{L,z}W\log(eW),
}
\label{eq:dens-periodic-seam-comparison}
\end{equation}
where $C_{L,z}$ is a constant depending only on $L$ and $z$.
\end{proposition}
\begin{proof}
This follows by combining Proposition~\ref{prop:dens-reset-seam},
Proposition~\ref{prop:dens-packet-coeff-evaluation},
\eqref{eq:local-boundary-volume-vs-exterior}, and
\eqref{eq:dens-cut-exterior-sum}, with \(\Theta=R_{\rm out}\).
\end{proof}
We then list all the technical estimates used in this section.
\begin{proposition}[From polynomial coefficient to plane wedge]
\label{prop:dens-reset-seam}
Let the three packet sites be independent of a possibly random boundary
relation \(\Theta\in\GL_{2W}(\C)\).  Then
\begin{equation}
 \E_{E_\partial}\left[
  \left|
   \log|\mathscr C(\Theta;E_\partial)|
   -\frac12\log\det(I_{2W}+\Theta^*\Theta)
  \right|\ \middle|\ \Theta
 \right]
 \leq C_{L,z}W\log(eW).
 \label{eq:dens-seam-L1}
\end{equation}
\end{proposition}
The proof of Proposition \ref{prop:dens-reset-seam} is deferred to Section \ref{remainingsection3841}.

\begin{proposition}[Conditional evaluation of the packet polynomial]
\label{prop:dens-packet-coeff-evaluation2}
Fix \(\Theta\in\GL_{2W}(\C)\), and condition on arbitrary invertible
endpoint blocks
\[
 E_\partial=\operatorname{diag}(C_L,B_R).
\]
Let
\[
 P_{\Theta,E_\partial}(\mathbf x)
 :=\mathscr D_\Theta(E_\partial,\mathbf x),
 \qquad
 \mathscr C(\Theta;E_\partial)
 :=\|\Coeff_{\mathbf x}P_{\Theta,E_\partial}\|_2.
\]
If \(\Xi_7\) denotes the random evaluation of the seven blocks
\[
 A_L,B_L,C_C,A_C,B_C,C_R,A_R,
\]
then, uniformly over the conditioned endpoint values, one has
\begin{equation}
 \E_{\Xi_7}\left|
  \log|P_{\Theta,E_\partial}(\Xi_7)|
  -\log\mathscr C(\Theta;E_\partial)
 \right|
 \leq C_LW\log(eW).
 \label{eq:dens-pathwise-boundary-volume2}
\end{equation}
\end{proposition}
The proof of Proposition \ref{prop:dens-packet-coeff-evaluation2} is deferred to Section \ref{remainingsection3841}.

\begin{proposition}\label{propositions6.8}
Let \(E_\partial=\operatorname{diag}(C_L,B_R)\) consist of the two random
endpoint blocks, and let \(\mathbf x\) denote the formal seven-block
variables.  Recall the three transfer matrices defined in
\eqref{eq:local-three-companions}, and write
\[
 R_\partial(E_\partial;\mathbf x;z):=T_RT_CT_L,
 \qquad
 b_\partial(E_\partial;\mathbf x)
 :=\det B_L\det B_C\det B_R,
\]
and, for \(0\leq r\leq2W\), define
\[
 \mathcal Q^{(r)}(E_\partial;\mathbf x;z)
 :=
 b_\partial(E_\partial;\mathbf x)
 \bigwedge\nolimits^r
 R_\partial(E_\partial;\mathbf x;z).
\]
We use the packet expectation \(\E_{\rm pkt}\) defined in
\eqref{eq:dens-packet-expectation}.
Then for every
\(0\leq r\leq2W\) and every deterministic or outside-measurable
decomposable unit wedges \(u,v\),
\begin{equation}
 \E_{\rm pkt}\left[
  \log_+\frac1{|\langle u,\mathcal Q^{(r)}v\rangle|}
  \ \middle|\ u,v
 \right]
 \leq C_{L,z}W\log(eW).
 \label{eq:dens-reset-L1}
\end{equation}
\end{proposition}
The proof of Proposition \ref{propositions6.8} is deferred to Section \ref{remainingsection3841}.

\subsection{The deterministic limit of the transfer operator via a short-cycle circular law}
\label{subsec:dens-anchor}
The aim is to determine the expected logarithmic core pressures
\(F_{W,r}\) from the already-proved short-cycle circular law.  This
subsection serves the same purpose as
Section~\ref{todetermine5.5.1}, where no density was assumed.

Here we only need to use a single
short ring as the anchor:
\begin{equation}
 s_W:=\lceil W^{1/200}\rceil,\qquad
 c_W:=s_W+3,\qquad
 \ell_W:=c_WW=W^{201/200+o(1)}.
 \label{eq:dens-cell-scale}
\end{equation}
For a core of \(s_W\) consecutive block sites, we define the following short-cycle products
\begin{equation}
 Y_{W,r}:=\log\left\|
  \prod_{1\le j\le s_W}^{\leftarrow}\mathcal A_j^{(r)}
 \right\|,\qquad
 F_{W,r}(z):=\E Y_{W,r},
 \label{eq:dens-core-pressure}
\end{equation}
where the product inside the norm is
\(\mathcal A_{s_W}^{(r)}\mathcal A_{s_W-1}^{(r)}
\cdots\mathcal A_1^{(r)}\).
These expectations are finite by
Lemma~\ref{lem:dens-integrated-Hodge}.

Take an independent cyclic ring of total matrix dimension \(\ell_W\), reserve its
last three block sites (of size $3W$) as a reset packet \eqref{eq:gb-cell-reset} (that give rise to $R_\partial$ and $b_\partial$), and use the other \(s_W\)
sites as the core \eqref{eq:gb-cell-core} (that give $R_{\text{out}}$ and $c_\text{out}$).  The proof-level unregularized logarithmic-determinant
input is now supplied by
Proposition~\ref{prop:density-block-high-band}.  Indeed, fix
$\omega_*=1/20$.  Since
$W=\ell_W^{200/201+o(1)}$ and
$200/201>8/9+\omega_*$, that proposition gives, under the assumptions of
Theorem~\ref{thm:density-main},
\begin{equation}
 \frac1{\ell_W}\log|\det(X_{\ell_W}-zI_{\ell_W})|
 \xrightarrow{\Prob}U_{\cir}(z)
 \label{eq:dens-anchor-input}
\end{equation}
for every fixed \(z\).

We remove the terminal packet by
\eqref{eq:dens-periodic-seam-comparison} and apply
Lemma~\ref{lem:dens-row-concentration} to obtain
\begin{align}
 \log|\det(X_{\ell_W}-zI_{\ell_W})|
 &=\max_{0\leq r\leq2W}F_{W,r}(z)
 \notag\\
 &\quad+O_{L^1}\!\left(
  W\log(eW)+\sqrt{W\ell_W}\log(eW)
 \right).
 \label{eq:dens-anchor-pressure-comparison}
\end{align}
The error divided by \(\ell_W\) is
\begin{equation}
 O\bigl(W^{-1/200}\log(eW)+W^{-1/400}\log(eW)\bigr)=o(1).
 \label{eq:dens-anchor-errors}
\end{equation}
Since the middle quantity in
\eqref{eq:dens-anchor-pressure-comparison} is deterministic,
\eqref{eq:dens-anchor-input} identifies the pressure:
\begin{equation}
 \boxed{
 \frac1{\ell_W}\max_{0\leq r\leq2W}F_{W,r}(z)
 \longrightarrow U_{\cir}(z).}
 \label{eq:dens-pressure-calibration}
\end{equation}
For further use, we choose the deterministic degree
\begin{equation}
 r_*(W,z):=\min\arg\max_{0\leq r\leq2W}F_{W,r}(z).
 \label{eq:dens-pressure-degree}
\end{equation}

\subsection{Mean stitching by an integrated reset}
\label{subsec:dens-mean-stitching}

A complete cell begins with three reset sites and ends with an independent
\(s_W\)-site core.  Use the definitions
\eqref{eq:gb-cell-reset}--\eqref{eq:gb-cell-product}, but no clipping and
no event \(H_k\).  Expose a core and choose a top singular pair
\(\widehat w_k,\widehat u_k\) for its degree-\(r\) cleared exterior
operator, exactly as in \eqref{eq:gb-top-singular-pair}.  If
\(\widehat v_{k-1,r}\) is the past-measurable incoming decomposable unit
wedge from \eqref{eq:gb-recursive-wedge}, then
\begin{align}
 \|\mathscr R_{{\rm core},k}^{(r)}
       \mathscr R_{{\rm reset},k}^{(r)}\widehat v_{k-1,r}\|
 &\geq
 \|\mathscr R_{{\rm core},k}^{(r)}\|
 \left|\left\langle
  \widehat w_k,
  \mathscr R_{{\rm reset},k}^{(r)}\widehat v_{k-1,r}
 \right\rangle\right|.
 \label{eq:dens-reset-test}
\end{align}
Proposition~\ref{propositions6.8} with $\mathcal Q_k^{(r)}=\mathscr R_{\text{reset},k}^{(r)}$
therefore gives
\[
 \E\left[
  \left.
  \log_+\frac{1}{
   |\langle\widehat w_k,\mathcal Q_k^{(r)}
   \widehat v_{k-1,r}\rangle|}
  \,\right|\,\text{past and core }k
 \right]
 \leq C_{L,z}W\log(eW).
\]
Here both vectors are decomposable unit wedges: \(\widehat w_k\) is the
top right singular wedge of the core, and
\(\widehat v_{k-1,r}\) is the incoming wedge from the preceding cells.
This implies the conditional
lower bound
\begin{equation}
 \E\left[
  \log\|\mathscr R_{{\rm core},k}^{(r)}
       \mathscr R_{{\rm reset},k}^{(r)}\widehat v_{k-1,r}\|
  \ \middle|\ \text{past and core }k
 \right]
 \geq
 \log\|\mathscr R_{{\rm core},k}^{(r)}\|
 -C_{L,z}W\log(eW).
 \label{eq:dens-conditional-reset}
\end{equation}
The complementary upper estimate follows from submultiplicativity and
Lemma~\ref{lem:dens-integrated-Hodge}:
\begin{equation}
 \E\log\|\mathscr M_{{\rm cell},k}^{(r)}\|
 \leq F_{W,r}(z)+C_{L,z}W\log(eW).
 \label{eq:dens-cell-upper-mean}
\end{equation}

  For any integer $K_c\in\mathbb{N}$ and \(K_c\) complete cells,
write \(\mathscr M_{1:K_c}^{(r)}\) as in \eqref{productcellus1} for their chronological product.  Then taking the expectation over the past and core $k$ for the previous two bounds,
\begin{align}
 \E\log\|\mathscr M_{1:K_c}^{(r_*)}\|
 &\geq K_c F_{W,r_*}-CK_cW\log(eW),
 \label{eq:dens-mean-lower}\\
 \E\log\|\mathscr M_{1:K_c}^{(r)}\|
 &\leq K_c F_{W,r}+CK_cW\log(eW),
 \qquad0\leq r\leq2W.
 \label{eq:dens-mean-upper}
\end{align}
The first line uses a fixed initial decomposable wedge and the iteratively defined wedge vectors as in \eqref{eq:gb-recursive-wedge} and the fact that
the operator norm dominates its image.  The same deterministic degree
\(r_*\) is used in every cell; stationarity is what makes this possible.

Now apply Lemma~\ref{lem:dens-row-concentration} once to the entire
\(K_c c_W\)-site product.  Write
\[
 F_*:=\max_{0\le r\le 2W}F_{W,r},\qquad
 X_r:=\log\|\mathscr M_{1:K_c}^{(r)}\|,\qquad
 X:=\max_{0\le r\le 2W}X_r.
\]
The fixed-degree concentration bound and
\eqref{eq:dens-mean-lower} give
\begin{align*}
 X&\geq X_{r_*}
 \geq\E X_{r_*}
      -O_{\Prob}\!\left(\sqrt{W(K_c\ell_W)}\log(eW)\right)\\
 &\geq K_cF_*-O_{\Prob}\!\left(
  K_cW\log(eW)+\sqrt{W(K_c\ell_W)}\log(eW)\right).
\end{align*}
In the other direction, the all-degree concentration bound and
\eqref{eq:dens-mean-upper} give
\begin{align*}
 X&\leq\max_r\E X_r+\max_r|X_r-\E X_r|\\
 &\leq K_cF_*+O_{\Prob}\!\left(
  K_cW\log(eW)+\sqrt{W(K_c\ell_W)}\log(eW)\right).
\end{align*}
Combining the two inequalities yields
\begin{align}
 \max_{0\leq r\leq2W}\log\|\mathscr M_{1:K_c}^{(r)}\|
 &=K_c\max_rF_{W,r}
 \notag\\
 &\quad+O_{\Prob}\left(
  K_cW\log(eW)+\sqrt{W(K_c\ell_W)}\log(eW)
 \right).
 \label{eq:dens-global-pressure-product}
\end{align}
This is the continuous analogue of \eqref{eq:gb-pressure-sandwich}.  It contains no
clipped cell variables.
\subsubsection{The error incurred in the last remaining interval}
Let an outside arc contain \(m_{\rm out}\) block sites, and write
\[
 m_{\rm out}=K_c c_W+q,\qquad 0\leq q<c_W.
\]
The first \(K_c c_W\) sites form complete cells, and the remaining
interval \(\mathcal R\) has \(q\) sites.  For \(0\leq r\leq2W\), set
\begin{equation}
 R_{\rm rem}^{(r)}
 :=
 \prod_{j\in\mathcal R}^{\leftarrow}\mathcal A_j^{(r)}.
 \label{eq:dens-remainder-product}
\end{equation}
On the probability-one event on which all interface blocks
\(\mathsf B_j,\mathsf C_j\) are invertible, every
\(R_{\rm rem}^{(r)}\) is invertible.  Let
\[
 H_{\mathcal R}:=\sum_{j\in\mathcal R}\mathfrak h_j,
\]
where the variables \(\mathfrak h_j\) are given by
Lemma~\ref{lem:dens-integrated-Hodge}.  Submultiplicativity, applied both
to the product and to its inverse, gives simultaneously for all exterior
degrees
\begin{equation}
 \max_{0\leq r\leq2W}
 \left\{
  \log_+\|R_{\rm rem}^{(r)}\|
  +
  \log_+\|(R_{\rm rem}^{(r)})^{-1}\|
 \right\}
 \leq 2H_{\mathcal R}.
 \label{eq:dens-remainder-Hodge}
\end{equation}

Let \(A^{(r)}\) denote the product of the \(K_c\) complete cells, before
the remaining interval \(\mathcal R\) is attached.  For every
\(0\leq r\leq2W\), the pathwise inequalities
\begin{equation}
 -\log\|(R_{\rm rem}^{(r)})^{-1}\|
 \leq
 \log\|R_{\rm rem}^{(r)}A^{(r)}\|-\log\|A^{(r)}\|
 \leq
 \log\|R_{\rm rem}^{(r)}\|
 \label{eq:dens-remainder-pathwise}
\end{equation}
imply
\[
 \left|
  \log\|R_{\rm rem}^{(r)}A^{(r)}\|
  -\log\|A^{(r)}\|
 \right|
 \leq 2H_{\mathcal R}.
\]
Consequently, the error caused by the remaining interval is upper bounded by
\begin{align}
 \Delta_{\rm rem}
 &:=
 \left|
  \max_{0\leq r\leq2W}
   \log\|R_{\rm rem}^{(r)}A^{(r)}\|
  -
  \max_{0\leq r\leq2W}
   \log\|A^{(r)}\|
 \right|
 \notag\\
 &\leq 2H_{\mathcal R}.
 \label{eq:dens-remainder-uniform}
\end{align}
Notice that no union bound over \(r\) is needed here.

By Lemma~\ref{lem:dens-integrated-Hodge} and \(q=|\mathcal{R}|<c_W\), the first moment of the error is
\begin{equation}
 \E H_{\mathcal R}
 \leq C_{L,z}qW\log(eW)
 \leq C_{L,z}\ell_W\log(eW).
 \label{eq:dens-remainder-L1}
\end{equation}
Therefore, for every fixed \(\varepsilon>0\), Markov's inequality gives
\begin{align}
 \Prob\left\{\frac{\Delta_{\rm rem}}{N}>\varepsilon\right\}
 &\leq
 \frac{2\E H_{\mathcal R}}{\varepsilon N}
 \notag\\
 &\leq
 \frac{C_{L,z}\ell_W\log(eW)}{\varepsilon N}.
 \label{eq:dens-remainder-Markov}
\end{align}
In the long branch \(N>W^{101/100}\), while
\(\ell_W=c_WW\leq C W^{201/200}\).  Hence the right-hand side of
\eqref{eq:dens-remainder-Markov} is bounded by
\[
 C_{L,z,\varepsilon}W^{-1/200}\log(eW)=o(1).
\]
Thus the unnormalized remainder cost is
\[
 \Delta_{\rm rem}
 =
 O_{\Prob}\bigl(\ell_W\log(eW)\bigr),
\]
and, after division by \(N\),
\begin{equation}\label{equations259}\frac{\Delta_{\rm rem}}{N}=o_{\Prob}(1).
\end{equation}

\subsection{Target ring, direct branch, and completion}
\label{subsec:dens-target}
We can now complete the proof of circular law, primarily for smaller bandwidth.
Suppose first that
\begin{equation}
 N>W^{101/100}.
 \label{eq:dens-long-branch}
\end{equation}
Reserve the final three block sites for the terminal seam.  Split the
remaining sites into \(K_N\) complete \(c_W\)-site cells and one remainder
of fewer than \(c_W\) sites, where
\begin{equation}
 K_N:=\left\lfloor\frac{m-3}{c_W}\right\rfloor,
 \qquad
 q_N:=m-3-K_Nc_W,
 \qquad 0\leq q_N<c_W.
 \label{eq:dens-target-cell-count}
\end{equation}
In particular,
\begin{equation}
 \frac{K_N\ell_W}{N}
 =\frac{K_Nc_W}{m}
 =1-\frac{3+q_N}{m}
 \longrightarrow1.
 \label{eq:dens-target-cell-ratio}
\end{equation}
Combining
\eqref{eq:dens-periodic-seam-comparison},
\eqref{eq:dens-global-pressure-product},
\eqref{equations259}, and
\eqref{eq:dens-pressure-calibration} gives, for any $z\in\mathbb{C}$,
\begin{equation}
 \frac1N\log|\det(X_N-zI_N)|
 \xrightarrow{\Prob}U_{\cir}(z).
 \label{eq:dens-long-logdet}
\end{equation}
The explicit normalized errors, apart from the anchor input, are
\begin{equation}
 O_{\Prob}\left(
  \frac{W\log(eW)}{\ell_W}
  +\sqrt{\frac WN}\log(eW)
  +\frac{\ell_W\log(eW)}N
  +\frac{W\log(eW)}N
  +W^{-1/400}\log(eW)
 \right),
 \label{eq:dens-target-errors}
\end{equation}
which are \(o_{\Prob}(1)\) under
\eqref{eq:dens-long-branch}.  The last displayed term is the
short-anchor all-degree fluctuation from
\eqref{eq:dens-anchor-errors}; the nonquantitative \(o_{\Prob}(1)\) from
the high-band anchor input remains separate.

For the large bandwidth regime (where circular law is already known)
\begin{equation}
 N\leq W^{101/100},
 \label{eq:dens-short-branch}
\end{equation}
Proposition~\ref{prop:density-block-high-band} applies directly to the
target: indeed, $W\ge N^{100/101}$ and
$100/101>8/9+\omega_*$.  It gives
\begin{equation}
 \frac1N\log|\det(X_N-zI_N)|
 \xrightarrow{\Prob}U_{\cir}(z).
 \label{eq:dens-direct-input}
\end{equation}
The two branches cover every sequence of $(W,N)$ with \(W\to\infty\).
Taking \(z=0\) in either
\eqref{eq:dens-long-logdet} or \eqref{eq:dens-direct-input} gives
\(\Prob\{\det X_N=0\}\to0\), because \(U_{\cir}(0)=-1/2\) is finite
whereas the normalized log determinant is \(-\infty\) when \(X_N\) is
singular.

Finally, the following in-probability convergence of the Hilbert-Schmidt norm
\begin{equation}
 \frac1N\|X_N\|_{\rm HS}^2
 =\frac1{3WN}
  \sum_{\text{displayed atoms}}|\xi_{ij}|^2
 \xrightarrow{\Prob}1
 \label{eq:dens-HS-tightness}
\end{equation}
uses only \(\E|\xi|^2=1\).  The in-probability Tao--Vu replacement principle
\cite[Theorem 2.1]{TaoVuKrishnapur2010}, together with \eqref{eq:dens-long-logdet} and
\eqref{eq:dens-direct-input}, completes the proof of
Theorem~\ref{thm:density-main}.

Only Proposition~\ref{prop:density-block-high-band} uses the third
moment.  The lifting, least-singular-value cutoff, and Hilbert--Schmidt
tightness use only the second moment and the density bound.  In
particular, the Jain--Jana--Luh--O'Rourke input and all subgaussian norm
estimates are absent from this branch.

\subsection{Remaining technical proofs}
\label{remainingsection3841}
We collect here the proofs of the remaining technical estimates.
\begin{proof}[\proofname\ of Lemma \ref{lem:dens-affine-log}]
Set \(b=\|v_0\|\) and \(\nu=\max_{s\geq1}\|v_s\|\).  If \(\nu=0\),
the result is immediate.  Otherwise, we denote by
\begin{equation}
 \lambda:=\max\{b,p\nu\}.
 \label{eq:dens-affine-lambda}
\end{equation}
We first show
\begin{equation}
 \E\left|\log\|G(\mathbf x)\|-\log\lambda\right|^2
 \leq C_L\log^2(ep).
 \label{eq:dens-affine-centered}
\end{equation}

For the upper tail, \(\E|x_s|\leq1\), and hence
\[
 \E\|G(\mathbf x)\|\leq b+\sum_s\|v_s\|\leq2\lambda.
\]
Markov's inequality gives
\begin{equation}
 \Prob\{\|G(\mathbf x)\|>e^u\lambda\}\leq2e^{-u},
 \qquad u\geq0,
 \label{eq:dens-affine-upper-tail}
\end{equation}
which has a bounded logarithmic second moment.

For the lower tail, choose a label \(s_*\) with
\(\|v_{s_*}\|=\nu\), and use Hahn--Banach to choose a norm-one scalar
functional \(\varphi\) with \(|\varphi(v_{s_*})|=\nu\).  After
conditioning on the other variables, the real small-ball estimate is
\[
 \Prob\{\|G\|\leq u\lambda\}
 \leq \min\{1,2Lu\lambda/\nu\},
\]
while integration of the planar density over a disk gives, in the complex
case,
\[
 \Prob\{\|G\|\leq u\lambda\}
 \leq \min\{1,\pi L(u\lambda/\nu)^2\}.
\]
Under the directional alternative, first condition on all orthogonal
coordinates and then use the first bound for the remaining coordinate.
All bounds are uniform in the conditioned center.  Letting the
small-ball radius decrease to zero also gives
\(\Prob\{G(\mathbf x)=0\}=0\).

If \(\|v_0\|\leq2p\nu\), then \(\lambda\leq2p\nu\), so we have
\begin{equation}
 \Prob\{\|G\|\leq u\lambda\}
 \leq\min\{1,C_L(pu)^d\},
 \label{eq:dens-affine-small-center}
\end{equation}
where \(d=1\) in the real and directional cases and \(d=2\) in the
planar-density case.  Integrating with \(u=e^{-t}\) gives a contribution
\(O_L(\log^2(ep))\).

Suppose now that \(b=\|v_0\|>2p\nu\), so \(\lambda=b\), and set
\(\delta:=p\nu/b<1/2\).  Write \(S=\sum_sx_sv_s\).  By the triangle
inequality and Cauchy--Schwarz,
\begin{equation}
 \E\|S\|^2
 \leq \E\left(\sum_s|x_s|\|v_s\|\right)^2
 \leq p\sum_s\|v_s\|^2\leq p^2\nu^2.
 \label{eq:dens-affine-noise-second}
\end{equation}
For \(0<u\leq1/2\), the event \(\|v_0+S\|\leq ub\) implies
\(\|S\|\geq b/2\).  Combining Markov's inequality with the preceding
one-coordinate density estimate yields
\begin{equation}
 \Prob\{\|G\|\leq ub\}
 \leq\min\left\{4\delta^2,
 C_L\left(\frac{pu}{\delta}\right)^d\right\}.
 \label{eq:dens-affine-large-center}
\end{equation}
Split the logarithmic integral at
\[
 t_0=\log\left(\frac{C_Lp}{\delta^{1+2/d}}\right).
\]
Its two pieces are bounded by
\[
 C\delta^2t_0^2+C\delta^2(t_0+1)
 \leq C_L\log^2(ep),
\]
uniformly over \(0<\delta<1/2\).  This proves
\eqref{eq:dens-affine-centered}.  Applying it to
\(\mathbf x\) and \(\mathbf x'\) and
using \(|a-b|^2\leq2a^2+2b^2\) proves
\eqref{eq:dens-affine-resampling}.  Since
\(\sum_s\|v_s\|^2\leq p\nu^2\),
\[
 2^{-1/2}\rho\leq\lambda\leq p\rho.
\]
Thus \(|\log\lambda-\log\rho|\leq\log(\sqrt2p)\), and
\eqref{eq:dens-affine-Hilbert} follows from
\eqref{eq:dens-affine-centered}.

\end{proof}

\begin{proof}[\proofname\ of Corollary \ref{cor:dens-coefficient-evaluation}]
Before the first group is evaluated, regard the coefficients in all
remaining groups as a vector in their Euclidean coefficient space.  As a
function of the first group this vector is Hilbert-valued affine, and the
quantity \(\rho\) in \eqref{eq:dens-affine-rho} is exactly the Euclidean
norm of the coefficient tensor before that group is evaluated.  After
evaluation, its Hilbert norm is the coefficient norm of the polynomial in
the remaining groups.  This remains exact for a complex group because its
atoms are evaluated as formal complex variables: the quantity \(\rho^2\)
is precisely the sum of the squared norms of the constant coefficient and
the \(p_j\) complex linear coefficients.  The first-moment consequence of
\eqref{eq:dens-affine-Hilbert} bounds the expected change of its logarithm
by \(C_L\log(ep)\).  Repeat this operation group by group and use the
tower property.  The logarithms telescope from
\(\log\|\Coeff P\|_2\) to \(\log|P|\), proving
\eqref{eq:dens-tensor-evaluation}.  Lemma~\ref{lem:dens-affine-log} also
shows at each stage that a nonzero coefficient vector evaluates to zero
with probability zero.
\end{proof}

\begin{proof}[\proofname\ of Lemma \ref{lem:dens-integrated-Hodge}]
Let
\[
 G:=W^{-1/2}(\xi_{ab})_{a,b\leq W}.
\]
The only determinant estimate needed is the Gram--Schmidt identity
\[
 |\det G|=\prod_{a=1}^W
 \dist(G_{a,*},\operatorname{span}\{G_{1,*},\ldots,G_{a-1,*}\}).
\]
Condition on the previous rows and choose a unit vector normal to their
span.  For real atoms, multiply the normal functional by a phase so that
the real part of its coefficient vector has norm at least $2^{-1/2}$.
After the $W^{-1/2}$ normalization, one real coefficient has modulus at
least $(\sqrt2W)^{-1}$; conditioning on the remaining entries bounds the
density of this real projection by $C_LW$.  The same argument works under
the directional alternative after conditioning on the orthogonal
coordinates in the new row.

For complex atoms with planar density, the coefficient vector of the
complex linear form has squared norm $W^{-1}$.  Hence
Lemma~\ref{lem:lsv-planar-linear-form} gives a planar density bounded by
$C_LW$, and therefore a disk small-ball bound $C_LWt^2$.  Since the row
distance dominates the modulus of either scalar projection, all three
cases satisfy the common estimate
\begin{equation}
 \Prob\{\dist(G_{a,*},H)<t\mid H\}
 \leq\min\{1,C_LWt\},
 \label{eq:dens-row-distance}
\end{equation}
and logarithmic integration yields
\begin{equation}
 \E\log_+\frac1{|\det G|}
 \leq C_LW\log(eW).
 \label{eq:dens-det-negative-log}
\end{equation}
More explicitly, if \(d_a\) denotes the successive row distance, then
\[
 \Prob\{\log_+(d_a^{-1})>u\mid H\}
 \leq\min\{1,C_LWe^{-u}\}.
\]
Thus \(\E\log_+^2(d_a^{-1})\leq C_L\log^2(eW)\), and Cauchy--Schwarz
gives
\begin{equation}
 \E\left(\log_+\frac1{|\det G|}\right)^2
 \leq C_LW^2\log^2(eW).
 \label{eq:dens-det-negative-log-L2}
\end{equation}
The constant factor between \(W^{-1/2}(\xi_{ab})\) and the model
normalization \((3W)^{-1/2}(\xi_{ab})\) contributes only \(O(W)\) to
these determinant logarithms.

The complementary-minor expansion in the proof of
Lemma~\ref{lem:dens-row-affinity}, the bound
\(\binom{2W}{r}\leq4^W\), and Hadamard's inequality give the pathwise
estimate
\begin{align}
 \max_r\log_+\|\mathcal A_j^{(r)}\|
 &\leq C W\log(eW)
 \notag\\
 &\quad+C W\log\!\left(
  1+|z|+\|\mathsf A_j\|_{\rm HS}
       +\|\mathsf B_j\|_{\rm HS}
       +\|\mathsf C_j\|_{\rm HS}\right).
 \label{eq:dens-cleared-pathwise-upper}
\end{align}
Indeed, each summand is a product of complementary minors of total degree
at most \(W\), and the number of summands and matrix entries is at most
\(e^{CW\log(eW)}\).  Since
\(\E\|\mathsf A_j\|_{\rm HS}^2=O(W)\), with the same identity for the
other two blocks, Markov integration gives, for the logarithm in the
second line of \eqref{eq:dens-cleared-pathwise-upper}, first and second
moments bounded respectively by \(C_z\log(eW)\) and
\(C_z\log^2(eW)\).  In particular,
\begin{equation}
 \E\max_r\log_+\|\mathcal A_j^{(r)}\|
 \leq C_{L,z}W\log(eW).
 \label{eq:dens-cleared-upper-L1}
\end{equation}

Finally the exact Hodge identity from \eqref{eq:gb-one-step-inverse} is
\begin{equation}
 \|(\mathcal A_j^{(r)})^{-1}\|
 =\frac{\|\mathcal A_j^{(2W-r)}\|}
 {|\det\mathsf B_j\det\mathsf C_j|}.
 \label{eq:dens-exact-Hodge}
\end{equation}
After increasing the constant, define explicitly
\begin{align*}
 \mathfrak h_j
 &:={C_zW\log(eW)}
 +CW\log\!\left(
  1+|z|+\|\mathsf A_j\|_{\rm HS}
       +\|\mathsf B_j\|_{\rm HS}
       +\|\mathsf C_j\|_{\rm HS}\right)\\
 &\quad+\log_+\frac1{|\det\mathsf B_j|}
       +\log_+\frac1{|\det\mathsf C_j|}.
\end{align*}
Equations \eqref{eq:dens-det-negative-log}--
\eqref{eq:dens-cleared-upper-L1} and the exact Hodge identity imply
\(\E\mathfrak h_j\leq C_{L,z}W\log(eW)\) and
\(\E\mathfrak h_j^2\leq C_{L,z}W^2\log^2(eW)\).  This proves
\eqref{eq:dens-one-site-Hodge-L1}.
Submultiplicativity for a product and for its inverse gives
\eqref{eq:dens-interval-Hodge-L1}.  Repeating the same estimates with
squared logarithms also completes the truncation justification in
Lemma~\ref{lem:dens-row-concentration}.
\end{proof}

\begin{proof}[\proofname\ of Proposition \ref{prop:dens-reset-seam}]
For the universal constant $C>0$ in Lemma~\ref{lem:local-all-minor}, define
\[
 V(A):=\det(I_{2W}+A^*A)^{1/2},
 \qquad E:=E_\partial=\operatorname{diag}(\mathsf C_L,\mathsf B_R),
 \qquad a_W:=CW\log(eW).
\]
Also set
\begin{equation}
 \Lambda_\partial:=
 2W\log_+(1+\|E_\partial\|)
 +\log_+\frac1{|\det E_\partial|}.
 \label{eq:dens-endpoint-loss}
\end{equation}
First suppose that \(\Theta\in\mathcal U\) with $\mathcal U$ defined in \eqref{eq:local-U-domain}.  By
\eqref{eq:local-double-elimination} and
Lemma~\ref{lem:local-all-minor},
\begin{equation}
 \mathscr C(\Theta;E_\partial)
 =e^{\pm a_W}
  |\det\Theta_{11}|\,V(ES(\Theta)).
 \label{eq:dens-endpoint-proof-1}
\end{equation}
Define
\[
 \kappa(E):=
 \max_{0\leq r\leq2W}
 \max\bigl\{\|\wedge^rE\|,\|\wedge^rE^{-1}\|\bigr\}.
\]
Functoriality of exterior powers and
\eqref{eq:local-compound-sum} give
\[
 \kappa(E)^{-1}V(S(\Theta))
 \leq V(ES(\Theta))
 \leq \kappa(E)V(S(\Theta)).
\]
Moreover, Hodge duality implies
\[
 \log\kappa(E)
 \leq
 2W\log(1+\|E\|)
 +\log_+\frac1{|\det E|}
 =\Lambda_\partial.
\]
Finally, the graph--Gram identity
\eqref{eq:local-graph-identity} gives
\[
 |\det\Theta_{11}|\,V(S(\Theta))=V(\Theta).
\]
Combining these relations yields
\[
 e^{-a_W-\Lambda_\partial}V(\Theta)
 \leq \mathscr C(\Theta;E_\partial)
 \leq e^{a_W+\Lambda_\partial}V(\Theta)
\]
for every \(\Theta\in\mathcal U\). We now extend this relation to all $\Theta\notin\mathcal U$ via continuity.

For arbitrary \(\Theta\in\GL_{2W}(\C)\), choose a sequence $\varepsilon_n$ such that
\[
 \Theta_n=\Theta+\varepsilon_n
 \begin{pmatrix}I_W&0\\0&0\end{pmatrix}
 \in\mathcal U,
 \qquad \varepsilon_n\longrightarrow0.
\]
Every coefficient of \(\mathscr D_\Theta\) is polynomial in the entries
of \(\Theta\).  Hence
\[
 \mathscr C(\Theta_n;E_\partial)\longrightarrow
 \mathscr C(\Theta;E_\partial),
 \qquad
 V(\Theta_n)\longrightarrow V(\Theta).
\]
Passing to the limit in the preceding multiplicative inequalities and
then taking logarithms proves the following pointwise estimate for every
\(\Theta\in\GL_{2W}(\C)\):
\begin{equation}
 \left|
 \log\mathscr C(\Theta;E_\partial)
 -\frac12\log\det(I_{2W}+\Theta^*\Theta)
 \right|
 \leq C_zW\log(eW)+\Lambda_\partial.
 \label{eq:dens-seam-pathwise}
\end{equation}
Lemma~\ref{lem:dens-integrated-Hodge} implies
\begin{equation}
 \E_{E_\partial}\Lambda_\partial\leq C_LW\log(eW).
 \label{eq:dens-endpoint-loss-L1}
\end{equation}
This completes the proof of \eqref{eq:dens-seam-L1}.
\end{proof}

\begin{proof}[\proofname\ of Proposition \ref{prop:dens-packet-coeff-evaluation2}]
Divide the \(7W^2\) variables into the \(3W\) packet-row groups
\begin{align*}
 \mathbf x_{L,a}
 &:=
 \bigl((A_L)_{ab},(B_L)_{ab}\bigr)_{b=1}^W,\\
 \mathbf x_{C,a}
 &:=
 \bigl((C_C)_{ab},(A_C)_{ab},(B_C)_{ab}\bigr)_{b=1}^W,\\
 \mathbf x_{R,a}
 &:=
 \bigl((C_R)_{ab},(A_R)_{ab}\bigr)_{b=1}^W,
 \qquad 1\leq a\leq W.
\end{align*}
Each group contains at most \(3W\) independent atoms.

By \eqref{eq:local-K-first-elimination},
\[
 P_{\Theta,E_\partial}(\mathbf x)
 =\det K_{\Theta,E_\partial}(\mathbf x).
\]
Every group \(\mathbf x_{\alpha,a}\) occurs only in the \(a\)-th scalar
row of the corresponding packet block row of
\(K_{\Theta,E_\partial}\).  Since the determinant is affine-linear in
each matrix row, \(P_{\Theta,E_\partial}\) is affine separately in every
one of the \(3W\) groups.

The pointwise comparison \eqref{eq:dens-seam-pathwise} shows that
\(\mathscr C(\Theta;E_\partial)>0\), so the polynomial is nonzero.  Applying
Corollary~\ref{cor:dens-coefficient-evaluation} with
\(n_{\rm grp}=3W\) and \(p\leq3W\) gives
\[
 \E_{\Xi_7}\left|
  \log|P_{\Theta,E_\partial}(\Xi_7)|
  -\log\mathscr C(\Theta;E_\partial)
 \right|
 \leq C_L(3W)\log(3eW),
\]
which is \eqref{eq:dens-pathwise-boundary-volume2} after changing the
constant.
\end{proof}

\begin{proof}[\proofname\ of Proposition \ref{propositions6.8}]
It remains to prove \eqref{eq:dens-reset-L1}.  Condition on
\(E_\partial\) and on the deterministic, or outside-measurable,
decomposable unit wedges \(u,v\).  Applying a complex QR decomposition, we can find isometries \(U,V\) such that
\(\widehat U=u\) and \(\widehat V=v\) in the sense of \eqref{eq:local-UhatVhat}, and write
\[
 Z:=Z_{U,V}^{(r)}
   =\langle u,\mathcal Q^{(r)}v\rangle,
 \qquad
 \Gamma:=\Gamma_{U,V}^{(r)}
   =\|\Coeff Z\|_2.
\]

Consider the matrix
\(\Theta_\lambda^{(r;U,V)}\) defined in \eqref{definetheframe}.  By
the deterministic computations in \eqref{eq:local-polynomial-limit}, coefficientwise convergence in the
fixed finite-dimensional polynomial space gives
\[
 \lambda^{-r}
 \mathscr C(\Theta_\lambda^{(r;U,V)};E_\partial)
 \longrightarrow \Gamma,\quad\lambda\to\infty.
\]
On the other hand, the deterministic computation in \eqref{eq:local-graph-limit} gives
\[
 \lambda^{-r}
 \det\!\left(
 I_{2W}
 +(\Theta_\lambda^{(r;U,V)})^*
  \Theta_\lambda^{(r;U,V)}
 \right)^{1/2}
 \longrightarrow1,\quad\lambda\to\infty.
\]
Applying the pointwise estimate \eqref{eq:dens-seam-pathwise} to
\(\Theta_\lambda^{(r;U,V)}\), exponentiating, dividing by
\(\lambda^r\), and using the preceding two deterministic limits gives
\begin{equation}
 e^{-C_zW\log(eW)-\Lambda_\partial}
 \leq\Gamma\leq
 e^{C_zW\log(eW)+\Lambda_\partial}.
 \label{eq:dens-fixed-degree-coeff-bound}
\end{equation}
In particular, \(\Gamma>0\).

After conditioning on \(E_\partial,u,v\), the polynomial \(Z\) is affine
separately in the \(3W\) packet-row groups, each containing at most
\(3W\) independent atoms.  Hence
Corollary~\ref{cor:dens-coefficient-evaluation} gives
\[
 \E_{\Xi_7}\left[
  \left|\log|Z|-\log\Gamma\right|
  \,\middle|\,E_\partial,u,v
 \right]
 \leq C_LW\log(eW).
\]
Since
\[
 \log_+\frac1{|Z|}
 \leq
 \left|\log|Z|-\log\Gamma\right|
 +\log_+\frac1{\Gamma},
\]
\eqref{eq:dens-fixed-degree-coeff-bound} first gives, conditionally on
\(E_\partial,u,v\),
\[
 \E_{\Xi_7}\left[
  \log_+\frac1{|Z|}
  \,\middle|\,E_\partial,u,v
 \right]
 \leq C_{L,z}W\log(eW)+\Lambda_\partial.
\]
Now average over the endpoint blocks.  Since $u,v$ are deterministic or
outside-measurable, they are independent of all nine packet blocks, and
\eqref{eq:dens-endpoint-loss-L1} yields
\begin{align*}
 &\E_{E_\partial}\left[
   \E_{\Xi_7}\left[
    \log_+\frac1{|Z|}
    \,\middle|\,E_\partial,u,v
   \right]
   \,\middle|\,u,v
  \right]\\
 &\qquad=
 \E_{\rm pkt}\left[
  \log_+\frac1{|Z|}
  \,\middle|\,u,v
 \right]
 \leq C_{L,z}W\log(eW).
\end{align*}
This is \eqref{eq:dens-reset-L1}.
\end{proof}

\iffalse

\end{document}
\fi

\appendix

\section{Repair of the full-block least-singular-value input}
\label{app:repaired-jjlo}

\begin{proof}[Proof of Proposition~\ref{prop:jjlo-block-lsv}]
We retain the argument of
\cite[Theorem~2.1]{JainJanaLuhORourke2021} after replacing the
operator-norm input used to construct its event $\mathcal E_K$.  On this
event, every normalized block has bounded operator norm and the two
off-diagonal blocks in each block row have least singular value at least
$b^{-5}$.

For an unnormalized $b\times b$ iid block $A$, the standard subgaussian
matrix-norm tail \cite[Theorem~4.4.5]{vershynin2019high} gives
\begin{equation}
 \Pp\{\|A\|>C_{\xi,K_0}\sqrt b\}\le 2e^{-c_\xi b}.
 \label{eq:repaired-jjlo-block-norm}
\end{equation}
For a normalized off-diagonal block $U=(3b)^{-1/2}A$,
\cite[Proposition~2.2]{JainJanaLuhORourke2021}, with its small parameter
chosen of order $b^{-2}$, gives
\begin{equation}
 \Pp\{s_{\min}(U)<b^{-5}\}
 \le C_{\xi,K_0}b^{-2}+C_{\xi,K_0}e^{-c_\xi b^{1/50}}.
 \label{eq:repaired-jjlo-block-floor}
\end{equation}
The diagonal shift $-zI_b$ changes the normalized block-norm bound only
by the fixed amount $|z|\le K_0$.  Since $m\le b$, a union bound over the
$O(m)$ blocks in \eqref{eq:repaired-jjlo-block-norm}--
\eqref{eq:repaired-jjlo-block-floor} yields
\begin{equation}
 \Pp(\mathcal E_K^c)
 \le C_{\xi,K_0}b^{-1}+C_{\xi,K_0}b e^{-c_\xi b^{1/50}}
 \le C_{\xi,K_0}b^{-1/2}.
 \label{eq:repaired-jjlo-good-event}
\end{equation}
The remainder of the proof of the cited Theorem~2.1 is carried out on
$\mathcal E_K$ and gives the floor $(3b)^{-25m}$ with an additional
$O(b^{-1/2})$ failure probability.  Combining that estimate with
\eqref{eq:repaired-jjlo-good-event} proves the proposition.
\end{proof}

\section{Hilbert--Schmidt-truncated least-singular-value estimate}
\label{app:high-band-lsv}

We prove Theorem~\ref{thm:high-band-lsv}.  We use the following consequence
of the marginal-density theorem of Livshyts, Paouris and Pivovarov.

\begin{lemma}[Projected densities]
\label{lem:lsv-projected-density}
Let $\mathbf x=(x_1,\ldots,x_m)$ have independent real coordinates whose
densities are bounded by $\rho$.  If $E\subset\R^m$ has dimension
$d\in\{1,2\}$, then the density of the orthogonal projection of
$\mathbf x$ onto $E$ is bounded by $C^d\rho^d$.
\end{lemma}
\begin{proof}
After rescaling, this is \cite[Theorem~1.1]{LivshytsPaourisPivovarov2016}.
\end{proof}

\begin{lemma}[Planar density of a complex linear form]
\label{lem:lsv-planar-linear-form}
Let $\eta_1,\ldots,\eta_m$ be independent complex random variables whose
densities with respect to planar Lebesgue measure are bounded by $L$.
For every nonzero $a=(a_1,\ldots,a_m)\in\C^m$, the random variable
$S=\sum_ja_j\eta_j$ has a planar density satisfying
\begin{equation}
 \|f_S\|_\infty\le \frac{C L}{\|a\|_2^2}.
 \label{eq:lsv-planar-linear-density}
\end{equation}
Consequently, uniformly in $w\in\C$ and $s>0$,
\begin{equation}
 \Pp\{|S-w|\le s\}
 \le \min\left\{1,\frac{C Ls^2}{\|a\|_2^2}\right\}.
 \label{eq:lsv-planar-linear-small-ball}
\end{equation}
\end{lemma}
\begin{proof}
For $a_j\ne0$, the planar density of $a_j\eta_j$ is bounded by
$L/|a_j|^2$.  The two-dimensional maximum-density inequality
\cite[Theorem~1]{BobkovChistyakov2014} therefore gives
\[
 \|f_S\|_\infty^{-1}
 \ge c\sum_{j:a_j\ne0}\frac{|a_j|^2}{L},
\]
which is \eqref{eq:lsv-planar-linear-density}.  Integrating this density
over a disk of radius $s$ proves
\eqref{eq:lsv-planar-linear-small-ball}.
\end{proof}

\begin{lemma}[Cyclic block partition]
\label{lem:lsv-block-partition}
There is a cyclic partition
\begin{equation}\label{eq:lsv-block-partition}
 [N]=B_1\sqcup\cdots\sqcup B_J
\end{equation}
into consecutive intervals such that, with $m_j=|B_j|$ and
$d=\min_jm_j$,
\begin{equation}\label{eq:lsv-block-sizes}
 c_*W\le d\le m_j\le d+1\le C_*W,
 \qquad J\le C_*\frac NW.
\end{equation}
Moreover, $a\in B_j$ and
$b\in B_{j-1}\cup B_j\cup B_{j+1}$ imply $|a-b|_N\le W$.
\end{lemma}
\begin{proof}
The row-sum condition and
$\sigma_{ij}^2\le C_{\mathrm{lsv}}/W$ imply
$W\le C_{\mathrm{lsv}}N$.
Denote by $d_0=\lfloor\min(W,N)/8\rfloor$ and divide the cyclically ordered set
$[N]$ as evenly as possible into $J=\lceil N/d_0\rceil$ intervals.
For large $N$ their sizes differ by at most one, lie between $d_0/2$ and
$d_0$, and two equal or adjacent intervals have cyclic diameter at most
$2d_0\le W$. This proves the claim, with constants allowed to depend on
$C_{\mathrm{lsv}}$.
\end{proof}

For $u\in\C^N$, write $u[j]$ for its restriction to $B_j$.  The next lemma
is the only net estimate needed below.  Its two-dimensional part gives the
usual cancellation between complex entropy and small-ball probability;
the one-dimensional part covers vectors that are nearly real after a phase
rotation.

\begin{lemma}[Complex block net]
\label{lem:lsv-block-net}
Let $0<h<s<1$, let $m\ge1$, and let
$\boldsymbol\xi=(\xi_1,\ldots,\xi_m)$ have either independent real
coordinates with densities bounded by $\rho$, or independent complex
coordinates with planar densities bounded by $\rho^2$.  The conclusion
also holds for complex coordinates satisfying alternative~\textnormal{(iii)}
of Assumption~\ref{ass:indicator-density}, with conditional density bounded
by $\rho$.  Suppose
$\rho s\ge h$.  The unit ball
of $\C^m$ is covered, up to distance $Ch$, by at most $(C/h)^3$ families
of finite nets.  Each family has parameters
$(\star,\alpha,x,y)\in\{R,I\}\times[-1,1]\times[0,2]^2$, taken on
$h$-meshes, and its net $\mathcal N(\star,\alpha,x,y)$ satisfies
\begin{equation}\label{eq:lsv-block-net-cardinality}
 |\mathcal N(\star,\alpha,x,y)|
 \le\left(\frac{C\max(x,h)\max(y,h)}{h^2}\right)^m.
\end{equation}
For every $u$ in this net and every $w\in\C$,
\begin{equation}\label{eq:lsv-block-small-ball}
 \Pp\{|\langle\boldsymbol\xi,u\rangle-w|\le s\}
 \le \min\left\{1,
 \frac{C\rho^2s^2}{\max(x,h)\max(y,h)}\right\}.
\end{equation}
\end{lemma}
\begin{proof}
First suppose that the coordinates of $\boldsymbol\xi$ are real.  Write
$u=a+ib$ with $a,b\in\R^m$.  If $\|a\|_2\ge\|b\|_2$, write
\begin{equation}\label{eq:lsv-orthogonal-decomposition}
 b=\alpha a+q,\qquad q\perp a,\qquad |\alpha|\le1.
\end{equation}
Choose $h$-approximations to $\alpha$, $\|a\|_2$ and $\|q\|_2$.
Standard annular volume nets give $p,r\in\R^m$ with $p\perp r$ and
\[
 u'=(1+i\alpha')p+ir,\qquad \|u-u'\|_2\le Ch,
\]
where $\|p\|_2\asymp\max(x,h)$ and
$\|r\|_2\asymp\max(y,h)$ whenever the corresponding parameter is larger
than a fixed multiple of $h$.  Here one first approximates $a$ and then
projects $q$ onto $p^\perp$; the projection changes $q$ by $O(h)$.
The standard lattice-point bound
$|\mathbb Z^m\cap B(0,R)|\le(2+CR/\sqrt m)^m$, after rescaling, gives
\eqref{eq:lsv-block-net-cardinality}.

If $x,y\ge Ch$, Lemma~\ref{lem:lsv-projected-density}, applied to the
plane spanned by $p,r$, bounds the density of
$(\langle\boldsymbol\xi,p\rangle,
\langle\boldsymbol\xi,r\rangle)$ by $C\rho^2/(xy)$.
The map $(s_1,s_2)\mapsto s_1+i(\alpha's_1+s_2)$ has determinant one,
which proves \eqref{eq:lsv-block-small-ball}.  If $y<Ch$, the
one-dimensional version instead gives
$C\rho s/\max(x,h)$; this is bounded by the right-hand side of
\eqref{eq:lsv-block-small-ball} because $\rho s\ge h$.  If also $x<Ch$,
use the trivial bound one.  The case $\|b\|_2>\|a\|_2$ follows by
interchanging real and imaginary parts.  Finally, the choice of $\star$
and the three scalar meshes produces at most $(C/h)^3$ families.

For complex coordinates, the construction and cardinality of the nets
are unchanged.  Lemma~\ref{lem:lsv-planar-linear-form} gives
\[
 \Pp\{|\langle\boldsymbol\xi,u\rangle-w|\le s\}
 \le \min\left\{1,\frac{C\rho^2s^2}{\|u\|_2^2}\right\}.
\]
In the above decomposition,
$\max(x,h)\max(y,h)\le C\|u\|_2^2$ unless both parameters are $O(h)$;
in that remaining case the trivial bound one applies.  This proves
\eqref{eq:lsv-block-small-ball} in the complex case as well.

Under the directional alternative, condition on all orthogonal
coordinates.  Up to a common phase, the remaining random part is a real
linear form with independent coordinates whose conditional densities are
bounded by $\rho$.  The real proof applies conditionally, uniformly in
the resulting center, and averaging proves the claim.
\end{proof}

Set $X_z=X-zI_N$.  For $i\in[N]$, let $P_i$ set the $i$th coordinate to
zero and define
\begin{equation}\label{eq:lsv-deleted-row-matrix}
 M^{(i)}:=P_iX_z^*.
\end{equation}
Thus $\ker M^{(i)}$ is the space of normal vectors to the span of all
columns of $X_z$ except the $i$th.

In this proof, we still use the operator norm to control the approximation error, and we eventually use the simple bound $\|X\|_{\text{op}}\le\|X\|_{\text{HS}}$ to close. Thus in the following Proposition, we assume a generic polynomial upper bound  $\|X_z\|\le N^A$ on the operator norm of $X_z$ without assuming the optimal scale $\|X_z\|=O(1)$.
\begin{proposition}[Block spread of deleted-column normals]
\label{prop:lsv-normal-spread}
Assume the hypotheses of Theorem~\ref{thm:high-band-lsv}, let
$W\ge N^{1/2+\chi}$, and fix $K_z,A,K_0<\infty$ and
$\kappa\in(0,\chi/4)$.  Let $1\le K_N\le K_0N^A$.
Denote the constant by
\begin{equation}\label{eq:lsv-delta}
 \delta:=\exp\left(-N^\kappa\frac NW\right).
\end{equation}
Uniformly for $|z|\le K_z$ and all sufficiently large $N$,
\begin{align}
 \Pp\bigl\{&\|X_z\|_{\op}\le K_N,\ \exists i\in[N],\ \exists k\in[J],
 \ \exists v\in\ker M^{(i)}:\notag\\[-1mm]
 &\hspace{38mm}\|v\|_2=1,\ \|v[k]\|_2<\delta\bigr\}
 \le \exp(-N^{1+\kappa/4}).
 \label{eq:lsv-normal-spread}
\end{align}
\end{proposition}
\begin{proof}
Fix $i,k$ and abbreviate $M=M^{(i)}$.  Let
\begin{equation}\label{eq:lsv-net-mesh}
 h:=\frac{\delta}{C_1(K_N+1)\sqrt J},
\end{equation}
where $C_1$ is sufficiently large.  Apply
Lemma~\ref{lem:lsv-block-net} on every $B_j$.  A global parameter family
has at most $(C/h)^{3J}$ choices; writing
\[
 \zeta_j:=\max(x_j,h)\max(y_j,h),
\]
the corresponding net has size at most
\begin{equation}\label{eq:lsv-global-net-cardinality}
 \prod_{j=1}^J\left(\frac{C\zeta_j}{h^2}\right)^{m_j}.
\end{equation}
Every unit vector $v$ has an approximant $\widehat v$ with
\begin{equation}\label{eq:lsv-global-net-error}
 \|v-\widehat v\|_2\le C\sqrt Jh
 \le\frac{\delta}{4(K_N+1)}.
\end{equation}
Hence $Mv=0$ and $\|M\|_{\op}\le K_N$ imply
$\|M\widehat v\|_2\le\delta/2$.

If $\|v[k]\|_2<\delta$, then $\zeta_k\le C\delta^2$.  Since
$\|v\|_2=1$, some block $B_\ell$ satisfies
$\|v[\ell]\|_2\ge J^{-1/2}$.  In the decomposition
\eqref{eq:lsv-orthogonal-decomposition} on that block, the larger real or
imaginary component has norm at least $(2J)^{-1/2}$, and therefore
\begin{equation}\label{eq:lsv-zeta-ratio}
 \zeta_\ell\ge chJ^{-1/2},
 \qquad \frac{\zeta_k}{\zeta_\ell}\le C(K_N+1)J\delta.
\end{equation}
For large $N$, $k\ne\ell$.  Choose a simple nearest-neighbor cyclic path
$k=k_0,k_1,\ldots,k_s=\ell$ and set
\begin{equation}\label{eq:lsv-path-map}
 \pi(k_a)=k_{a+1}\ (0\le a<s),\qquad
 \pi(j)=j\quad\text{otherwise}.
\end{equation}
The map $\pi$ is intentionally not a permutation: $k$ has no preimage and
$\ell$ has two.  Consequently,
\begin{equation}\label{eq:lsv-path-telescope}
 \prod_{j=1}^J\zeta_{\pi(j)}
 =\left(\prod_{j=1}^J\zeta_j\right)\frac{\zeta_\ell}{\zeta_k}.
\end{equation}

Let $r=d-1$.  From each $B_j$ choose $r$ rows of $M$, avoiding its zero
row $i$.  In a selected row $q\in B_j$, retain only the random coordinates
indexed by $b\in B_{\pi(j)}$ and condition on the others.  The retained
coefficients satisfy $\sigma_{bq}^2\ge c_{\mathrm{lsv}}/W$ by
Lemma~\ref{lem:lsv-block-partition}.  Thus, after scaling, their densities
are bounded by $C_L\sqrt W$ in the real case and their planar densities
are bounded by $C_LW$ in the complex case.  Under the directional
alternative, the scaled conditional densities are bounded by
$C_L\sqrt W$.  Apply Lemma~\ref{lem:lsv-block-net} with
$\rho=C_L\sqrt W$ in the real and directional cases, and with
$\rho=C\sqrt{LW}$ in the planar-density case, always taking $s=\delta$.
It gives
for every selected row $q\in B_j$
\begin{equation}\label{eq:lsv-row-small-ball}
 \Pp\{|(M\widehat v)_q|\le\delta\mid\text{conditioned variables}\}
 \le \frac{C_LW\delta^2}{\zeta_{\pi(j)}}.
\end{equation}
Here $\rho\delta\ge h$ follows from \eqref{eq:lsv-net-mesh}; the
conditional deterministic terms only change the center of the ball.
The selected rows correspond to distinct original columns, so their
retained entry sets are disjoint and independent.  Tensorization yields
\begin{equation}\label{eq:lsv-tensorized-small-ball}
 \Pp\{\|M\widehat v\|_2\le\delta/2\}
 \le\prod_{j=1}^J
 \left(\frac{CW\delta^2}{\zeta_{\pi(j)}}\right)^r.
\end{equation}

Combining \eqref{eq:lsv-global-net-cardinality},
\eqref{eq:lsv-zeta-ratio}, \eqref{eq:lsv-path-telescope}, and
\eqref{eq:lsv-tensorized-small-ball}, and using
$m_j-r\in\{1,2\}$ and $\zeta_j\le C$, bounds the probability for fixed
$i,k,\ell$ by
\begin{equation}\label{eq:lsv-net-union-bound}
 \left(\frac Ch\right)^{3J}C^Nh^{-2N}
 (CW\delta^2)^{rJ}\bigl(C(K_N+1)J\delta\bigr)^r.
\end{equation}
The choice $r=d-1$ is precisely what permits us to avoid the deleted row.

Let
$\Lambda=\log(1/\delta)=N^\kappa N/W$ and
$H=\log(1/h)=\Lambda+O_A(\log N)$.  Since
$N-rJ=\sum_j(m_j-r)\in[J,2J]$, the coefficient of $\Lambda$ in the
logarithm of \eqref{eq:lsv-net-union-bound} is
\[
 3J+2(N-rJ)-r.
\]
All remaining factors cost $O_A(N\log N)$.  Denote the fixed-triple
probability bounded in \eqref{eq:lsv-net-union-bound} by $p_{i,k,\ell}$.
Then
\begin{equation}\label{eq:lsv-entropy-gain}
 \log p_{i,k,\ell}
 \le-cW\Lambda+CJ\Lambda+O_A(N\log N)
 \le-c'N^{1+\kappa}.
\end{equation}
Indeed, $W\Lambda=N^{1+\kappa}$ and
$J\Lambda\le CN^{2+\kappa}/W^2
\le CN^{1+\kappa-2\chi}$.  Summing over $i,k,\ell$ preserves an
$\exp(-c''N^{1+\kappa})$ bound, which is at most the right-hand side of
\eqref{eq:lsv-normal-spread} for large $N$.
\end{proof}

\begin{proof}[\proofname\ of Theorem~\ref{thm:high-band-lsv}]
Denote by
\[
 \tau:=t\exp\left(-N^{3\kappa}\frac NW\right),
 \qquad
 \delta:=\exp\left(-N^\kappa\frac NW\right).
\]
For the $i$th column $[X_z]^i$, let $H_i$ be the span of the other
columns.  For every $v\in\C^N$,
\begin{equation}\label{eq:lsv-distance-to-span}
 \|X_zv\|_2\ge |v_i|\operatorname{dist}([X_z]^i,H_i).
\end{equation}
Let $\mathcal G_i$ be the event that every unit vector normal to $H_i$ has
norm at least $\delta$ on every block, and write
$\mathcal G=\bigcap_i\mathcal G_i$ and
$\mathcal F_R=\{\|X\|_{\HS}\le R\sqrt N\}$.  Each $\mathcal G_i$ is measurable
with respect to the columns other than the $i$th.  On $\mathcal F_R$ and
$|z|\le K_z$,
$\|X_z\|_{\op}\le\|X\|_{\HS}+|z|
\le(R+K_z+1)\sqrt N$ for all large $N$.
Proposition~\ref{prop:lsv-normal-spread}, with
$K_N=(R+K_z+1)\sqrt N$, gives
\begin{equation}\label{eq:lsv-normal-good-event}
 \Pp\{\mathcal F_R,\ \mathcal G^c\}
 \le\exp(-N^{1+\kappa/4}).
\end{equation}

If $s_{\min}(X_z)\le\tau$, choose a unit $v$ with
$\|X_zv\|_2\le\tau$ and an $i$ with $|v_i|\ge N^{-1/2}$.
Then \eqref{eq:lsv-distance-to-span} gives
\begin{equation}\label{eq:lsv-small-distance}
 \operatorname{dist}([X_z]^i,H_i)\le\tau\sqrt N.
\end{equation}
Condition on the columns other than the $i$th and choose a unit normal
$n_i\in H_i^\perp$ by a fixed measurable rule.  If $B_{q(i)}$ contains
$i$, then on $\mathcal G_i$,
$\|n_i[q(i)]\|_2\ge\delta$.  In the real case, the random entries of
$[X_z]^i[q(i)]$ have scaled densities bounded by $C_L\sqrt W$.
Lemma~\ref{lem:lsv-projected-density}, after a phase rotation, therefore
gives for every $s>0$
\begin{equation}\label{eq:lsv-column-small-ball}
 \Pp\{ |\langle[X_z]^i,n_i\rangle|\le s
 \mid [X_z]^j,\ j\ne i\}
 \le \frac{C_L\sqrt W}{\delta}s.
\end{equation}
For complex entries, the coefficient vector of the random linear form on
$B_{q(i)}$ has squared norm at least $c\delta^2/W$.
Lemma~\ref{lem:lsv-planar-linear-form} instead gives the stronger bound
\[
 \Pp\{ |\langle[X_z]^i,n_i\rangle|\le s
 \mid [X_z]^j,\ j\ne i\}
 \le \min\left\{1,\frac{C_LWs^2}{\delta^2}\right\}
 \le \frac{C_L\sqrt W}{\delta}s,
\]
so \eqref{eq:lsv-column-small-ball} holds in both cases.
Under the directional alternative, condition on the orthogonal
coordinates in the $i$th column and apply the real projected-density
bound; this again gives \eqref{eq:lsv-column-small-ball}.

Distance at most $s$ implies
$|\langle[X_z]^i,n_i\rangle|\le s$.  Thus a union bound over $i$, applied
with the columnwise measurable events $\mathcal G_i$, yields
\begin{align*}
 &\Pp\{s_{\min}(X_z)\le\tau,\ \mathcal F_R,\ \mathcal G\}\\
 &\qquad\le C N^{3/2}\sqrt W\,\frac{\tau}{\delta}
 \le Ct.
\end{align*}
The last inequality holds uniformly for $|z|\le K_z$ and large $N$, because
$(N^{3\kappa}-N^\kappa)N/W$ dominates
$\log(N^{3/2}\sqrt W)$.  Adding
\eqref{eq:lsv-normal-good-event} proves the theorem.  Thus the net argument
is truncated only by a Hilbert--Schmidt event; in the iid applications the
second moment removes that cutoff through
\eqref{eq:lsv-remove-HS-cutoff}, and no operator-norm moment estimate is
used.
\end{proof}

\printbibliography

\begin{thebibliography}{99}

\bibitem{GuEisenstat}
M.~Gu and S.~C. Eisenstat,
\emph{Efficient algorithms for computing a strong rank-revealing QR
factorization},
SIAM J. Sci. Comput. \textbf{17} (1996), 848--869.

\bibitem{Cook}
N.~A. Cook,
\emph{Lower bounds for the smallest singular value of structured random
matrices},
Ann. Probab. \textbf{46} (2018), 3442--3500;
\href{https://arxiv.org/abs/1608.07347}{arXiv:1608.07347}.

\bibitem{Nguyen}
H.~H. Nguyen,
\emph{Random matrices: overcrowding estimates for the spectrum},
J. Funct. Anal. \textbf{275} (2018), 2197--2224;
\href{https://arxiv.org/abs/1709.06682}{arXiv:1709.06682}.

\bibitem{RV}
M.~Rudelson and R.~Vershynin,
\emph{The Littlewood--Offord problem and invertibility of random matrices},
Adv. Math. \textbf{218} (2008), 600--633;
\href{https://arxiv.org/abs/math/0703503}{arXiv:math/0703503}.

\bibitem{RVrect}
M.~Rudelson and R.~Vershynin,
\emph{The smallest singular value of a random rectangular matrix},
Comm. Pure Appl. Math. \textbf{62} (2009), 1707--1739;
\href{https://arxiv.org/abs/0802.3956}{arXiv:0802.3956}.

\bibitem{TaoVuReplacement}
T.~Tao and V.~Vu,
\emph{Random matrices: universality of ESDs and the circular law},
Ann. Probab. \textbf{38} (2010), 2023--2065.

\bibitem{vershynin2019high}Vershynin, R. High-dimensional probability. {\em Cambridge Series In Statistical And Probabilistic Mathematics}. \textbf{47} (2019)



\end{thebibliography}

\end{document}